\documentclass{amsart}

\usepackage{enumitem}
\usepackage{amsmath}
\usepackage{amssymb}
\usepackage{amsthm}
\usepackage{stmaryrd}
\usepackage{hyperref}
\usepackage{tikz-cd}
\usepackage[english]{babel}
\usetikzlibrary{babel}

\usepackage{amsaddr}

\newtheorem{theorem}{Theorem}[section]
\newtheorem{lemma}[theorem]{Lemma}
\newtheorem{corollary}[theorem]{Corollary}
\newtheorem{proposition}[theorem]{Proposition}

\theoremstyle{definition}
\newtheorem{definition}[theorem]{Definition}
\theoremstyle{remark}
\newtheorem{example}[theorem]{Example}
\theoremstyle{remark}
\newtheorem{remark}[theorem]{Remark}

\newcommand{\Set}{\ensuremath{\mathsf{Set}}}
\newcommand{\s}{\ensuremath{\mathsf{s}}}
\newcommand{\TOP}{\ensuremath{\mathsf{TOP}}}

\newcommand{\Sh}{\mathsf{Sh}}

\newcommand{\C}{\mathbb{C}}
\newcommand{\R}{\mathbb{R}}
\newcommand{\Z}{\mathbb{Z}}

\newcommand{\A}{\mathcal{A}}
\newcommand{\E}{\mathcal{E}}
\newcommand{\F}{\mathcal{F}}
\newcommand{\T}{\mathcal{T}}
\newcommand{\U}{\mathcal{U}}

\newcommand{\op}{\mathrm{op}}
\newcommand{\Ocal}{\mathcal{O}}

\newcommand{\Q}{\mathbb{Q}}

\newcommand{\PSh}{\mathsf{PSh}}
\newcommand{\SLC}{\mathsf{SLC}}
\newcommand{\Cont}{\mathsf{Cont}}
\newcommand{\HH}{\mathrm{H}}

\newcommand{\Hom}{\mathrm{Hom}}
\newcommand{\HOM}{\mathrm{HOM}}

\newcommand{\Spec}{\mathrm{Spec}}
\newcommand{\etale}{\mathrm{\acute{e}t}}
\newcommand{\sh}{\mathrm{sh}}

\newcommand{\cosk}{\mathrm{cosk}}

\begin{document}

\title{The smallest $n$-pure subtopos and dimension theory}
\author{Jens Hemelaer}
\email{hemelaerjens@gmail.com}

\begin{abstract}
We introduce the notion of $n$-pure geometric morphism between Groth\-en\-dieck toposes, over a Grothendieck base topos $\T$. This is a higher-dimensional analogue of the concepts of dense and pure geometric morphism. We extend the construction of the smallest dense subtopos and smallest pure subtopos by constructing a smallest $n$-pure subtopos, for each natural number $n$. Based on this, we then propose a concept of dimension for a Grothendieck topos, in this way also arriving naturally at a distinction between toposes with boundary and toposes without boundary. We show that the zero-dimensional toposes without boundary are precisely the Boolean toposes, and that the topos associated to an $n$-manifold is again $n$-dimensional (with boundary if the manifold has a boundary). Some other toposes for which we calculate the dimension are the topos associated to the rational line and the toposes associated to a right Ore monoid or free monoid. Finally, we move to algebraic geometry: for a scheme $X$ of characteristic $0$ and Krull dimension $d$, we prove that the dimension of the associated petit \'etale topos is $2d$, assuming that $X$ is excellent and regular, or that $X$ is variety. As a first example in mixed characteristic, we show that the petit \'etale topos associated to $\Spec(\Z)$ is two-dimensional.
\end{abstract}

\maketitle

\tableofcontents

\section{Introduction}

\subsection{Motivation}

The notion of dimension has been studied extensively for topological spaces \cite{engelking} \cite{nagata-dimension} and more recently in the setting of frames \cite{charalambous} \cite{isbell-dimension} \cite{banaschewski-gilmour} \cite{brijlall-baboolal}. In this paper, we would like to extend this concept to (Grothendieck) toposes.

We will propose a definition of dimension for toposes that is closely related to the construction of the smallest dense subtopos \cite{blass-scedrov} and the smallest pure subtopos (also called Lebesgue subtopos) \cite[Theorem 2.5]{funk-branched} \cite[Proposition 9.2.10]{SCT}. For each integer $n\geq -1$, we will give a definition of \emph{$n$-pure geometric morphism} and then show that each Grothendieck topos has a smallest $n$-pure subtopos $\E_{\leq n+1}$. The $(-1)$-pure geometric morphisms are precisely the dense geometric morphisms, and the $0$-pure geometric morphisms agree with the pure geometric morphisms, under the assumption that the codomain is locally connected. In general, we get a chain of subtoposes
\begin{equation*}
\E_{\leq 0} \subseteq \E_{\leq 1} \subseteq \dots \subseteq \E_{\leq n} \subseteq \dots
\end{equation*}
and we will say that $\E$ is \emph{$n$-dimensional without boundary}, precisely if $\E_{\leq n} = \E$ and $n$ is the smallest integer with this property. In this terminology, the $0$-dimensional toposes without boundary are then the Boolean toposes. Following \cite[Example 9.2.11]{SCT}, the topos of sheaves on the real line is $1$-dimensional without boundary but the topos of sheaves on the real plane is not.

The topological intuition for $n$-pure geometric morphisms is that base change of cohomology classes along the geometric morphism works bijectively up to degree $n$ and injectively in degree $n+1$. For example, distinct $0$-cocycles remain distinct after base change along a dense morphism. Very loosely, the $n$-dimensional toposes without boundary will have the property that puncturing them creates new $(n-1)$-cycles. For example, puncturing the real plane results in new $1$-cocycles corresponding to nontrivial covering spaces. For the connection between pure geometric morphisms and branched covering spaces, see \cite{funk-branched} \cite{SCT}.

If we apply this dimension theory to toposes of the form $\Sh(X)$ for $X$ a topological space, then the hope is that the dimension of $\Sh(X)$ agrees with that of $X$. We show that, to some extent, this is the case. If $X$ is an $n$-dimensional manifold without boundary, then we demonstrate that $\Sh(X)$ is $n$-dimensional without boundary as well. 
This is related to the typical proof that $\R^n$ and $\R^m$ are not homeomorphic if $n\neq m$: puncturing $\R^n$ creates a new cohomology class in degree $n-1$.

To account for manifolds with boundary, we will introduce the notion of toposes with boundary. We define the \emph{content} $\E_{< \infty}$ of a topos $\E$ to be the union of the subtoposes $\E_{\leq n}$ for $n\geq -1$, 
and we say that $\E$ is a topos with boundary if $\E_{< \infty} \neq \E$.
In this case, we say that $\E$ is $n$-dimensional if $n$ is the smallest integer with $\E_{\leq n} = \E_{<\infty}$, or $\infty$-dimensional if such an integer does not exist.
With this definition we prove that for $X$ an $n$-dimensional manifold with boundary, the topos $\Sh(X)$ is also $n$-dimensional with boundary.

It will not be true in general that the dimension of $\Sh(X)$ agrees with that of the topological space $X$. For example, we will show that $\Sh(\Q)$ is $1$-dimensional without boundary, for $\Q$ the rational number line. In topology however, $\Q$ is thought of as zero-dimensional: the small and large inductive dimension and the covering dimension are all zero in this case.

Dimension is also a fundamental concept in algebraic geometry. One of the motivations for the development of scheme theory was the analogy between number fields and function fields of curves over finite fields. In both cases we can use some of the methods that were originally developed for complex curves (genus, ramification, Riemann--Roch...). This means that we consider number fields and function fields to be $2$-dimensional from a topological point of view.

In the context of this article, we have the advantage of being able to explicitly calculate a (topological) dimension for these schemes, by looking at the dimension of the associated petit \'etale topos. It turns out that the petit \'etale toposes of complex curves and of $\Spec(\Z)$ are both $2$-dimensional. We did not manage to calculate the topos-theoretic dimension in the case of curves over finite fields, but it would be interesting to see if in that case the dimension is $2$ as well (the main difficulty is calculating cohomology with $p$-group coefficients for $p$ the characteristic, this calculation is usually avoided).

There is also a Riemann Hypothesis for function fields, proven by Weil in the case of curves and by Deligne in general. An intriguing question is now whether the original Riemann Hypothesis (which would correspond to the field $\Q$) can be solved using similar geometric methods as in the function field case, applied to a different topos. This is an avenue that was explored in \cite{connes-essay}, inspired by the ideas of Soul\'e \cite{soule} and Manin \cite{manin}. Intuitively, the right topos would then (aside from other properties) have topological dimension $2$. The theory of dimension we propose in this article could help point in the right direction.

Another place where we would like to apply dimension theory of toposes, is in the theory of monoids. In \cite{hr1}, a dictionary was built between algebraic properties of a monoid $M$ and geometric properties of its associated topos of right $M$-sets. In the same spirit, it is natural to calculate the dimension of these associated toposes, for various families of monoids. We start in this article with two classes of monoids in particular: right Ore monoids and free monoids. For a right Ore monoid $M$ we find that the associated topos is $0$-dimensional, and that it has a boundary precisely if $M$ is not a group. For a free monoid $M$ on at least two generators, we calculate that the associated topos is $1$-dimensional without boundary (the free monoid on one generator is right Ore so it already falls in the previous category).

The proposed definition of dimension depends on the choice of base topos. For a geometric morphism $\rho: \F \to \E$, we can take $\E$ as base topos and then the calculated dimension for $\F$ can be seen as the relative dimension of $\F$ over $\E$. As an example, we show that a projection map $\R^2 \to \R$ turns $\Sh(\R^2)$ into a one-dimensional topos over $\Sh(\R)$. So in this case, the relative dimension agrees with geometric intuition.

\subsection{Outline of the article}

A considerable amount of background is needed to state the definitions and to prove that every Grothendieck topos has a smallest $n$-pure subtopos. Results from higher category theory and higher topos theory are needed, and we do not assume previous familiarity with these topics. Because we are not fluent in the language of higher category theory, we will also not expect this from the reader. Results on higher category theory are taken from the literature, and if needed they are translated to more classical language.

We avoid the more challenging aspects of higher category theory by working with simplicial toposes equipped with a fixed model structure (the injective model structure). In this sense, we stay close to the original approach of Joyal \cite{joyal-letter-grothendieck}, and to that of Jardine \cite{jardine-book} \cite{jardine-lectures} and Rezk \cite{rezk}.

In Section 2, we first show what we can do without using these homotopical methods, in dimension $0$ and $1$. For a class of objects $\A$ in a Grothendieck topos $\E$, we will define $\A$-dense and $\A$-pure geometric morphisms, generalizing dense and pure geometric morphisms. We are mainly interested in the case where $\A$ is the class of objects that are constant over a base topos $\T$. In this case, the $\A$-dense geometric morphisms coincide with the dense geometric morphisms over $\T$, while the $\A$-pure geometric morphisms coincide with what we will call the $0$-pure geometric morphisms over $\T$. The $0$-pure geometric morphisms were already studied in the literature, as those geometric morphisms $\rho$ for which $\rho$ preserves $\T$-indexed coproducts \cite[p.~14]{BF}. We construct, for each choice of $\A$, the smallest $\A$-dense subtopos and the smallest $\A$-pure subtopos. In this way, we recover the construction of the smallest dense subtopos and the smallest pure subtopos \cite[Theorem 2.5]{funk-branched}, while also tackling the case $n=0$ for the construction of the smallest $n$-pure subtopos later on.

In Section 3, we start building the necessary theory in order to extend this proof to higher dimensions. We will give an overview of the homotopy theory of toposes, following the approach of Joyal \cite{joyal-letter-grothendieck}, Jardine \cite{jardine-book} \cite{jardine-lectures} and Rezk \cite{rezk}. Starting from a Grothendieck topos $\E$, we construct the topos $\s\E$ of simplicial objects in $\E$, and we equip $\s\E$ with the injective model structure as introduced by Joyal. We give some background on model categories and Quillen adjunctions, with a focus on simplicial toposes and the geometric morphisms between them.

In Section 4, we give an overview of the theory of localization and descent. This is the homotopical analogue to the classical story regarding Grothendieck topologies and subtoposes. Since the goal is to construct a smallest $n$-pure subtopos, these ideas play a central role in this article. In particular, we have to take into account the issue of hypercompleteness, a common topic in higher topos theory.

In Section 5, we discuss sheaf cohomology, using the language that was introduced in the previous two sections. This notion of sheaf cohomology agrees with the more classical notion via derived categories and injective resolution. There are two main reasons for that we use model categories and simplicial objects instead. The first reason is that, in the definition of $n$-pure geometric morphisms, we also need cohomology with nonabelian coefficients. The second reason is that the language of model categories and simplicial objects is close to, and sometimes agrees with, the language used by experts in higher topos theory, so we can more easily build on existing results the literature.

In Section 3, 4 and 5 the point is to discuss existing results in the literature, as background and to set the notations. For that reason, it might be helpful to skip these three sections during a first reading. We do not claim to introduce any new results in these sections.

In Section 6, we arrive at our definition of $n$-pure geometric morphisms. We start with generalizing the notion of $\A$-pure geometric morphisms from Section 2. Whereas $\A$ was a family of objects in Section 2, now it is a family of simplicial objects, more specifically Eilenberg--Mac Lane objects. As a special case, $n$-pure geometric morphisms are introduced. Roughly, the $n$-pure geometric morphisms are the ones that induce bijections on cohomology with coefficients in the base topos, for degrees up to $n$. So the base topos plays an important role for $n$-pure geometric morphisms. This is the main difference between $n$-pure geometric morphisms and $n$-connected geometric morphisms. The latter is a concept that already exists in the literature, and for $n$-connectedness the idea is that the cohomology coefficients can be arbitrary objects in the codomain, they do not have to be constant over the base topos. After the definition, we discuss some equivalent formulations for $n$-pure geometric morphisms and some first properties, and we look at the relation to Galois theory and fundamental groups.

In Section 7, we show the existence of the smallest $n$-pure subtopos and then introduce the concept of dimension of a Grothendieck topos in terms of these smallest $n$-pure subtoposes, for varying $n$. Because the definition of $n$-pure geometric morphism depends on the base topos, the definition of dimension does as well. The base topos is again assumed to be a Grothendieck topos, because the theory of simplicial objects that we rely on is worked out to a higher degree for Grothendieck toposes than it is for elementary toposes. Some basic properties of dimensions are discussed, most notably the fact that dimension can be computed \'etale-locally.

In Section 8, we discuss methods for calculating the dimension of a topos, and we apply it in various cases. We show that Boolean toposes are zero-dimensional, and calculate the dimension of toposes associated to locally Euclidean topological spaces, manifolds with boundary, the rational line, boolean toposes, irreducible topological spaces, right Ore monoids and free monoids. In all these examples, we work over $\Set$ as base topos. We also show that totally connected toposes over an arbitrary base topos are zero-dimensional, and we calculate the relative dimension associated to a projection $\R^2 \to \R$.

In Section 9, we continue calculating dimensions of toposes, but now the focus shifts to algebraic geometry, more specifically to petit \'etale toposes associated to schemes. The main result here is that for an excellent, regular scheme of characteristic $0$ and Krull dimension $d$, the dimension of the associated petit \'etale topos is $2d$. In other words, the topological dimension is twice the algebraic dimension, something we are used to from working with complex varieties. The calculations in positive characteristic and mixed characteristic are more difficult and outside the scope of this article. However, as one example in mixed characteristic, we do show that the petit \'etale topos associated to $\Spec(\Z)$ is two-dimensional.

Some of the most technical parts are moved to a Appendix A, B and C. Similarly to Sections 3, 4 and 5, these discuss ideas from the existing literature, translated to the language and notations of this article where needed.

Throughout the article, the word ``topos'' will often mean ``Grothendieck topos'', even if it is not made explicit. We rely on results in the literature that only apply to Grothendieck toposes, and it seems difficult to generalize further towards elementary toposes. For example, it is unclear to what extent the category of simplicial objects in an elementary topos can be equipped with a simplicial model structure, and to what extent the results regarding Eilenberg--Mac Lane objects, localization and hypercompletion still hold.

\subsection{Acknowledgements}

I first would like to thank Jonathon Funk for his important contributions. The current article grew out of joint discussions and notes he shared on dimension theory (starting in December 2021). He explained to me the topological link between pure geometric morphisms and one-dimensionality, and suggested to use orthogonality to extend this approach to higher dimensions. It is this original idea that was later worked out using the necessary homotopical methods.

I would also like to thank Morgan Rogers, Thomas Streicher, Mat\'ias Menni, Lieven Le Bruyn and Mark Sioen for interesting discussions regarding this topic. It is through the joint work with Morgan Rogers that I came in contact with the rich theory surrounding pure geometric morphisms, as developed by Jonathon Funk and Marta Bunge.
\section{Low dimension}

\subsection{Topologies defined by an object in the topos}
\label{ssec:topologies-defined-by-an-object}

Let $\E$ be a topos and let $$j: \Omega_\E \to \Omega_\E$$ be a Lawvere--Tierney topology on $\E$. Following Par\'e \cite{P1}, we can equivalently specify the Lawvere--Tierney topology via a subobject $J \subseteq \Omega_\E$, using the universal property of $\Omega_\E$. For each object $U$ in $\E$, we know that $\Omega_\E(U)$ is the set of subobjects of $U$, and by definition we say that $$J(U) \subseteq \Omega_\E(U)$$ is the subset of \emph{$J$-dense monomorphisms} $V \hookrightarrow U$ with target $U$. We use here the shorthand $F(U)$ when we mean $\Hom_\E(U,F)$, to get back the kind of notations that we are used to when working with sheaves on a site. The family of $J$-dense monomorphisms behaves similarly to a Grothendieck topology, with the distinction that we are now working in a site-independent way. In the following, by \emph{topology on $\E$} we mean a Lawvere--Tierney topology on $\E$, with the above point of view in mind. We write $\E_J \subseteq \E$ for the subtopos corresponding to a topology $J$ on $\E$. An object $X$ of $\E$ is called a \emph{$J$-sheaf} if it is the pushforward of an object in $\E_J$, or equivalently, if $X(U) \to X(V)$ is a bijection for every $J$-dense monomorphism $V \hookrightarrow U$. We say that $X$ is a \emph{separated presheaf} for $J$ if $X(U) \to X(V)$ is injective for every $J$-dense monomorphism $V \hookrightarrow U$.

\begin{proposition} \label{prop:characterization-topologies}
Let $\E$ be a Grothendieck topos. Let $\Xi$ be a class of monomorphisms in $\E$. Then $\Xi$ is the class of $J$-dense monomorphisms for a topology $J$ on $\E$ if and only if the following three conditions are satisfied:
\begin{enumerate}
\item $\Xi$ is local, i.e.\ for a jointly epimorphic family $\{U_i \to U\}_{i \in I}$ we have that $V \hookrightarrow U$ is in $\Xi$ if and only if each
\begin{equation*}
U_i \times_U V \hookrightarrow U_i
\end{equation*}
is in $\Xi$;
\item $\Xi$ defines a subcategory, i.e. it contains the identity maps and is closed under composition;
\item $\Xi$ is upwards closed, i.e.\ for each $W \hookrightarrow V \hookrightarrow U$, if the composition $W \hookrightarrow U$ is in $\Xi$ then $V \hookrightarrow U$ is in $\Xi$ as well.
\end{enumerate}
\end{proposition}
\begin{proof}
We will only show the ``if'' direction because the ``only if'' direction is more straightforward.

Starting from $\Xi$, we can construct a presheaf $J$ on $\E$, with $J(U)$ consisting of isomorphism equivalence classes of monomorphisms in $\Xi$ with fixed codomain $U$. Presheaf restriction corresponds to pullback (the stated properties ensure that the monomorphisms in $\Xi$ are stable under pullback). We can in fact view $J$ as an object of $\E$, because $(1)$ states that it is a sheaf for the canonical topology on $\E$. More precisely, it is a subobject $J \subseteq \Omega_\E$.

To show that $J$ defines a topology, we now use \cite[Proposition 3.18]{TT}, which states that such a $J \subseteq \Omega_\E$ defines a topology if and only if $\Xi$ contains all isomorphisms and moreover $\psi\phi$ is in $\Xi$ if and only if both $\psi$ and $\phi$ are in $\Xi$.

The most difficult step is showing that if $\psi\phi$ is in $\Xi$ then $\phi$ is in $\Xi$. Here we use that $\Xi$ is stable under pullback, for the particular pullback square of the form
\begin{equation} \label{eq:particular-pullback}
\begin{tikzcd}
W \ar[r,"{\phi}"] \ar[d,equals] & V \ar[d,"{\psi}"] \\
W \ar[r,"{\psi\phi}"'] & U
\end{tikzcd}\,\,.
\end{equation}
\end{proof}

\begin{remark}
For Proposition \ref{prop:characterization-topologies} the main argument is the use of the pullback square \eqref{eq:particular-pullback}, this is an argument due to Jonathon Funk.
\end{remark}

\begin{corollary} \label{cor:largest-topology-X-sheaf-separated-presheaf}
For an object $X$ in $\E$, let $\Xi$ be the class of monomorphisms $\phi: V \hookrightarrow U$ with the property that for every pullback
\begin{equation} \label{eq:pullback-2}
\begin{tikzcd}
V' \ar[r,"{}"] \ar[d] & U' \ar[d] \\
V \ar[r,"{\phi}"] & U
\end{tikzcd}
\end{equation}
in $\E$ the induced map $\Hom(V',X) \to \Hom(U',X)$ is bijective (resp. injective). Then $\Xi$ is the class of $J$-dense monomorphisms with $J$ the largest topology  for which $X$ is a sheaf (resp. separated presheaf).
\end{corollary}
\begin{proof}
To show that $\Xi$ defines a topology, we use Proposition \ref{prop:characterization-topologies}. Condition $(2)$ follows from $\Hom(-,X)$ being a presheaf. For condition $(1)$ we use that $\Hom(-,X)$ is a sheaf for the canonical topology: we can then write $\Hom(V,X) \to \Hom(U,X)$ as a limit of the maps $\Hom(V_i,X) \to \Hom(U_i,X)$, and use that taking the limit preserves isomorphisms (resp. monomorphisms). For condition $(3)$, assume that the composition $W \hookrightarrow V \hookrightarrow U$ is in $\Xi$. We claim that then $V \hookrightarrow U$ is in $\Xi$ as well. Now consider the diagram
\begin{equation*}
\begin{tikzcd}
W' \ar[r] \ar[d] & V' \ar[r,"{}"] \ar[d] & U' \ar[d] \\
W \ar[r] & V \ar[r,"{\phi}"] & U
\end{tikzcd}
\end{equation*}
with the right square an arbitrary pullback diagram as in \eqref{eq:pullback-2}, and the left square the corresponding pullback diagram. We now find that $\Hom(U',X) \to \Hom(V',X)$ is an isomorphism (resp.\ monomorphism) because both $\Hom(U',X) \to \Hom(W',X)$ and $\Hom(V',X) \to \Hom(W',X)$ are. We conclude that $\Xi$ is the class of $J$-dense monomorphisms for a topology $J$ on $\E$.

By definition, $X$ is a sheaf (resp.\ separated presheaf) for $J$. Conversely, let $J'$ be a topology such that $X$ is a sheaf (resp.\ separated presheaf). Because the $J'$-dense monomorphisms are stable under pullback, we necessarily have $J' \subseteq J$. So $J$ is the largest topology for which $X$ is a sheaf (resp.\ separated presheaf).
\end{proof}

\begin{remark}
Corollary \ref{cor:largest-topology-X-sheaf-separated-presheaf} gives a characterization of the largest topology for which a given object $X$ is a sheaf or separated presheaf. This is a classical result going back to Par\'e \cite{P1}, see e.g.\ \cite[Examples 22.(1)]{P1} for the description of the largest topology for which $X$ is a sheaf. Unfortunately, page 396 of \cite{P1} seems to be missing at this time (at least in digital versions of the article).
\end{remark}

\subsection{Alternative description}
\label{ssec:alternative-description}

We write $J_X$ (resp.\ $J_X'$) for the largest topology for which $X$ is a sheaf (resp.\ separated presheaf).

Consider $\phi: V \hookrightarrow U$ in $\E$ as above. We now take a more geometric point of view, and look at the induced geometric morphism
\begin{equation*}
\begin{tikzcd}
\E/V \ar[rd,"{\pi_V}"'] \ar[rr,"{\xi}"] & & \E/U \ar[ld,"{\pi_U}"]\\
& \E &
\end{tikzcd}
\end{equation*}
over $\E$. We can then describe the results of Corollary \ref{cor:largest-topology-X-sheaf-separated-presheaf} as follows. For an object $X$ in $\E$, the monomorphism $\phi$ is $J_X$-dense (resp.\ $J_X'$-dense) if and only if the map
\begin{equation} \label{eq:restriction}
(\pi_U^*X)(U') \to (\pi_V^*X)(V')
\end{equation}
is bijective (resp.\ injective), for any pullback diagram \eqref{eq:pullback-2}. We can rewrite the map \eqref{eq:restriction} first as
\begin{equation*}
(\pi_U^*X)(U') \to (\xi^*\pi_U^*X)(\xi^*U')
\end{equation*}
and then as
\begin{equation*}
(\pi_U^*X)(U') \to (\xi_*\xi^*\pi_U^*X)(U').
\end{equation*}

We conclude that $\phi$ is $J_X$-dense (resp.\ $J_X'$-dense) precisely if
\begin{equation} \label{eq:xi-sheafification}
\pi_U^*X \longrightarrow \xi_*\xi^*\pi_U^*X
\end{equation}
is an isomorphism (resp.\ monomorphism). So the following two statements are equivalent:
\begin{enumerate}
\item $\phi: V \hookrightarrow U$ is dense for the largest topology on $\E$ for which $X$ is a sheaf (resp.\ separated presheaf);
\item the object $\pi_U^*X$ in $\E/U$ is a sheaf (resp.\ separated presheaf) for the topology defined by the subtopos $\E/V \subseteq \E/U$.
\end{enumerate}
This captures the essence of why the smallest dense subtopos and smallest pure subtopos exist.

\begin{remark}
Let $\A = \{X_i\}_{i \in I}$ be a family of objects in $\E$. Then the largest topology for which each $X_i$ is a sheaf (resp.\ separated presheaf) is given by $J_\A = \bigcap_{i \in I} J_{X_i}$ (resp.\ $J_\A' = \bigcap_{i \in I} J_{X_i}'$), and the arguments above still go through.
\end{remark}

\subsection{\texorpdfstring{$\mathcal{A}$}{A}-dense and \texorpdfstring{$\mathcal{A}$}{A}-pure geometric morphisms}
\label{ssec:A-dense-A-pure}

Let $\rho: \F \to \E$ be a geometric morphism. For a family of objects $\A$ in $\E$, we will say that $\rho$ is \emph{$\A$-pure} if for each $X$ in $\A$ the map
\begin{equation} \label{eq:generalized-sheafification}
X \longrightarrow \rho_*\rho^*X
\end{equation}
is an isomorphism. Similarly, we say that $\rho$ is \emph{$\A$-dense} if for each $X$ in $\A$ the map \eqref{eq:generalized-sheafification} is a monomorphism.

We recover the following notions.
\begin{example} We fix a base topos $\T$ and structure maps $e: \E \to \T$, $f: \F \to \T$ such that $e\rho \cong f$.
\begin{enumerate}
\item If $\A$ is the family of all objects $e^*T$ for $T$ in $\T$, then the $\A$-pure geometric morphisms are the geometric morphisms $\rho$ such that $\rho_*$ preserves $\T$-indexed coproducts, see Bunge and Funk \cite[Definition 2.2.6]{SCT}. The $\A$-dense geometric morphisms agree with the geometric morphisms that are dense over $\T$ \cite[Definition 2.1]{BF}. In this case, we can just as well take $\A = \{e^*\Omega_\T\}$ without changing the $\A$-dense geometric morphisms \cite[Proposition 2.2]{BF}.

It is this choice of $\A$ (i.e.\ the sheaves that are constant over the base topos) that is the most relevant in this article. For this $\A$, we will recover $\A$-dense and $\A$-pure as special cases of $n$-pure geometric morphisms later on, for $n=-1$ resp.\ $n=0$.
\item If $\A$ contains the single object $e^*\Omega_\T$, then the $\A$-pure geometric morphisms are precisely the pure geometric morphisms, again we refer to Bunge and Funk \cite[Definition 2.2.6]{SCT}.
\item If $\A$ is the family of all objects in $\E$, then $\rho: \F \to \E$ is $\A$-pure if and only if it is connected, and $\A$-dense if and only if it is surjective.
\end{enumerate}
\end{example}

For a monomorphism $j: V \hookrightarrow U$ in $\E$, we say that $j$ is $\A$-pure (resp.\ $\A$-dense) if the associated geometric morphism $\E/V \hookrightarrow \E/U$ is $\A$-pure (resp.\ $\A$-dense).

Using the results of Subsection \ref{ssec:alternative-description}, we arrive at the following result.

\begin{theorem}
Let $\E$ be a Grothendieck topos, and let $\A$ be a family of objects in $\E$. Let $J$ be the topology with as $J$-dense monomorphisms precisely the $\A$-pure (resp.\ $\A$-dense) monomorphisms, and let $\E_J$ be the corresponding subtopos. Then $\E_J$ is the smallest $\A$-pure (resp.\ $\A$-dense) subtopos. More precisely, a subtopos inclusion $\E' \subseteq \E$ is $\A$-pure (resp.\ $\A$-dense) if and only if $\E_J \subseteq \E'$.
\end{theorem}


\section{Homotopy theory in a topos via simplicial objects}

Let $\rho: \F \to \E$ be a geometric morphism between Grothendieck toposes. The derived pushforward $R^i\rho_*(A)$ is often defined only via injective resolutions, for an abelian group object $A$ in $\F$. We will use the alternative approach via Eilenberg–-Mac Lane objects, which has the advantages that then $R^i\rho^*(A)$ is also defined for a nonabelian group object $A$ if $i = 1$, or even an arbitrary object $A$ if $i=0$.

Our approach is based on the theory of toposes of sheaves of simplicial sets as studied by Joyal \cite{joyal-letter-grothendieck}, Jardine \cite{jardine-lectures} \cite{jardine-book} and Rezk \cite{rezk}. The themes that we consider are common themes in higher topos theory, as discussed (using different language) by Lurie in \cite{lurie}.

\subsection{Simplicial objects}
\label{ssec:simplicial-objects}

The first step is to consider simplicial objects in $\E$, 
rather than chain complexes of abelian group objects.
The category of simplicial objects in $\E$ will be denoted by $\s\E$.
Note that $s\E$ is again a topos:
$\s\Set$ is the topos of simplicial sets and if
$e : \E \to \Set$ and $\gamma : \s\Set \to \Set$ are the
natural geometric morphisms,
then we have a pullback of toposes
\begin{equation*}
\begin{tikzcd}
\s\E \ar[r,"{\gamma_\E}"] \ar[d,"{\s e}"'] & \E \ar[d,"{e}"] \\
\s\Set \ar[r,"{\gamma}"] & \Set
\end{tikzcd}.
\end{equation*}
More generally, for any geometric morphism $\rho : \F \to \E$ there is an induced
geometric morphism $\s \rho : \s\F \to \s\E$, obtained by taking
the pullback of $\rho$ along $\gamma$. 
By Proposition \ref{prop:direct-image-pointwise} and Proposition \ref{prop:inverse-image-pointwise}, 
$\s \rho_*$ and $\s \rho^*$ are computed pointwise,
i.e.\ $(\s \rho_*X)_m \simeq \rho_*(X_m)$ and $(\s \rho^*X)_m \simeq \rho^*(X_m)$.

If $\rho$ has a certain property that is local with respect to base change (inclusion, surjection, localic, hyperconnected, locally connected, tidy, \dots) then $\s \rho$ has that same property. Moreover, the geometric morphism $\gamma : \s\Set \to \Set$ is hyperconnected ($\s\Set$ is two-valued), so in particular it is an open surjection. So if $\s \rho$ has a certain property that can be checked after base change by an open surjection (surjection, hyperconnected, locally connected, tidy, \dots), then $\rho$ has that same property, see \cite[Corollary C5.1.7]{EL}.

The category $\s\E$ is enriched over simplicial sets:
for two objects $X$ and $Y$ in $\s\E$, we define the simplicial set
\begin{equation*}
\HOM(X,Y) = (\s e)_*(Y^X).
\end{equation*}
We can recover the usual Hom-set in $\s\E$ as
\begin{equation*}
\Hom(X,Y) = \gamma_*\HOM(X,Y),
\end{equation*}
i.e.\ $\Hom(X,Y)$ is the set of points of $\HOM(X,Y)$.

If we write $\E$ as a topos of sheaves $\E \simeq \Sh(\C,J)$, then the category of simplicial objects in $\E$ is equivalent to the category of sheaves of simplicial sets. For a simplicial object $F$ in $\E$, the corresponding sheaf of simplicial sets $\tilde{F}$ is given by
\begin{equation*}
\tilde{F}(U) = \HOM(\gamma_\E^*(\mathbf{y} U), F).
\end{equation*}

It will often be clear from the context whether we interpret $F$ as simplicial object in $\E$ or as sheaf of simplicial sets; in this case, we can simply write $F$ when we mean $\tilde{F}$.

\subsection{Homotopy groups}
\label{ssec:homotopy-groups}

Let $\E \simeq \Sh(\C,J)$ be a Grothendieck topos. Following Joyal \cite{joyal-letter-grothendieck}, for each sheaf $F$ in $\s\E$ we define $\pi_0(F)$ to be the sheafification of the presheaf
\begin{equation*}
U \mapsto \pi_0(F(U))
\end{equation*}
for $U$ in $\C$.

\begin{remark} \label{rem:independence-of-site}
The above definition is independent of the choice of Grothendieck site.
This is due to a phenomenon unrelated to $\pi_0$ that works for any presheaf on $\E$.
Indeed, let $\C \subseteq \E$ be a small full subcategory and let
$J_\C$ be the restriction to $\C$ of the canonical topology on $\E$.
If $G: \E^\op \to \Set$ is a presheaf, then we can restrict
$G$ to $\C$ and compute the sheafification of $G$ 
with respect to $J_\C$.
Because $\Sh(\C,J_\C)\simeq \E$, this defines an object $X_G$ in $\E$.
We claim that $X_G$ is independent of the choice of $\C$.
For different choices $\C,\C' \subseteq \E$, take a small full 
subcategory $\C'' \subseteq \E$ containing both $\C$ and $\C'$,
and then consider the commutative diagram of inclusions
\begin{equation*}
\begin{tikzcd}
& \PSh(\C) \ar[d,hook] \\
\E \ar[ru,hook] \ar[r,hook] \ar[rd,hook] & \PSh(\C'') \\
& \PSh(\C') \ar[u,hook]
\end{tikzcd}
\end{equation*}
The restrictions of the canonical topology on $\E$ to $\C$, $\C'$ and $\C''$
we will call $J_\C$, $J_{\C'}$ and $J_{\C''}$, respectively.
Commutativity of the diagram expresses precisely that taking the sheafification
with respect to $J_{\C''}$ agrees with first restricting to $\C$ or $\C'$
and then sheafifiying with respect to $J_\C$ resp.\ $J_{\C'}$. This proves
the claim above. Because $X_F$ is independent of $\C$, we can call it the
sheafification of $F$ with respect to the canonical topology on $\E$.
\end{remark}

\begin{proposition} \label{prop:pi0-as-connected-components}
Let $F$ be an object in $\s\E$, and let $\gamma : \s\E \to \E$ be the 
pullback of the global sections geometric morphism $\s\Set \to \Set$.
Then $\pi_0(F) \simeq \gamma_!(F)$.
\end{proposition}
\begin{proof}
For $\E \simeq \Set$ the stated isomorphism holds by definition.
If $\E$ is a presheaf topos, then $\gamma_!$ is computed pointwise 
by Proposition \ref{prop:connected-components-pointwise}
applied to $\T = \PSh(\C)$, $\rho : \s\Set \to \Set$
the global sections geometric morphism
and $T$ a representable presheaf (which is tiny by \cite[Example 2.2]{nlab-tiny}).
If $\E$ is an arbitrary sheaf topos, the statement follows
because sheafification commutes with $\gamma_!$. More precisely, consider the diagram
\begin{equation*}
\begin{tikzcd}
\s\E \ar[r,"{\s j}",hookrightarrow] \ar[d,"{\gamma}"'] & \s\PSh(\C) \ar[d,"{\gamma}"] \\
\E \ar[r,"{j}"',hookrightarrow] & \PSh(\C)
\end{tikzcd}
\end{equation*}
with $j$ an inclusion in a presheaf topos $\PSh(\C)$. Because $\gamma$ is locally connected, the Beck--Chevalley condition $\gamma^*j_* \simeq (\s j)_*\gamma^*$ holds \cite[Theorem C3.3.15]{EL}. Taking left adjoints, we find $j^*\gamma_! \simeq \gamma_!(\s j)^*$, so we can calculate $\gamma_!$ by first calculating it at the level of presheaves and then sheafifying (and this agrees with how we defined $\pi_0$).
\end{proof}

For a simplicial set $X$ and a basepoint $x \in X_0$, we can also define the higher homotopy group $\pi_n(X,x)$ for $n \geq 1$. In order to remove the basepoint from the notation, we define
\begin{equation*}
\pi_n(X) = \bigsqcup_{x \in X_0} \pi_n(X,x),
\end{equation*}
again following Joyal \cite{joyal-letter-grothendieck}. There is a natural projection map $\pi_n(X) \to X_0$. For an object $F$ in $\s\E$, we then define $\pi_n(F)$ as the sheafification of
\begin{equation*}
U ~\mapsto~ \pi_n(F(U)).
\end{equation*}
Again we obtain a natural map 
\begin{equation} \label{eq:nth-homotopy-map}
\pi_n(F) \to F_0.
\end{equation}
We will say that a morphism $F \to G$ is a \textbf{weak equivalence} if
\begin{enumerate}
\item[(WE1)] the induced map $\pi_0(F) \to \pi_0(G)$ is an isomorphism, and
\item[(WE2)] for each $n \geq 1$,
\begin{equation*}
\begin{tikzcd}
\pi_n(F) \ar[r]\ar[d] & \pi_n(G) \ar[d] \\
F_0 \ar[r] & G_0
\end{tikzcd}
\end{equation*}
is a pullback diagram.
\end{enumerate}

\subsection{The injective model structure}
\label{ssec:injective-model-structure}

A model structure on a category is a choice of three families of morphisms, called weak equivalences, cofibrations and fibrations, satisfying certain axioms, see \cite[Definition 7.1.3]{hirschhorn}. If a morphism is both a cofibration and a weak equivalence, then it is called a trivial cofibration (or acyclic cofibration). Similarly, if a morphism is both a fibration and a weak equivalence, then it is called a trivial fibration (or acyclic fibration).

\begin{theorem}[Injective model structure {\cite[Theorem 1]{joyal-letter-grothendieck}}]
\label{thm:injective-model-structure}
Let $\E$ be a topos. There is a unique model structure on $\s\E$ with weak equivalences as above and as cofibrations the monomorphisms.
\end{theorem}

The uniqueness follows since in in any model category, the fibrations are precisely the maps that have the right lifting property with respect to trivial cofibrations \cite[Proposition 7.2.3]{hirschhorn}. The model structure from Theorem \ref{thm:injective-model-structure} is called the \textbf{injective model structure}.\footnote{The injective model structure is due to Joyal, but is not to be confused with the Joyal model structure on $\s\Set$. In the special case $\E=\Set$, the injective model structure agrees with the classical (Quillen) model structure on $\s\Set$, not with the Joyal model structure on $\s\Set$.}

\begin{remark} \label{rem:all-objects-cofibrant}
Because the cofibrations are precisely the monomorphisms, we find that all objects are cofibrant in this model structure. As a result, if two fibrant objects are weakly equivalent, i.e.\ connected by a zig-zag of weak equivalences, then they define isomorphic objects in the homotopy category, which in turn implies that the two (cofibrant--)fibrant objects are homotopy equivalent (see \cite[Theorem 1.2.10]{hovey}).
\end{remark}

\subsection{Quillen adjunctions}
\label{ssec:quillen-adjunctions}

Recall that a \textbf{Quillen adjunction} between model categories is an adjunction $L \dashv R$ that satisfies one of the following equivalent properties \cite[Proposition 8.5.3]{hirschhorn}:
\begin{itemize}
\item $L$ preserves cofibrations and trivial cofibrations;
\item $R$ preserves fibrations and trivial fibrations;
\item $L$ preserves cofibrations and $R$ preserves fibrations;
\item $L$ preserves trivial cofibrations and $R$ preserves trivial fibrations.
\end{itemize}
Now let $L \dashv R$ be a Quillen adjunction. If for every cofibrant object $X$ the derived adjunction unit $X \to R(L(X)^f)$ is a weak equivalence, then we say that $L \dashv R$ is a \textbf{Quillen co-reflection}. If for every fibrant object $Y$ the derived adjunction counit $L(R(Y)^c) \to Y$ is an equivalence, then we say that $L \dashv R$ is a \textbf{Quillen reflection}. We say that $L \dashv R$ is a \textbf{Quillen equivalence} if it is both a Quillen reflection and Quillen co-reflection \cite{nlab-quillen-reflection} \cite{nlab-quillen-equivalence}.

\begin{example} \label{eg:projective-model-structure}
On a category of simplicial presheaves $\s\PSh(\C)$, we can define a \textbf{projective model structure} (or \textbf{Bousfield--Kan model structure}) with the same weak equivalences as in the injective model structure, and with as fibrations the maps $X \to Y$ such that $X(C) \to Y(C)$ is a fibration for each $C$ in $\C$ \cite[Example 4.1]{rezk}. There is an adjunction
\begin{equation*}
\begin{tikzcd}
\s\PSh(\C)^\mathrm{inj} \ar[r,bend right] & \s\PSh(\C)^\mathrm{proj} \ar[l,bend right]
\end{tikzcd}
\end{equation*}
consisting of two identity functors; with the superscript we mean that we take the injective model structure on the left and the projective model structure on the right. This is a Quillen adjunction, because, as we will recall in Corollary \ref{cor:fibration-is-sectionwise-fibration}, any fibration for the injective model structure is also a fibration for the projective model structure. Since both functors in the adjunction are identity functors, we can conclude that it is even a Quillen equivalence (as observed in e.g.\ \cite{nlab-model-structure-on-simplicial-presheaves}).
\end{example}

\subsection{Geometric morphisms} In calculations, it will often be important that the adjunctions we work with are Quillen adjunctions. Most notably, if $\rho : \F \to \E$ is a geometric morphism, we would like that $\s \rho^* \dashv \s \rho_*$ is a Quillen adjunction. It was already stated by Joyal in \cite{joyal-letter-grothendieck} that this is the case. In fact, a stronger property holds: $\s \rho^*$ not only preserves cofibrations (since the cofibrations are the monomorphisms) but it also preserves weak equivalences. To prove that $\s \rho^*$ preserves weak equivalences, Joyal argues that homotopy groups have a construction via colimits and finite limits, both of which are preserved by $\s \rho^*$. This was later written out in more detail by Illusie in \cite[p.~22]{illusie}.

We state the conclusion as follows.
\begin{proposition}[Joyal] \label{prop:pullback-preserves-homotopy-groups}
Let $\rho: \F \to \E$ be a geometric morphism. Then 
\begin{equation*}
\pi_k(\s\rho^*(X)) \simeq \rho^*\pi_k(X)
\end{equation*}
for any $X$ in $\s\E$ and natural number $k$, and $\s\rho^* \dashv \s\rho_*$ is a Quillen adjunction.
\end{proposition}

If $\rho$ is a local geometric morphism, then $\s \rho$ is local as well. In this case, $\s \rho_*$ has a further right adjoint $\s \rho^!$. Now $\s \rho_*$ preserves both colimits and finite limits (even all limits), so we can use the argument by Joyal above to conclude that $\s \rho_*$ preserves weak equivalences. In particular, $\s \rho_* \dashv \s \rho^!$ is a Quillen adjunction.

If $\rho$ is a locally connected geometric morphism, then $\s \rho$ is again locally connected, and in this case $\s \rho^*$ has a further left adjoint $\s \rho_!$. Here $\s \rho_!$ does not necessarily preserve monomorphisms, and in that case $\s \rho_! \dashv \s \rho^*$ fails to be a Quillen adjunction. However, we claim that $\s \rho_! \dashv \s \rho^*$ is a Quillen adjunction whenever $\rho$ is \'etale. To prove this, we need the following proposition:

\begin{proposition} \label{prop:connected-components-etale-preserves-homotopy}
Let $\E$ be a topos and $X$ an object of $\E$. Consider the \'etale geometric morphism $\xi : \E/X \to \E$.
Then for each $n \geq 0$, $\pi_n(\s\xi_!(F)) \simeq \xi_!(\pi_n(F))$. In particular, $\s\xi_!$ preserves weak equivalences.
\end{proposition}
\begin{proof}
It is enough to construct an isomorphism $\pi_n(\s\xi_!(F)) \simeq \xi_!(\pi_n(F))$ in the case where $\E\simeq\PSh(\C)$ is a presheaf topos, because both $\pi_n$ and $\xi_!$ commute with sheafification.
In this case, $\PSh(\C)/X \simeq \PSh(\int_\C X)$ is again a presheaf topos, with $\int_\C X$ the category of elements of $X$ \cite{nlab:category_of_elements}.
Taking sections over an object $U$ of $\C$ also commutes with $\pi_n$. We find:
\begin{align*}
\pi_n(\s\xi_!(F))(U) &\simeq \pi_n(\s\xi_!(F)(U)) \\
&\simeq \pi_n(\bigsqcup_{x \in X(U)} F(x)) \\
&\simeq \bigsqcup_{x \in X(U)} \pi_n(F(x)) \\
&\simeq \bigsqcup_{x \in X(U)} \pi_n(F)(x) \\
&\simeq (\xi_! (\pi_n(F)))(U).
\end{align*}
We conclude that $\pi_n(\s\xi_!(F)) \simeq \xi_!(\pi_n(F))$. We also have $(\s\xi_!F)_0 \simeq \xi_!(F_0)$ by Proposition \ref{prop:connected-components-pointwise}. The statement that $\s\xi_!$ preserves weak equivalences then follows because $\xi_!$ preserves pullback diagrams.
\end{proof}

\begin{corollary}
If $\rho: \F \to \E$ is \'etale, then $\s\rho_! \dashv \s\rho^*$ is a Quillen adjunction.
\end{corollary}
\begin{proof}
Note that $\s \rho_!$ preserves monomorphisms, and by Proposition \ref{prop:connected-components-etale-preserves-homotopy}
it also preserves weak equivalences. So it preserves both cofibrations and trivial cofibrations.
\end{proof}

\begin{corollary} \label{cor:fibration-is-sectionwise-fibration}
Let $f: X \to Y$ be a fibration in $\s\E$. Then for every $U$ in $\E$, the induced map of simplicial sets $X(U) \to Y(U)$ is again a fibration.
\end{corollary}
\begin{proof}
Let $\xi: \E/U \to \E$ be the \'etale geometric morphism associated to $U$, and let $p: \E/U \to \Set$ be the global sections geometric morphism. Then the induced map $X(U) \to Y(U)$ is precisely $\s p_*(\s\xi^*(f))$, and both $\s p_*$ and $\s \xi^*$ preserve fibrations.
\end{proof}


\section{Localization, descent and orthogonality}

\subsection{Localization}

Following Rezk \cite{rezk}, we note that for a Grothendieck topos $\E$ over $\Set$, the topos $\s\E$ with the injective model structure is a simplicial model category, via the enrichment in simplicial sets as in Subsection \ref{ssec:simplicial-objects}.

\begin{definition} \label{def:derived-mapping-space}
Let $X$ and $Y$ be two objects in a simplicial model category. We define the \textbf{derived mapping space} $R\HOM(X,Y)$ as
\begin{equation*}
R\HOM(X,Y) = \HOM(X^c,Y^f)
\end{equation*}
with $X^c$ a cofibrant replacement of $X$ and $Y^f$ a fibrant replacement of $Y$.
\end{definition}

The derived mapping space is a simplicial set defined up to weak equivalence. 

\begin{definition}[{\cite[Definition 3.1.4]{hirschhorn}}]
Let $\mathcal{C}$ be a simplicial model category and $S$ a class of morphisms in $\mathcal{C}$. Then we say that an object $X$ in $\mathcal{C}$ is \textbf{$S$-local} if it is fibrant and if for every morphism $A \to B$ in $S$ the induced map $R\HOM(B,X) \to R\HOM(A,X)$ is a weak equivalence. A map $f: A \to B$ is said to be an \textbf{$S$-local equivalence} if the induced map $R\HOM(B,X) \to R\HOM(A,X)$ is a weak equivalence for every $S$-local object $X$.
\end{definition}

\begin{remark}
Let $\E$ be a Grothendieck topos and $X$ a fibrant object in $\s\E$. As is standard in this article, we consider $\s\E$ with the injective model structure. Because every object in $\s\E$ is cofibrant, the condition for $X$ to be $S$-local simplifies to the condition that the induced map
\begin{equation*}
\HOM(B,X) \longrightarrow \HOM(A,X)
\end{equation*}
is a weak equivalence, for every morphism $A \to B$ in $S$.
\end{remark}

\begin{definition}
Let $\mathbf{M}$ be a simplicial model category and let $S$ be a class of morphisms in $\mathbf{M}$. A \textbf{left Bousfield localization} $\mathbf{M}_S$ of $\mathbf{M}$ with respect to $S$ is a model category with the following properties:
\begin{itemize}
\item the underlying category of $\mathbf{M}_S$ is the same as that of $\mathbf{M}$;
\item the cofibrations in $\mathbf{M}_S$ are precisely the cofibrations in $\mathbf{M}$;
\item the weak equivalences in $\mathbf{M}_S$ are precisely the $S$-local equivalences in $\mathbf{M}$.
\end{itemize}
\end{definition}

If $\mathbf{M}$ is left proper \cite[Definition 13.1.1(1)]{hirschhorn}, then the fibrant objects in $\mathbf{M}_S$ are precisely the $S$-local objects in $\mathbf{M}$ \cite[Proposition 3.4.1(1)]{hirschhorn}.

A model category is said to be combinatorial if it is locally presentable as a category and cofibrantly generated, see \cite[\S 2]{dugger}. 

\begin{theorem}[{See \cite[Theorem 7.1]{nlab:left-bousfield}}]
If $S$ is a set and $\mathbf{M}$ is a left proper combinatorial simplicial model category, then the left Bousfield localization $\mathbf{M}_S$ exists and is again a left proper combinatorial simplicial model category.
\end{theorem}

\begin{example} 
Fix a Grothendieck topos $\E$. We consider the model category $\s\E$ (with the injective model structure). This is the model category of simplicial sheaves as originally introduced by Joyal in \cite{joyal-letter-grothendieck}. This model category is left proper because all objects are cofibrant, see \cite[Corollary 13.1.3(1)]{hirschhorn}. It is also cofibrantly generated, see \cite[Theorem 2.8, Example 3.1]{beke-model}. Finally, because $\s\E$ is a topos, it is also locally presentable as a category. So we see that $\s\E$ is a left proper combinatorial simplicial model category.

By the above theorem, it follows that the left Bousfield localization of $\s\E$ with respect to a set of morphisms $S$ exists and that it is again a left proper combinatorial simplicial model category.
\end{example}

\subsection{Localizations versus Quillen reflections}
\begin{lemma}
Let $\mathbf{M}$ be a simplicial model category and let $\mathbf{M}_S$ be a left Bousfield localization with respect to some class of morphisms $S$. Then the adjunction
\begin{equation*}
\begin{tikzcd}
\mathbf{M}_S \ar[r,bend right] & \mathbf{M} \ar[l,bend right]
\end{tikzcd}
\end{equation*}
consisting of identity functors, is a Quillen reflection.
\end{lemma}
\begin{proof}
We need to show that the cofibrant replacement $Y^c \to Y$ of $Y$ in $\mathbf{M}$ is an $S$-local equivalence. This is the case because any weak equivalence is an $S$-local equivalence \cite[Proposition 3.1.5]{hirschhorn}.
\end{proof}

\begin{lemma} \label{lmm:when-is-quillen-reflection-a-localization}
Let $\mathbf{M}$ be a model category and let $\mathbf{M}_S$ be a left Bousfield localization with respect to some class of morphisms $S$. Then a Quillen reflection
\begin{equation*}
\begin{tikzcd}
\mathbf{M}' \ar[r,bend right,"{R}"'] & \ar[l,bend right,"{L}"'] \mathbf{M}
\end{tikzcd}
\end{equation*}
establishes a Quillen equivalence between $\mathbf{M}'$ and $\mathbf{M}_S$ if and only if
\begin{enumerate}
\item for each map $f$ in $\mathbf{M}$ that is a cofibration and $S$-local equivalence, $Lf$ is a weak equivalence; and
\item for each cofibrant object $X$ in $\mathbf{M}$, the map $X \to R(L(X)^f)$ is an $S$-local equivalence.
\end{enumerate}
\end{lemma}
\begin{proof}
The ``only if'' direction follows from the fact that in a Quillen equivalence the left adjoint preserves trivial cofibrations and the counit is a weak equivalence.

For the ``if'' direction, suppose that the two assumptions hold. Consider the following adjunctions.
\begin{equation} \label{eq:adjunction-factorization}
\begin{tikzcd}
\mathbf{M}' \ar[r,bend right,"{R}"'] & \ar[l,bend right,"{L}"'] \mathbf{M}_S \ar[r,bend right,"{\mathrm{id}}"'] & \mathbf{M} \ar[l,bend right,"{\mathrm{id}}"']
\end{tikzcd}
\end{equation}
We know that $L$ preserves cofibrations here and assumption (1) states that it also preserves trivial cofibrations. So $L \dashv R$ is a Quillen adjunction between $\mathbf{M}'$ and $\mathbf{M}_S$. Assumption (2) states that it is a Quillen co-reflection. It remains to show that it is also a Quillen reflection. Because the composed adjunction in \eqref{eq:adjunction-factorization} is a Quillen reflection, we have that $L(R(X)^c) \to X$ is an $S$-local equivalence, for $X$ a fibrant object in $\mathbf{M}'$ and $R(X)^c \to R(X)$ a cofibrant replacement in $\mathbf{M}$. And because the right adjunction in \eqref{eq:adjunction-factorization} is a Quillen reflection, this cofibrant replacement $R(X)^c \to R(X)$ is also a cofibrant replacement in $\mathbf{M}_S$. So the left adjunction in \eqref{eq:adjunction-factorization} is a Quillen reflection as well.
\end{proof}

\subsection{Composition of localizations}

As an application of Lemma \ref{lmm:when-is-quillen-reflection-a-localization}, we can take a look at what happens for a composition of two left Bousfield localizations.

\begin{proposition} \label{prop:composition-of-left-Bousfield-localizations}
For $\mathbf{M}$ a left proper model category and $S$ and $S'$ two families of morphisms in $\mathbf{M}$, we find that $(\mathbf{M}_S)_{S'}$ and $\mathbf{M}_{S \cup S'}$ are Quillen equivalent, if both exist.
\end{proposition}
\begin{proof}
Consider the Quillen reflections
\begin{equation*}
\begin{tikzcd}
(\mathbf{M}_S)_{S'} \ar[r, bend right, "{R'}"']
& \mathbf{M}_S \ar[r, bend right,"{R}"'] \ar[l, bend right, "{L'}"']
& \mathbf{M} \ar[l, bend right,"{L}"']
\end{tikzcd}.
\end{equation*}

Taking into account Lemma \ref{lmm:when-is-quillen-reflection-a-localization}, it is enough to prove the following two statements:
\begin{enumerate}
\item for a map $f$ in $\mathbf{M}$ that is a cofibration and $(S\cup S')$-local equivalence, $L'(L(f))$ is a weak equivalence; and
\item for each cofibrant object $X$ in $\mathbf{M}$, the map $X \to R(R'(L'(L(X))^f))$ is an $(S\cup S')$-local equivalence.
\end{enumerate}

For the first item, we prove the stronger statement that the weak equivalences in $(\mathbf{M}_S)_{S'}$ are precisely the $(S \cup S')$-local equivalences in $\mathbf{M}$. The weak equivalences in $(\mathbf{M}_S)_{S'}$ are the morphisms $f$ such that $R\HOM(f,X)$ is a weak equivalence for every object $X$ that is $S'$-fibrant in $\mathbf{M}_S$. Because $\mathbf{M}$ is left proper, we find that $X$ is $S'$-fibrant in $\mathbf{M}_S$ if and only if it is both $S$-fibrant and $S'$-fibrant in $\mathbf{M}$. This is the case if and only if $X$ is $(S \cup S')$-fibrant in $\mathbf{M}$, which is what we wanted to prove.

Now that we know that the $(S\cup S')$-local equivalences in $\mathbf{M}$ are precisely the weak equivalences in $(\mathbf{M}_S)_{S'}$, the second statement also becomes easier. The given map, as a map in $(\mathbf{M}_S)_{S'}$, is precisely a fibrant resolution of $X$ in $(\mathbf{M}_S)_{S'}$ (keeping in mind that all functors are the identity on objects and morphisms). Such a fibrant resolution is a weak equivalence by definition.
\end{proof}

\subsection{Localizations in practice}

Through some examples we show how left Bousfield localizations are used in practice.

\begin{example} \label{eg:localization} \ 
\begin{enumerate}
\item Let $\mathbf{M}$ be a simplicial model category and let $S$ be the class of weak equivalences in $\mathbf{M}$. From \cite[Corollary 9.3.3(2)]{hirschhorn} we know that $R\HOM(-,X)$ preserves weak equivalences, for $X$ fibrant. So the $S$-local objects are precisely the fibrant objects. From \cite[Proposition 9.7.1(1)]{hirschhorn} it then follows that the $S$-local equivalences are the weak equivalences.
\item \label{eg:reflection-as-localization} Consider a Quillen reflection between simplicial model categories
\begin{equation*}
\begin{tikzcd}
\mathbf{M}' \ar[r, bend right, "{R}"'] & \mathbf{M} \ar[l, bend right, "{L}"']
\end{tikzcd}
\end{equation*}
and let $S$ be the class of morphisms $\phi$ in $\mathbf{M}$ such that $L(\phi^c)$ is a weak equivalence.
For a fibrant object $Y$ in $\mathbf{M}'$ it then follows that $R(Y)$ is $S$-local, using the adjunction and \cite[Corollary 9.3.3(2)]{hirschhorn}. We then conclude that each $S$-local equivalence $\phi$ has the property that $L(\phi^c)$ is a weak equivalence using the adjunction and \cite[Proposition 9.7.1(1)]{hirschhorn}. Conversely, all morphisms in $S$ are $S$-local equivalences, so we see that the $S$-local equivalences are exactly the morphisms $\phi$ with $L(\phi^c)$ a weak equivalence. This can be used to show the two properties in Lemma \ref{lmm:when-is-quillen-reflection-a-localization}, and in this way conclude that $\mathbf{M}'$ is Quillen equivalent to $\mathbf{M}_S$ (assuming $\mathbf{M}_S$ exists).
\item As a special case of the above, we can consider the Quillen reflection
\begin{equation*}
\begin{tikzcd}
\s\Sh(\C,J) \ar[r, bend right, "{R}"'] & \s\PSh(\C) \ar[l, bend right, "{L}"']
\end{tikzcd}
\end{equation*}
with $R$ the inclusion and $L$ the sheafification, and with both categories equipped with the injective model structure. Note that in these model categories all objects are cofibrant. If we now take $S$ to be the class of morphisms $\phi$ in $\s\PSh(\C)$ such that $L(\phi^c)=L(\phi)$ is a weak equivalence, as in the example above, then we conclude that $\s\Sh(\C,J)$ is Quillen equivalent to the left Bousfield localization of $\s\PSh(\C)$ with respect to $S$. This is a well-known result that we can in this way recover. Originally, Joyal in \cite{joyal-letter-grothendieck} worked with the model category $\s\Sh(\C,J)$. As an equivalent approach, Jardine later worked with the left Bousfield localization $\s\PSh(\C)_S$. By definition of the left Bousfield localization, the fibrant objects in $\s\PSh(\C)_S$ are the objects that are fibrant in $\s\PSh(\C)$ and in addition satisfy descent with respect to the morphisms in $S$.
\item Now let $i: \E' \hookrightarrow \E$ be an inclusion, with $\E$ and $\E'$ Grothendieck toposes. We then similarly find that $\s\E'$ is Quillen equivalent to a left Bousfield localization of $\s\E$. More precisely, it is the left Bousfield localization with respect to the morphisms $\phi$ in $\s\E$ such that $(\s i)^*(\phi)$ is a weak equivalence.
\end{enumerate}
\end{example}

\subsection{Descent}

The following definition is inspired by \cite[\S 7]{jardine-lectures}.

\begin{definition} \label{def:descent}
Consider a Quillen reflection
\begin{equation*}
\begin{tikzcd}
\mathbf{M}' \ar[r, bend right,"{R}"'] & \mathbf{M} \ar[l, bend right,"{L}"'] \\
\end{tikzcd}.
\end{equation*}
For a cofibrant object $X$ of $\mathbf{M}$, we say that $X$ satisfies \textbf{descent} with respect to $L \dashv R$ if and only if the derived adjunction unit $X \to R((LX)^f)$ is a weak equivalence.
\end{definition}

Note that in $\s\E$, all objects are cofibrant. But for the moment, we want a definition that makes sense for more general model categories as well.

\begin{lemma} \label{lmm:satisfies-descent-if-in-image}
For a Quillen reflection $L \dashv R$ as above, a cofibrant object $X$ of $\mathbf{M}$ satisfies descent if and only if it is weakly equivalent to $R(Y)$ for some fibrant object $Y$ of $\mathbf{M}'$.
\end{lemma}
\begin{proof}
The ``only if'' direction follows immediately from the definition.

Now suppose that $X$ is weakly equivalent to $R(Y)$. Take the fibrant replacement $p: X \to X^f$. Then $X^f$ and $R(Y)^c$ are weakly equivalent, and since they are both cofibrant--fibrant, this means that they are homotopy equivalent \cite[Theorem 1.2.10]{hovey}. Take a homotopy equivalence $g: X^f \to R(Y)^c$. Now the map $\epsilon_Y \circ L(g\circ p)$ establishes $Y$ as fibrant replacement of $L(X)$, with $\epsilon_Y : L(R(Y)^c)\to Y$ the derived counit, which is a weak equivalence. We have to show that its companion map $h: X \to R(Y)$ is a weak equivalence. Note however that $h = gp$, and both $g$ and $p$ are weak equivalences by definition.
\end{proof}

\subsection{Descent in practice}

The following special cases can be considered:

\begin{example} \label{eg:descent} \ 
\begin{enumerate}
\item If $F \dashv U$ is a Quillen equivalence, then any cofibrant object $X$ in $\mathbf{M}$ satisfies descent with respect to $F \dashv U$.
\item For a subtopos $i: \E' \hookrightarrow \E$, the adjunction
\begin{equation*}
\begin{tikzcd}
\s\E' \ar[r, bend right,"{\s i_*}"'] & \s\E \ar[l, bend right,"{\s i^*}"'] \\
\end{tikzcd}
\end{equation*}
is a Quillen reflection. In this case, we say that an object $X$ in $\s\E$ satisfies descent with respect to $\E' \subseteq \E$ if it satisfies descent with respect to $\s i^* \dashv \s i_*$. By Lemma \ref{lmm:satisfies-descent-if-in-image} this is the case if and only if $X$ is weakly equivalent to $\s i_*(Y)$ for some fibrant object $Y$ of $\s\E'$.
\item In the special case where $\E = \PSh(\C)$ is a presheaf topos, the above example reduces to the concept of descent from \cite[\S 7]{jardine-lectures}. Indeed, in this case there is an alternative model structure on $\s\PSh(\C)$ in which the cofibrations are again the monomorphisms, but the weak equivalences are those maps that become weak equivalences after applying $\s i^*$. If we write $\s\PSh(\C)$ with this alternative model structure as $\s\PSh(\C)^\mathrm{loc}$, then $\s i^* \dashv \s i_*$ defines a Quillen equivalence between $\s\E'$ and $\s\PSh(\C)^\mathrm{loc}$. Composing this Quillen equivalence with $\s i^* \dashv \s i_*$ yields a Quillen reflection
\begin{equation} \label{eq:spsh-spshloc}
\begin{tikzcd}[column sep=large, row sep=large]
\s\PSh(\C)^\mathrm{loc} \ar[r,bend right] & \ar[l,bend right] \s\PSh(\C)
\end{tikzcd}
\end{equation}
consisting of identity functors. Jardine in \cite[\S 7]{jardine-lectures} then says that $X$ satisfies descent if the fibrant replacement $X \to X^f$ in $\s\PSh(\C)^\mathrm{loc}$ is a sectionwise weak equivalence, or in other words a weak equivalence in $\s\PSh(\C)$. This is the case precisely if $X$ satisfies descent with respect to the Quillen reflection \eqref{eq:spsh-spshloc} according to Definition \ref{def:descent}. In this way Definition \ref{def:descent} generalizes Jardine's notion of descent.
\item Let $\mathbf{M}$ be a left proper model category and let $\mathbf{M}_S$ be a left Bousfield localization with respect to some class of morphisms $S$. Then by Lemma \ref{lmm:satisfies-descent-if-in-image}, a cofibrant object $X$ in $\mathbf{M}$ satisfies descent with respect to
\begin{equation*}
\begin{tikzcd}
\mathbf{M}_S \ar[r, bend right] & \mathbf{M} \ar[l, bend right] \\
\end{tikzcd}
\end{equation*}
if and only if $X$ is weakly equivalent to an $S$-local object (or equivalently, if and only if $X^f$ is $S$-local).
\end{enumerate}
\end{example}

\subsection{Composition and descent}

\begin{lemma} \label{lmm:descent-transitive}
Consider two Quillen reflections
\begin{equation*}
\begin{tikzcd}
\mathbf{M}'' \ar[r, bend right, "{R'}"']
& \mathbf{M}' \ar[r, bend right,"{R}"'] \ar[l, bend right, "{L'}"']
& \mathbf{M} \ar[l, bend right,"{L}"']
\end{tikzcd}.
\end{equation*}
Then for a cofibrant object $X$ in $\mathbf{M}$, the following are equivalent:
\begin{enumerate}
\item $X$ satisfies descent with respect to $L'L \dashv RR'$;
\item $X$ satisfies descent with respect to $L \dashv R$, and $L(X)$ satisfies descent with respect to
$L' \dashv R'$.
\end{enumerate}
\end{lemma}
\begin{proof}
We use the characterization of Lemma \ref{lmm:satisfies-descent-if-in-image}. If $X$ satisfies descent with respect to $L'L \dashv RR'$, then $X$ is weakly equivalent to $R(R'(Z))$ for some fibrant object $Z$ in $\mathbf{M}''$. In particular, $X$ is weakly equivalent to $R(Y)$ with $Y = R'(Z)$ fibrant, so $X$ satisfies descent with respect to $L\dashv R$. Because $L\dashv R$ is a Quillen reflection, we know that $L(X)$ is weakly equivalent to $Y$, and because $Y$ satisfies descent with respect to $L'\dashv R'$ the same can be said about $L(X)$.

Conversely, if $X$ satisfies descent with respect to $L \dashv R$, then $X \to R(L(X)^f)$ is a weak equivalence, and if additionally $L(X)$ satisfies descent with respect to $L' \dashv R'$ we get that $L(X) \to R'(L'(L(X))^f)$ is a weak equivalence as well. Combining the two gives a weak equivalence $X \to R(R'(L'(L(X))^f))$, i.e.\ $X$ satisfies descent with respect to $L'L \dashv RR'$.
\end{proof}

\begin{corollary} \label{cor:descent-quillen-equivalence}
In the setting of the previous lemma, if $L'\dashv R'$ is a Quillen equivalence, then $X$ satisfies descent with respect to $L \dashv R$ if and only if it satisfies descent with respect to $L'L \dashv RR'$.
\end{corollary}

\subsection{Relation between localization and descent}

The notions of localization and descent are closely related to each other.

\begin{proposition} \label{prop:descent-for-sheaves}
Consider a Quillen reflection
\begin{equation*}
\begin{tikzcd}
\mathbf{M}' \ar[r, bend right, "{R}"'] & \mathbf{M} \ar[l, bend right, "{L}"']
\end{tikzcd}
\end{equation*}
with $\mathbf{M}$ and $\mathbf{M}'$ simplicial model categories. Let $S$ be the class of morphisms $\phi$ in $\mathbf{M}$ such that $L(\phi^c)$ is a weak equivalence. We assume that $\mathbf{M}_S$ exists, so $\mathbf{M}'$ is Quillen equivalent to $\mathbf{M}_S$ as in Example \ref{eg:localization}(\ref{eg:reflection-as-localization}). Then the following are equivalent, for a cofibrant object $X$ in $\mathbf{M}$:
\begin{enumerate}
\item the object $X$ satisfies descent with respect to $L \dashv R$;
\item for each $\phi: A \to B$ in $S$, the induced map
\begin{equation*}
R\HOM(B,X) \longrightarrow R\HOM(A,X)
\end{equation*}
is a weak equivalence.
\end{enumerate}
\end{proposition}
\begin{proof}
We can assume using Corollary \ref{cor:descent-quillen-equivalence} that $\mathbf{M}' = \mathbf{M}_S$ and $L$ and $R$ are identity functors. Then using Example \ref{eg:descent}(4), we see that $X$ satisfies descent with respect to $L \dashv R$ if and only if $X^f$ is $S$-local. The latter condition agrees with condition (2) above.
\end{proof}

\subsection{Left exact localizations}
\label{ssec:left-exact-localization}

We say that a left Bousfield localization $\mathbf{M}_S$ of $\mathbf{M}$ is \emph{left exact} if the derived left adjoint $\mathbf{M} \to \mathbf{M}_S$ of the identity map preserves derived finite limits. In this case, we also say that $\mathbf{M}_S$ is a left exact localization of $\mathbf{M}$ (the fact that we consider left Bousfield localizations is left implicit).

For a set of morphisms $\Sigma$ in $\s\E$, there exists a smallest family $S$ of morphisms containing $\Sigma$ such that $\s\E \longrightarrow (\s\E)_S$ is left exact, see \cite{lurie} \cite{ABFJ-I}. Following \cite{ABFJ-I}, we can even take $S$ to be a set, more precisely
\begin{equation} \label{eq:sigma-to-S}
S ~=~ (\Sigma^\Delta)^\mathrm{bc}
\end{equation}
where $(-)^\Delta$ is the closure under taking the diagonal of a morphism, and $(-)^\mathrm{bc}$ is the closure under derived pullback, along morphisms with domain in a predetermined set of `generators' for $\s\E$.

In the above we borrow the terminology `generators' from \cite{ABFJ-I}. In the context of this article, it is better to use the more verbose terminology `homotopy generators' to stress that we mean it in a homotopical sense rather than in a $1$-categorical sense. More precisely, we say that a family of objects $\A = \{X_i\}_{i \in I}$ in $\s\E$ is a family of \emph{homotopy generators} for $\s\E$ if every object in $\s\E$ can be written as a homotopy colimit of objects in $\A$ (up to weak equivalence).

Note that the objects of $\E$, interpreted as discrete objects in $\s\E$, form a family of homotopy generators for $\s\E$ \cite[Remark 2.1]{dugger-hollander-isaksen}. Taking this idea a bit further, we see that $\s\Sh(\C,J)$ has a set of homotopy generators consisting of the representable sheaves (which are all discrete); a fact that is used in \cite{lurie} \cite{ABFJ-I}. As a basic example, the category of simplicial sets has as set of homotopy generators the singleton containing the terminal object.

\begin{definition}
A (left Bousfield) localization $\mathbf{M}_S$ of $\mathbf{M}$ is called \emph{accessible} if, up to equivalence, we can take $S$ to be a set \cite{lurie} \cite{ABFJ-I}.
\end{definition}

In the notations as above, $(\s\E)_S$ is an accessible left exact localization of $\s\E$, for every set of morphisms $\Sigma$.

\subsection{The \texorpdfstring{\v{C}ech}{Cech} localization}
\label{ssec:cech-localization}

For a Grothendieck topos $\E$ and a subtopos $\E'\subseteq\E$, we saw in Example \ref{eg:localization}(4) that, up to Quillen equivalence, we can view $\s\E'$ as a left Bousfield localization of $\s\E$. More precisely, we take the left Bousfield localization with respect to the maps in $\s\E$ that become weak equivalences after derived pullback to $\s\E'$.

There is a different left Bousfield localization of $\E$ that is also interesting in this context. Consider the set $\Sigma$ of monomorphisms $\Sigma$ of the form 
\begin{equation*} \label{eq:representable-dense}
\gamma^*V \hookrightarrow \gamma^*U
\end{equation*}
for $V \hookrightarrow U$ a monomorphism in $\E$ that is dense with respect to $\E'\subseteq \E$, with $U$ in some set of generators for $\E$. In this sense, $\Sigma$ contains the covering sieves corresponding to the subtopos $\E'\subseteq\E$. Now take the accessible left exact localization associated to $\Sigma$ as in Subsection \ref{ssec:left-exact-localization}. This is called the \emph{\v{C}ech localization} of $\s\E$ corresponding to the subtopos $\E'\subseteq\E$. We will use the notation $\check{\s}\E'$ for this \v{C}ech localization; note that the original topos $\E$ is left implicit in this notation.

\subsection{Model toposes}

In \cite{rezk}, the concept of hypercompleteness is introduced in the setting of \emph{model toposes}. In this subsection, we give a short introduction to model toposes. We refer to \cite{rezk} for more details.

Until now we have only worked with simplicial presheaves on a small (1-)category. We generalize this by moving from categories to simplicially enriched categories (i.e. categories enriched in simplicial sets). Let $\C$ be a small simplicially enriched category. A \emph{simplicial presheaf} on $\C$ is then a simplicial functor $\C \to \s\Set$. The category of simplicial presheaves and simplicially enriched natural transformations between them \cite{nlab:enriched_natural_transformation} will be denoted by $\s\PSh(\C)$. If the simplicial Hom-sets in $\C$ are discrete, then we can think of $\C$ as a non-enriched small category, and then the current definition of $\s\PSh(\C)$ agrees with the one earlier in this article.

Note that small simplicially enriched categories can be interpreted as internal categories in $\s\Set$, with the special property that the simplicial set of objects is discrete. The category $\s\PSh(\C)$ above then has an alternative interpretation as the category of internal presheaves on $\C$. In particular, $\s\PSh(\C)$ is a topos \cite[Corollary B2.3.18]{EL}.

For $\C$ a simplicially enriched category, the topos $\s\PSh(\C)$ can be equipped with an \emph{injective model structure}: the weak equivalences and cofibrations are the objectwise weak equivalences and the objectwise cofibrations, respectively \cite[Proof of Proposition 3.6.1]{toen-vezzosi-1}. 

It can also be equipped with a \emph{projective model structure}: the weak equivalences and fibrations are the objectwise weak equivalences and objectwise fibrations, respectively \cite[\S 2.3.1]{toen-vezzosi-1}. There is a Quillen equivalence between the injective and projective model structures, in which the left and right adjoint are both the identity functor \cite[Proof of Proposition 3.6.1]{toen-vezzosi-1}.

\begin{definition}[Model topos ({\cite[\S 6]{rezk}})] \label{def:model-topos}
A model category is called a \emph{model topos} if it is Quillen equivalent to an accessible left exact localization (see \ref{ssec:left-exact-localization}) of $\s\PSh(\C)$, for $\C$ a simplicially enriched category. Here $\s\PSh(\C)$ is considered as a model category with the projective model structure.
\end{definition}

For further details, we refer to \cite{rezk}. For the purpose of this article, it is enough to point out the following facts:
\begin{itemize}
\item the family of model toposes is closed under Quillen equivalence (by definition) and closed under slices \cite[Corollary 6.10]{rezk};
\item $\s\E$ is a model topos for each Grothendieck topos $\E$ (using \cite[Example 6.3]{rezk});
\item an accessible left exact localization of a model topos is again a model topos.
\end{itemize}
To prove the last item, first use that for each model topos $\mathbf{M}$ you can construct a Quillen equivalence
\begin{equation*}
\begin{tikzcd}
\mathbf{M} \ar[r, bend right, "{R}"'] & \s\PSh(\C)_S \ar[l, bend right, "{L}"']
\end{tikzcd},
\end{equation*}
see \cite[Proof of Corollary 6.10]{rezk}. Now let $\mathbf{M}_{S'}$ be an accessible left exact localization of $\mathbf{M}$. The sets $S'$, $(S')^f$ and $L(R((S')^f)^c)$ have the same \textit{saturation} (see \cite[\S5.3]{rezk}), so using transport of left Bousfield localizations \cite[Theorem 3.3.20(1)(b)]{hirschhorn} we find that $\mathbf{M}_{S'}$ is Quillen equivalent to $(\s\PSh(\C)_S)_{R((S')^f)}$. By Proposition \ref{prop:composition-of-left-Bousfield-localizations} we can then show that $\mathbf{M}_{S'}$ is Quillen equivalent to an accessible left exact localization of $\s\PSh(\C)$, which is what we wanted to prove.

\begin{remark}
Using the properties above, note that we could equivalently use the injective model structure on $\s\PSh(\C)$ in Definition \ref{def:model-topos}.
\end{remark}

\begin{remark}
A \v{C}ech localization (see the previous subsection) is an accessible left exact localization, so the resulting model category $\check{\s}\E$ is a model topos.
\end{remark}

\subsection{Factorization into \texorpdfstring{$n$-connected and $n$-truncated}{n-connected and n-truncated} part}

We follow \cite{rezk}.

Take $n \geq -1$. We say that a simplicial set $X$ is $n$-truncated if for each $k > n$, the homotopy group $\pi_k(X,x_0)$ is trivial for each base point $x_0$; for $n=-1$ this means that $X$ is either empty or contractible. More generally, for $X$ an object in a simplicial model category $\mathbf{M}$, we say that $X$ is $n$-truncated if $R\HOM(Y,X)$ is an $n$-truncated simplicial set for each $Y$ in $\mathbf{M}$ \cite[p.~31]{rezk}. 

For the particular case of the model topos $\s\E$, the characterization of the $n$-truncated objects ($n \geq 0$) is very similar to that for simplicial sets: an object $F$ of $\s\E$ is $n$-truncated if and only if the maps $\pi_k(F) \to F_0$ (as in \eqref{eq:nth-homotopy-map}) are isomorphisms for $k > n$, see Remark \ref{rem:rezk-truncation}.

Following \cite[Proposition 8.3]{rezk}, we will define $n$-connected maps in a simplicial model category as those maps that are orthogonal in some sense to the $n$-truncated objects.

\begin{definition}[{\cite[Proposition 8.3]{rezk}}] \label{def:n-connected}
A map $Y \to X$ in a simplicial model category $\mathbf{M}$ is called \emph{$n$-connected} if for every $n$-truncated object $Z$ in $\mathbf{M}/X$, the induced map
\begin{equation*}
R\HOM_{\mathbf{M}/X}(X,Z) \longrightarrow R\HOM_{\mathbf{M}/X}(Y,Z)
\end{equation*}
is a weak equivalence.
\end{definition}

We say an object is $n$-connected if the map from this object to the terminal object is $n$-connected. The other way around, we say that a map $f: Z \to X$ is $n$-truncated if $Z$ is $n$-truncated in the slice topos over $X$, using $f$ as the structure map.

In a model topos, every map has a factorization as an $n$-connected map followed by an $n$-truncated map, for each $n \geq -1$ \cite[Proposition 8.5]{rezk}.

We say that a morphism or object is $\infty$-connected if it is $n$-connected for each natural number $n$.

\subsection{Topological and cotopological localization}

Let $\mathbf{M} \to \mathbf{M}_S$ be a left Bousfield localization. We say that $S$ is \emph{strongly saturated} if $S$ agrees with the family of $S$-local equivalences \cite{lurie} \cite{ABFJ-I} (in \cite{rezk}, these families are called `saturated'). For an arbitrary family $S$, the class $\overline{S}$ of $S$-local equivalences is the smallest strongly saturated class containing $S$, so we say that $\overline{S}$ is the \emph{strong saturation} of $S$. The left Bousfield localization with respect to the strong saturation $\overline{S}$ agrees with the left Bousfield localization with respect to $S$ itself. In this way, we get a one-to-one correspondence between left Bousfield localizations and strongly saturated families.

Following \cite{ABFJ-I}, a strongly saturated family $S$ is called a \emph{congruence} if the associated left Bousfield localization is left exact. For each set of maps $\Sigma$ there is a smallest congruence $S$ containing $\Sigma$ \cite{lurie} \cite{ABFJ-I}, see also Subsection \ref{ssec:left-exact-localization}. In this situation, we say that $S$ is the congruence \emph{generated by} $\Sigma$.

\begin{definition}[{Topological localization \cite[Definition 4.3.5]{ABFJ-I}}] A congruence is called \emph{topological} if it is generated by a set of $(-1)$-truncated maps.
\end{definition}

\begin{remark} \label{rem:topological-congruence}
For an arbitrary family $S$ of $(-1)$-truncated maps (that may be a proper class), there exists a smallest congruence containing it, and this congruence is topological \cite[Proposition 4.3.7]{ABFJ-I}. So in the above definition, we can replace the word `set' by `family'.
\end{remark}

In \cite{lurie} \cite{ABFJ-I}, $(-1)$-truncated maps are called monomorphisms. In the case of model toposes, this definition of monomorphism does not necessarily agree with the 1-categorical definition of monomorphism. Therefore, in \cite{rezk} the terminology `homotopy monomorphism' is used to refer to $(-1)$-truncated maps.

Because monomorphisms are $(-1)$-truncated, the \v{C}ech localizations that we discussed in Subsection \ref{ssec:cech-localization} are topological localizations. In fact, the converse is true as well: every topological localization of $\s\E$ is a \v{C}ech localization with respect to some topology on $\E$ \cite[Lurie 6.2.2.9]{lurie} \cite[Theorem 4.1.9]{ABFJ-II}. 

Next, we will define cotopological localizations. Recall that a map or object is said to be $\infty$-connected if it is $n$-connected for each $n$.

\begin{definition}[{Cotopological localization \cite[Proposition 4.2.2]{ABFJ-I}}]
A congruence is called \emph{cotopological} if each map in the congruence is $\infty$-connected.
\end{definition}

\subsection{Hypercompletion}

A special role is played by the hypercompletion, corresponding to the maximal possible cotopological congruence.
\begin{definition}[Hypercompletion]
For a model topos $\mathbf{M}$, the left Bousfield localization of $\mathbf{M}$ with respect to the $\infty$-connected maps is called the \emph{hypercompletion} of $\mathbf{M}$ and is denoted by $\mathsf{t}\mathbf{M}$.
\end{definition}

The hypercompletion was first defined in \cite{toen-vezzosi-1}, and there the name \emph{t-completion} was used instead. Later the terminology `hypercompletion' was introduced in \cite{lurie}.

In \cite[Proposition 10.3]{rezk} it is shown that the hypercompletion of a model topos is the left Bousfield localization with respect to a set (in the sense that we can find a set that has as strong saturation the class of $\infty$-connected maps). Further, it is shown that this left Bousfield localization is left exact. So the hypercompletion is an accessible left exact localization of the original model topos.

It is shown by Lurie that we can factor every accessible left exact left Bousfield localization $\mathbf{M} \to \mathbf{M}_S$ as a topological localization followed by a cotopological localization \cite[Proposition 6.5.2.19]{lurie} \cite[Remark 4.3.12]{ABFJ-II}.

More precisely, take a left exact left Bousfield localization $\mathbf{M} \to \mathbf{M}_S$. Now consider the congruence $S'$ generated by the family of $(-1)$-truncated maps in $S$. Using Remark \ref{rem:topological-congruence}, we see that $\mathbf{M} \to \mathbf{M}_{S'}$ is a topological localization. Further, all maps in $S$ become $\infty$-connected in $\mathbf{M}_{S'}$, see \cite{ABFJ-II}. Because they become $\infty$-connected, we can now consider the composition
\begin{equation*}
\mathbf{M} \to \mathbf{M}_{S'} \to (\mathbf{M}_{S'})_S
\end{equation*}
which is a factorization of $\mathbf{M}_S$, see Proposition \ref{prop:composition-of-left-Bousfield-localizations}. By the above this is a factorization in a topological part followed by a cotopological part.

The factorization in an $n$-connected followed by an $n$-truncated map is preserved by localization \cite[Proposition 7.6]{rezk}. So to show that all maps in $S$ become $\infty$-connected, it is enough to show that the $n$-truncated maps become weak equivalences, for each $n$. We refer to \cite{ABFJ-I} \cite{ABFJ-II} where these things are discussed in depth.

We end with the following definition.

\begin{definition}[Hypercompleteness] \label{def:hypercompleteness}
Let $\mathbf{M}$ be a model topos. Then we say that $\mathbf{M}$ is \emph{hypercomplete} if every $\infty$-connected map is a weak equivalence.
\end{definition}

For a hypercomplete model topos $\mathbf{M}$, the hypercompletion of $\mathbf{M}$ is $\mathbf{M}$ itself. Conversely, if $\mathbf{M}'$ is the hypercompletion of a model topos $\mathbf{M}$, we can use the definition of $\infty$-connected maps and the adjunction between $\mathbf{M}$ and $\mathbf{M}'$ to see that an $\infty$-connected map in $\mathbf{M}'$ is also $\infty$-connected in $\mathbf{M}$. From this it follows that $\mathbf{M}'$ is hypercomplete in the sense of Definition \ref{def:hypercompleteness}.

In other words, hypercompletion is an idempotent operation.

\subsection{Simplicial sheaves and hypercompletion}

The main theorem about hypercompletion, for the purposes of this article, is the following:

\begin{theorem}[{\cite{dugger-hollander-isaksen}, \cite[p.~49]{rezk}}] \label{thm:hypercompletion-of-cech}
For a topos $\E$ and a subtopos $\E' \subseteq \E$, there is a Quillen equivalence
\begin{equation*}
\mathsf{t}(\check{\s}\E') ~\simeq~ \s\E',
\end{equation*}
where $\check{\s}\E'$ denotes the \v{C}ech localization of $\s\E$ corresponding to the subtopos $\E'\subseteq\E$.
\end{theorem}
\begin{proof}
The case where $\E = \PSh(\C)$ for $\C$ a small category is \cite[Theorem 11.10]{rezk}. From this special case, it follows that $\s\E'$ is hypercomplete for each Grothendieck topos $\E'$.
Now we consider the general case. We factor the localization of $\s\E$ corresponding to $\s\E' \subseteq \s\E$ into a topological and cotopological part. Because $\s\E'$ is hypercomplete, the cotopological part must be a hypercompletion. The topological part is a \v{C}ech localization by \cite[6.2.2.9]{lurie} \cite[Theorem 4.1.9]{ABFJ-I}. Once we know that it is a \v{C}ech localization, it follows immediately that it must be precisely the \v{C}ech localization corresponding to $\E'\subseteq \E$.

The factorization into topological and cotopological part gives the Quillen equivalence $\mathsf{t}(\check{\s}\E') ~\simeq~ \s\E'$ from the theorem statement.
\end{proof}


\section{Homotopy theory with simplicial objects: cohomology}

\subsection{Eilenberg--Mac Lane objects and cohomology}
\label{ssec:em-objects-cohomology}

Fix a Grothendieck topos $\E$ and a natural number $n$. Let $A$ be an object in $\E$, equipped with a group structure if $n\geq 1$, with this group structure abelian for $n\geq 2$. We say that an object $X$ in $\s\E$ is an \emph{Eilenberg--Mac Lane object} associated to $(A,n)$ if $X$ is fibrant and
\begin{equation*}
\pi_k(X) \cong \begin{cases}
A \quad& \text{for }k=n; \\
1 \quad& \text{for }k\neq n,
\end{cases}
\end{equation*}
(the isomorphism should be an isomorphism of groups at $k=n$, for $n\geq 1$). In this case, we also say that $X$ is a $K(A,n)$. We will often simply write $K(A,n)$ to denote any Eilenberg--Mac Lane object associated to $(A,n)$. Under the assumptions above, the object $K(A,n)$ always exists and it is unique up to homotopy equivalence, see Appendix \ref{app:em-objects-and-loop-spaces}.

Now let $\rho : \F \to \E$ be a geometric morphism and let $K(A,n)$ be an Eilenberg--Mac Lane object in $\F$. Then we can define the higher direct image $R^n \rho_*(A)$ as
\begin{equation} \label{eq:derived-functor}
R^n \rho_*(A) = \pi_0(\s \rho_*K(A,n)),
\end{equation}
which is an object of $\E$.

As a special case, if $e : \E \to \T$ is the geometric morphism to the base topos $\T$, then we can use the notation
\begin{equation} \label{eq:cohomology-over-base-topos}
\HH_\T^n(\E,A) = R^n e_* A
\end{equation}
or just $\HH^n(\E,A)$ if we leave the base topos implicit.

Eilenberg--Mac Lane objects interact well with localization:

\begin{proposition} \label{prop:eilenberg-maclane-locally}
Let $\xi : \E' \to \E$ be an \'etale geometric morphism. Then for any Eilenberg--Mac Lane object $K(A,n)$ in $\E$ we have $\s\xi^*K(A,n) \simeq K(\xi^*A,n)$.
\end{proposition}
\begin{proof}
This follows from Proposition \ref{prop:pullback-preserves-homotopy-groups} and Proposition \ref{prop:connected-components-etale-preserves-homotopy}.
\end{proof}

\subsection{Comparison with definition via chain complexes} In $\eqref{eq:derived-functor}$, we defined the right derived functor via Eilenberg--Mac Lane objects. For $A$ an abelian group object in $\s\E$, the right derived functor is originally defined in a different way using chain complexes and injective resolutions. The two definitions are equivalent by the Dold-Kan correspondence \cite{nlab:dold-kan-correspondence}.

The approach with Eilenberg--Mac Lane objects is more elegant from a theoretical point of view, but in concrete calculations it is often a good idea to use the definition with chain complexes instead.

\subsection{Contravariance of cohomology}
\label{ssec:contravariance-of-cohomology}

For a geometric morphism $\rho: \F\to\E$ over $\T$
\begin{equation*}
\begin{tikzcd}
\F \ar[rr,"{\rho}"] \ar[dr,swap,"{f=e\rho}"] && \E \ar[dl,"{e}"]\\
& \T
\end{tikzcd}
\end{equation*}
and an Eilenberg--Mac Lane object $K(i,A)$ in $\s\E$ there is an induced map
\begin{equation}\label{eq:HHrho}
\HH^i(\rho,A)~:~ \HH^i(\E,A) \to \HH^i(\F,\rho^*A).
\end{equation}

This map can be constructed as follows.
For an Eilenberg--Mac Lane object $K(A,n)$ in $\E$,
take a fibrant replacement
$i : \s\rho^*K(A,n) \to K$
and now take the composition
\begin{equation*}
K(A,n) \to \s\rho_*\s\rho^*K(A,n) \to \s\rho_*K
\end{equation*}
with the adjunction unit. The functors $\pi_n$ commute with 
$\s\rho^*$ (Proposition \ref{prop:pullback-preserves-homotopy-groups}),
so $K \simeq K(\rho^*A,n)$.
Applying $\s e_*$ and then $\pi_0$ gives a map
\begin{equation*}
\HH^n(\E,A) \longrightarrow \HH^n(\F,\rho^*A).
\end{equation*}

If $A$ is an abelian group, then we can work with chain complexes and injective resolutions,
and then there is an analogous construction of the map above,
see e.g.\ Iversen \cite[Section II.5]{iversen}.

\subsection{The higher direct image as a sheaf}
\label{ssec:derived-pushforward-as-sheaf}

Let $\rho : \F \to \E$ be a geometric morphism, 
and let $K(A,n)$ be an Eilenberg--Mac Lane object in $\s\F$.
The derived pushforward
\begin{equation*}
R^n\rho_*(A) \simeq \pi_0(\s\rho_*K(A,n))
\end{equation*}
has the following description as a sheaf.
By definition of $\pi_0$, it is the sheafification of
\begin{equation*}
U ~\mapsto~ \pi_0\left(\vphantom{\int}(\s\rho_*K(A,n))(U)\right)
\end{equation*}
on $\E$. Now consider the commutative diagram
\begin{equation*}
\begin{tikzcd}
\F/\rho^*U \ar[r,"{\rho_U}"] \ar[d,"{\xi}"'] & \E/U \ar[d,"{\xi}"] \ar[r,"{\delta}"] & \Set \\
\F \ar[r,"{\rho}"] & \E
\end{tikzcd}.
\end{equation*}
Then we find:
\begin{align*}
(\s \rho_* K(A,n))(U) &\simeq \s\delta_*\s\xi^*\s\rho_*K(A,n) \\
&\simeq \s\delta_*\s\rho_{U,*}\s\xi^*K(A,n) \\
&\simeq \s\delta_*\s\rho_{U,*}K(\xi^*A,n).
\end{align*}
Here in the second isomorphism we use the Beck--Chevalley condition resulting from $\xi$ being locally connected \cite[Theorem C3.3.15]{EL}.
Note that $\delta\rho_U$ is the global sections geometric morphism
for $\F/\rho_*U$. So after taking $\pi_0$ we find
\begin{equation*}
\pi_0((\s \rho_* K(A,n))(U)) \simeq \HH^n(\F/\rho^*U,A).
\end{equation*}

We conclude:
\begin{proposition} \label{prop:derived-pushforward-as-sheaf}
Let $\rho : \F \to \E$ be a geometric morphism and $K(A,n)$ an Eilenberg--Mac Lane object in $\s\F$.
Then $R^n\rho_*(A)$ is the sheafification in $\E$ of
\begin{equation*}
U ~\mapsto~ \HH^n(\F/\rho^*U,A).
\end{equation*}
\end{proposition}

For $A$ an abelian group, this appears in \cite[Lemma 8.18]{TT}. For $n=1$ and $A$ any group, the above sheaf description could be taken as definition of $R^1\rho_*(A)$.


\section{\texorpdfstring{$n$-pure}{n-pure} geometric morphisms}

\subsection{\texorpdfstring{$n$-pure}{n-pure} geometric morphisms}

We say that a family of Eilenberg--\mbox{Mac Lane} objects  $\A$ in $\s\E$ is \textbf{downwards closed} if
\begin{equation*}
K(A,n) \in \A ~\text{and}~ m \leq n ~\Rightarrow~ K(A,m) \in \A.
\end{equation*}
In other words, $\A$ is closed under application of the loop functor $\Omega$ (see Appendix \ref{app:em-objects-and-loop-spaces}).

\begin{definition} \label{def:n-pure}
Let $\rho : \F \to \E$ be a geometric morphism over the base topos $\T$.
Let $\A$ be a downwards closed family of Eilenberg--Mac Lane objects in $\s\E$.
We shall say that $\rho$ is \textbf{pure with respect to $\A$} if
for every object $X$ in $\E$ and $K(A,m) \in \A$ the induced map 
\begin{equation*}
\HH^m(\rho/X): \HH^m(\E/X,A) \to \HH^m(\F/\rho^*X,\rho^*A)
\end{equation*}
is an isomorphism. For $n \geq 0$, we say that $\rho$ is 
\textbf{$n$-pure} over $\T$ if it is 
pure with respect to the family of all Eilenberg--Mac Lane objects
$K(e^*T,m)$ with $T$ in $\T$ and $m \leq n$.
Finally, we say that $\rho$ is \textbf{$n$-connected} if it is pure with
respect to the family of all Eilenberg--Mac Lane 
objects $K(A,m)$ in $\s\E$ with $m \leq n$.
\end{definition}

\begin{remark}
In the above definition, cohomology is taken over a 
base topos $\T$, as in \eqref{eq:cohomology-over-base-topos}.
However, we will see in Theorem \ref{thm:characterization-derived-pushforwards}
that pureness with respect to a downwards closed family is independent of the base topos.
In particular, $n$-connectedness is independent of the base topos.
The notion of being $n$-pure is not, since in this case 
the downwards closed family $\A$
is itself defined in terms of the base topos.
\end{remark}

By convention, $\rho$ is \textbf{$(-1)$-pure} (resp. $(-1)$-connected) 
if for all $X$ in $\E$ the map
\begin{equation*}
\HH^0(\rho/X) : \HH^0(\E/X,A) \to \HH^0(\F/\rho^*X,A)
\end{equation*}
is a monomorphism, for all $A = e^*T$ with $T$ in $\T$ 
(resp. for all $A$ in $\E$). We also follow the convention that every geometric morphism is $(-2)$-pure (and $(-2)$-connected).

A geometric morphism is by definition $n$-connected if and only if it is $n$-pure over the codomain topos. 
As special cases, we find that $(-1)$-connectedness and $0$-connectedness agree with surjectivity resp. connectedness.
Similarly, a geometric morphism is $(-1)$-pure over $\T$ if and only if it is $\T$-dense.
If $\A$ contains only Eilenberg--Mac Lane objects of the form $K(A,0)$, then the notion of $\A$-pure coincides with that of Section \ref{ssec:A-dense-A-pure}. In fact, $-1$-pure and $0$-pure geometric morphisms were already introduced there, see also Remark \ref{rem:comparison-to-0-pure-introduction}.

\begin{remark}
Definition \ref{def:n-pure} is closely related to the concepts of $n$-acyclic 
and $1$-aspherical morphisms of schemes as studied in SGA4 \cite[Expos\'e XV]{sga4-3}.
Before making a comparison, we first recall some definitions from SGA4.
For a set of primes $L$, an ind-$L$-group $G$ is a group such that
each finitely generated subgroup $H \subseteq G$ is finite, with the prime divisors
of the order of $H$ contained in $L$ \cite[Expos\'e IX, D\'efinition 1.3]{sga4-3}.
A sheaf of ind-$L$-groups on the petit \'etale site of $X$ is 
a sheaf of groups $F$ such that $F(U)$ is an ind-$L$-group for all $U \to X$ \'etale
with $U$ quasi-compact \cite[Expos\'e IX, D\'efinition 1.5]{sga4-3}.
An $L$-torsion sheaf is a sheaf of abelian groups $A$ such that
\begin{equation*}
A \simeq \varinjlim A_n
\end{equation*}
in the category of sheaves, with $A_n$ the kernel of the multiplication by $n$ map $A \to A$
\cite[Expos\'e IX, D\'efinition 1.1]{sga4-3}, and with the limit going over natural numbers $n$ that have prime divisors in $L$.

Let $g : X \to Y$ be a morphism of schemes, and let $\gamma : X_\etale \to Y_\etale$ be the induced
geometric morphism on petit \'etale toposes. Then:
\begin{itemize}
\item $g$ is $(-1)$-acyclic if and only if $\gamma$ is $(-1)$-connected (i.e.\ surjective) \cite[Expos\'e XV, D\'efinition 1.3]{sga4-3};
\item $g$ is $0$-acyclic if and only if $\gamma$ is $0$-connected (i.e.\ connected) \cite[Expos\'e XV, D\'efinition 1.3]{sga4-3};
\item $g$ is $n$-acyclic for a set of primes $L$, $n \geq 1$, if and only if $\gamma$ is pure with respect to the family of all $(m,A)$
with $m \leq n$ and $A$ an $L$-torsion sheaf \cite[Expos\'e XV, D\'efinition 1.7]{sga4-3};
\item $g$ is $1$-aspherical for a set of primes $L$ if and only if $\gamma$ is pure with respect to the family of all $(m,A)$ with $m\in \{0,1\}$ and
$A$ a sheaf of ind-$L$-groups \cite[Expos\'e XV, D\'efinition 1.7]{sga4-3}.
\end{itemize}
The above correspondences follow from the alternative characterizations of $n$-pureness that we discuss later on in this article.
\end{remark}

\begin{remark} \label{rem:n-connected-hoyois}
A related (but possibly different) notion of $n$-connected geometric morphism already appears in the higher topos theory literature \cite{hoyois}, extending the notion of $n$-connected topological space in topology, see Remark \ref{rem:relation-to-earlier}.
\end{remark}

We mention the following basic properties of $\A$-pure geometric morphisms.

\begin{proposition} \label{prop:basic-properties-A-pure}
Consider geometric morphisms
\begin{equation*}
\begin{tikzcd}
\mathcal{G} \ar[r,"{\psi}"] & \F \ar[r,"{\rho}"] & \E
\end{tikzcd}
\end{equation*}
and let $\A$ be a downwards closed family of Eilenberg--Mac Lane objects in $\s\E$. We write $\rho^*\A$ for the family consisting of the objects $\rho^*E$ for $E$ in $\A$.
\begin{enumerate}
\item If $\rho$ is $\A$-pure and $\psi$ is $\rho^*\A$-pure, then $\rho\psi$ is $\A$-pure.
\item If $\rho\psi$ is $\A$-pure and $\psi$ is $\rho^*\A$-pure, then $\rho$ is $\A$-pure.
\end{enumerate}
\end{proposition}
\begin{proof}
Straightforward once the definitions are written out in detail.
\end{proof}

\subsection{Independence of the base topos}

We claim that the notion of pure geometric morphism with respect to a 
downwards closed family $\A$ is independent of the base topos $\T$.

Let $\rho : \F\to\E$ be a geometric morphism over a Grothendieck topos $\T$ and
let $t : \T\to \Set$ be the unique geometric morphism to $\Set$.
Let $K(A,n)$ be an Eilenberg--Mac Lane object in $\s\E$.
We will show that the maps
\begin{equation} \label{eq:over-T}
\HH^m_\T(\rho/U,A) : \HH^m_\T(\E/U,A) \to \HH^m_\T(\F/\rho^*U,\rho^*A)
\end{equation}
are isomorphisms for all $U$ in $\E$ and $m\leq n$ if and only if the maps
\begin{equation} \label{eq:over-Set}
\HH^m(\rho/U,A) : \HH^m_\Set(\E/U,A) \to \HH^m_\Set(\F/\rho^*U,\rho^*A)
\end{equation}
are isomorphisms for all $U$ in $\E$ and $m \leq n$. Here $\HH^m_\T$ and $\HH^m_\Set$
denote cohomology with respect to the base topos $\T$ and $\Set$,
see \eqref{eq:cohomology-over-base-topos}.

\begin{proposition} \label{prop:from-Set-to-T}
Take a fixed Eilenberg--Mac Lane object $K(A,m)$ in $\E$.
If $\HH^m_\Set(\rho/U,A)$ is an isomorphism (resp.\ monomorphism) for all $U$ in $\E$,
then also $\HH^m_\T(\rho/U,A)$ is an isomorphism (resp.\ monomorphism) for all $U$ in $\E$.
\end{proposition}
\begin{proof}
We can write $\HH^m_\T(\E/U,A)$ as the sheafification of the presheaf
\begin{equation*}
T \mapsto \HH^m_\Set(\E/(U\times e^*T),A)
\end{equation*}
as in Subsection \ref{ssec:derived-pushforward-as-sheaf}, and similarly
for $\HH^m_\T(\F/\rho^*U,A)$. The statement then follows from the fact
that sheafification preserves isomorphisms and monomorphism.
\end{proof}

Conversely, we find:

\begin{proposition} \label{prop:from-T-to-Set}
Take an Eilenberg--Mac Lane object $K(A,n)$ in $\s\E$ and a fixed object $U$ in $\E$.
If $\HH^m_\T(\rho/U,A)$ is an isomorphism for all $0 \leq m \leq n$
then also $\HH^m_\Set(\rho/U,A)$ is an isomorphism for all $0 \leq m \leq n$.
\end{proposition}
\begin{proof}
We can assume without loss of generality that $U=1$, 
by replacing $\rho$ with $\rho/U$ if necessary.
We then write:
\begin{equation} \label{eq:looping}
\begin{split}
\HH^m_\T(\E,A) &\simeq \pi_0(\s e_* K(A,m)) \\
&\simeq \pi_{0}(\s e_* \Omega^{n-m}K(A,n)) \\ 
&\simeq \pi_0(\Omega^{n-m} \s e_* K(A,n)) \\
&\simeq \pi_{n-m}(\s e_* K(A,n))
\end{split}.
\end{equation}
We refer to Appendix \ref{app:em-objects-and-loop-spaces} for the definition of $\Omega$ and its relevant properties. Similarly we get $\HH^m_\T(\F,\rho^*A) \simeq \pi_{n-m}(\s f_* K(\rho^*A,n) )$.

So by assumption the induced map $\phi : \s e_* K(A,n) \to \s f_* K(\rho^* A,n)$ is an isomorphism on homotopy groups $\pi_k$ for $0 \leq k \leq n$. The Eilenberg--Mac Lane spaces $K(A,n)$ and $K(\rho^*A,n)$ are $n$-truncated, and pushforward preserves $n$-truncated objects (see the comment in the proof of Proposition 7.6 of \cite{rezk}). So the homotopy groups $\pi_k$ for $k \geq n+1$ vanish for $\s e_* K(A,n)$ and $\s f_* K(\rho^*A,n)$, and as a result $\phi$ is a weak equivalence. A weak equivalence between fibrant objects is a homotopy equivalence (see Remark \ref{rem:all-objects-cofibrant}), so it is preserved by $t_*$, with $t : \T \to \Set$ the global sections geometric morphism. The statement follows.
\end{proof}

\subsection{Independence of the base topos: applications}

We can use independence of the base topos to deduce various properties or equivalent
characterizations of pure geometric morphisms with respect to a downwards closed family of Eilenberg--Mac Lane objects $\A$.

\begin{proposition} \label{prop:coproduct-pure}
Let $\rho : \F \to \E$ be a geometric morphism and let $\A$ be a downwards closed family of Eilenberg--Mac Lane objects in $\s\E$.
Suppose that $\E$ can be written as a disjoint union of toposes $\E \simeq \bigsqcup_{i \in I} \E_i$.
For each $i \in I$, let $\rho_i : \F_i \to \E_i$ be the base change of $\rho$ along $\E_i \hookrightarrow \E$.
Then $\rho$ is pure with respect to $\A$ if and only if 
for each $i \in I$, $\rho_i$ is pure with respect to $\A$.
\end{proposition}
\begin{proof}
This follows by taking as base topos $\T ~\simeq~ \bigsqcup_{i \in I} \Set$.
\end{proof}

The following is an alternative characterization of $\A$-pure geometric morphisms 
in terms of derived pushforwards. It can be seen as a generalization
of the characterization of $n$-acyclic morphisms of schemes in terms of the derived pushforward
\cite[Expos\'e XV, \S1]{sga4-3}.

\begin{theorem} \label{thm:characterization-derived-pushforwards}
Let $\rho : \F \to \E$ be a geometric morphism over the base topos $\T$ 
and take $\A= \{ K(A,0),K(A,1),\dots, K(A,n)\}$.
Then the following are equivalent, for $n \geq 0$:
\begin{enumerate}
\item $\rho$ is $n$-pure with respect to $\A$;
\item the unit map $A \to \rho_*\rho^*A$ is an isomorphism and 
$R^i\rho_*(A) \simeq 0$ for all $1 \leq i \leq n$.
\end{enumerate}
\end{theorem}
\begin{proof}
We use independence of the base topos, taking as base topos $\T \simeq \E$. We can write $\HH_\E^m(\rho/U, A)$ as the map
\begin{equation*}
\pi_0(K(\xi_U^*A,m))
\to 
\pi_0 (\s \rho_{U,*} K((\xi_U')^*\rho_U^*A,m)) 
\end{equation*}
with $\xi_U : \E/U \to \E$ and $\xi_U' : \F/U \to \F$ the \'etale geometric morphisms associated to $U$. We can further rewrite this as
\begin{equation*}
\xi^*_U\pi_0(K(A,m))
\longrightarrow
\xi^*_U\pi_0(\s\rho_* K(\rho^*A,m)),
\end{equation*}
using Proposition \ref{prop:eilenberg-maclane-locally}, the Beck--Chevalley condition $\s\rho_{U,*}(\s\xi_U')^* \simeq \s\xi^*_U\s\rho_*$ \cite[Theorem C3.3.15]{EL} and Proposition \ref{prop:pullback-preserves-homotopy-groups}. This is an isomorphism for all $U$ if and only if the map
\begin{equation*}
\pi_0(K(A,m))
\to
\pi_0(\s\rho_* K(A,m))
\end{equation*}
is an isomorphism. The statement now follows.
\end{proof}

Let $\rho: \F \to \E$ be a geometric morphism over $\T$.
\begin{equation}\label{eq:rho-e}
\begin{tikzcd}
\F \ar[rr,"{\rho}"] \ar[dr,swap,"{f}"] && \E \ar[dl,"{e}"]\\
& \T
\end{tikzcd}
\end{equation}
We say that $\rho$ \emph{preserves $\T$-indexed coproducts} if the counit map
\begin{equation} \label{eq:T-indexed-coproducts}
e^*A \to \rho_*\rho_*e^*A
\end{equation}
is an isomorphism, for each object $A$ in $\T$ \cite[Definition 2.2.6]{SCT}.

As a special case of Theorem \ref{thm:characterization-derived-pushforwards} we find:

\begin{corollary}\label{cor:0pure}
A geometric morphism over $\T$
is $0$-pure if and only if it preserves $\T$-coproducts.
\end{corollary}

\begin{remark} \label{rem:comparison-to-0-pure-introduction}
Using Theorem \ref{thm:characterization-derived-pushforwards}, we recover the notion of $\A$-pure geometric morphisms from Subsection \ref{ssec:A-dense-A-pure} by restricting to Eilenberg--Mac Lane objects of the form $K(A,0)$.
\end{remark}

\begin{corollary}\label{cor:pure-reflection}
Let $\rho : \F\to\E$ be a geometric morphism over $\T$.
For any $n\,$, if the base change of $\rho$ along an \'etale surjection 
is $n$-pure, then $\rho$ is itself $n$-pure.
\end{corollary}
\begin{proof}
This follows from Theorem \ref{thm:characterization-derived-pushforwards} and the fact that derived pushforwards can be computed locally (using for example Beck--Chevalley conditions or Proposition \ref{prop:derived-pushforward-as-sheaf}).
\end{proof}

Finally, we can reformulate Theorem \ref{thm:characterization-derived-pushforwards} in the following way, which is maybe the most conceptual characterization of $\A$-pure geometric morphisms.
\begin{corollary} \label{cor:preserves-em-objects}
Let $\rho: \F \to \E$ be a geometric morphism and let $\A$ be a downwards closed family of Eilenberg--Mac Lane objects in $\s\E$. Then $\rho$ is $\A$-pure if and only if the derived unit
\begin{equation*}
K(A,n) \to \s\rho_*((\s\rho^*K(A,n))^f)
\end{equation*}
is a homotopy equivalence, for each $K(A,n)$ in $\A$.
\end{corollary}

\begin{remark} \label{rem:relation-to-earlier}
In general, the unit of an adjunction is an isomorphism if and only if the left adjoint functor is fully faithful. So in this sense, a geometric morphism is $n$-pure if and only if the left adjoint is fully faithful, in a homotopical sense, when restricted to the simplicial objects in $\A$.

This is the point of view that is taken in the definition of $n$-connected geometric morphisms in \cite{hoyois}. The difference is that in this article, we only consider Eilenberg--Mac Lane objects $K(A,m)$ with $m \leq n$, while in \cite{hoyois} all $n$-truncated objects are considered. We don't know to what extent the two definitions are equivalent.
\end{remark}

\subsection{Relation to Galois theory}

For a locally connected Grothendieck topos $\E$, let $\SLC(\E)$ be the full subcategory of $\E$ consisting of the objects that can be written as a coproduct of locally constant objects. By \cite[Theorem 2.4]{leroy}, $\SLC(\E)$ is a Grothendieck topos and the inclusion functor $\SLC(\E) \subseteq \E$ is the inverse image functor for a geometric morphism $f : \E \to \SLC(\E)$. Since $f^*$ is fully faithful, $f$ is connected (in particular $0$-pure).

A Grothendieck topos $\E$ is said to be \textbf{locally Galois} if $\E$ is locally connected and the natural geometric morphism $f : \E \to \SLC(\E)$ is an equivalence \cite[D\'efinition 3.1]{leroy}. We can then show that $\SLC$ is an idempotent endofunctor of the category of locally connected Grothendieck toposes, with as essential image the locally Galois toposes.

Recall that a locally Galois topos is atomic (in particular Boolean), see \cite[Corollaire 2.6(a)]{leroy}. So when working with locally Galois toposes, the following proposition will be relevant.

\begin{proposition} \label{prop:0-pure-implies-hyperconnected-for-boolean}
Let $\rho: \F \to \E$ be a geometric morphism with $\F$ and $\E$ Boolean. If $\rho$ is $0$-pure, then it is hyperconnected.
\end{proposition}
\begin{proof}
It is enough to show that $f^*$ induces an isomorphism
\begin{equation} \label{eq:hyperconnected}
\mathbf{Sub}_\E(A) \longrightarrow \mathbf{Sub}_\F(f^*A)
\end{equation}
for each $A$ in $\E$, see \cite[Proposition A4.6.6(vi)]{EL}. Because $\F$ and $\E$ are Boolean, we can rewrite \eqref{eq:hyperconnected} as
\begin{equation*}
\Hom_\E(A, 2_\E) \longrightarrow \Hom_\F(f^*A, 2_\F)
\end{equation*}
with $2_\E$ and $2_\F$ the coproduct of two times the terminal object in $\E$ resp.\ $\F$. Using $f^*(2_\F) \simeq 2_\E$ and the adjunction $f^* \dashv f_*$ we can further rewrite as
\begin{equation*}
\HOM_\E(A, 2_\E) \longrightarrow \HOM_\E( A, f_*f^*2_\E).
\end{equation*}
This is an isomorphism because, as a result of $f$ being $0$-pure, we have that
\begin{equation*}
2_\E \longrightarrow f_*f^*2_\E
\end{equation*}
is an isomorphism.
\end{proof}

We can use this to deduce the following more general statement for locally connected Grothendieck toposes.

\begin{proposition} \label{prop:mono-in-the-next-galois}
Let $\rho : \F \to \E$ be a $0$-pure geometric morphism with $\F$ and $\E$ locally connected. Then the induced geometric morphism $\rho_1 : \SLC(\F) \to \SLC(\E)$ is hyperconnected. As a result, the induced functor
\begin{equation*}
\TOP(\E,\PSh(G)) \longrightarrow \TOP(\F,\PSh(G))
\end{equation*}
is fully faithful, for every group $G$.
\end{proposition}
\begin{proof}
Consider the commutative square
\begin{equation*}
\begin{tikzcd}
\F \ar[r,"{\rho}"] \ar[d,"{\gamma'}"'] & \E \ar[d,"{\gamma}"] \\
\SLC(\F) \ar[r,"{\rho_1}"] & \SLC(\E)
\end{tikzcd}
\end{equation*}
with $\gamma$ and $\gamma'$ the natural geometric morphisms from a locally connected topos to its associated locally Galois topos. Since $\gamma$ and $\gamma'$ are connected, they are in particular $0$-pure. Further, $\rho$ is $0$-pure by assumption. It follows that $\rho_1$ is $0$-pure as well, see Proposition \ref{prop:basic-properties-A-pure}(ii). Using Proposition \ref{prop:0-pure-implies-hyperconnected-for-boolean} we conclude that $\rho_1$ is hyperconnected. Note that $\PSh(G)$ classifies $G$-torsors. Since $G$-torsors are locally constant, every $G$-torsor in $\F$ or $\E$ is contained in $\SLC(\F)$ resp.\ $\SLC(\E)$. It follows that the category of $G$-torsors in $\E$ is a full subcategory of the category of $G$-torsors in $\F$, via the functor $\rho^*$.
\end{proof}

We now want to show that the natural map $\E\to \SLC(\E)$ is $1$-pure. First, we need to show that the $\SLC$ construction behaves well with respect to slicing.

\begin{proposition} \label{prop:SLC-and-slices}
Let $\E$ be a locally connected Grothendieck topos and let $\gamma : \E\to\SLC(\E)$ be the natural geometric morphism. Take a well-supported object $X$ in $\SLC(\E)$. Then there is an equivalence
\begin{equation*}
\SLC(\E/\gamma^*X) \simeq \SLC(\E)/X
\end{equation*}
such that the natural geometric morphism $\E/\gamma^*X \to \SLC(\E/\gamma^*X)$ agrees with $\gamma/X$.
\end{proposition}
\begin{proof}
It is enough to show that an object $Y \to \gamma^*X$ of $\E/\gamma^*X$ is slc (i.e.\ a coproduct of locally constant objects) if and only if $Y$ is slc. Without loss of generality, we can assume that $Y$ is connected. It remains to show that $Y \to \gamma^*X$ is locally constant if and only if $Y$ is locally constant.

First, we consider the case where $X$ is locally constant. Take a jointly well-supported family $\{U_i\}_{i \in I}$ trivializing $\gamma^*X$. Then $Y \to \gamma^*X$ is locally constant if and only if the base change $Y \times U_i \to \gamma^*X \times U_i$ is locally constant for each $i \in I$. There is an isomorphism $\gamma^*X \times U_i \cong \bigsqcup_{a \in A} U_i$ respecting the projection maps to $U_i$, for some set $A$. Let $Z_a$ be the part of $Y \times U_i$ lying above the copy of $U_i$ indexed by $a$. Then $Y \times U_i \to \gamma^*X \times U_i$ is locally constant if and only if $Z_a \to U_i$ is locally constant in $\E/U_i$ for all $a \in A$, so if and only if $Y \times U_i \to U_i$ is locally constant in $\E/U_i$. This is the case for all $i \in I$ if and only if $Y$ is locally constant.

If $X$ is not locally constant, then write $X = \bigsqcup_{j} X_j$ with $X_j$ locally constant and let $Y_j$ be the part of $Y$ lying above $\gamma^*X_j$. Then $Y \to \gamma^*X$ is slc if and only if $Y_j \to \gamma^*X_j$ is slc for each $j$, which by the above is the case if and only if $Y_j$ is slc for each $j$, or equivalently if and only if $Y$ is slc.
\end{proof}

\begin{corollary} \label{cor:1-pure-gm-to-galois-topos}
Let $\E$ be a locally connected topos and let $\gamma : \E \to \SLC(\E)$ be the natural geometric morphism. Then $\gamma$ is $1$-pure.
\end{corollary}
\begin{proof}
We already saw that $\gamma$ is $0$-pure. It remains to show that for each $X$ in $\SLC(\E)$ and group $G$, the induced map $\HH^1(\gamma/X,G)$ is an isomorphism. By Proposition \ref{prop:SLC-and-slices}, we can interpret $\gamma/X$ as the natural geometric morphism $\E/\gamma^*X \to \SLC(\E/\gamma^*X)$, so it is enough to consider the case $X=1$. In other words, we need to prove that
\begin{equation*}
\HH^1(\rho, G): \HH^1(\E, G) \to \HH^1(\SLC(\E), G)
\end{equation*}
is an isomorphism. The sets $\HH^1(\E,G)$ and $\HH^1(\SLC(\E),G)$ have an interpretation as the sets of isomorphism classes of $G$-torsors in $\E$ resp.\ $\SLC(\E)$, see \cite[Theorem 9.8]{jardine-book}. Because $G$-torsors are locally connected, we get an equivalence of categories between the $G$-torsors in $\E$ and the $G$-torsors in $\SLC(\E)$, in particular the map $\HH^1(\rho, G)$ is an isomorphism, see also \cite[\S 3.4]{moerdijk-prodiscrete-galois}.
\end{proof}

\subsection{Relation to fundamental groups}

We will now study how the results above translate to fundamental groups. Take again $\E$ locally connected, with $\gamma : \E \to \SLC(\E)$ the natural geometric morphism. Because $\gamma$ is connected, $\E$ is connected if and only if $\SLC(\E)$ is. In this case, each point $p$ of $\E$ naturally determines a prodiscrete localic group $G$ and an equivalence
\begin{equation*}
\SLC(\E) \simeq \Cont(G), 
\end{equation*}
with $p^*\gamma^*$ corresponding to the forgetful functor $\Cont(G) \to \Set$, see \cite[Theorem 3.2]{moerdijk-prodiscrete-galois}. This localic group $G$ is called the \textbf{fundamental group} of $\E$ in $p$ and denoted by $\pi_1(\E,p)$, or $\pi_1(\E)$ if the point is left implicit.

\begin{remark}
If $\E$ is connected locally connected, and $p$ and $q$ are two different points of $\E$, then $\pi_1(\E,p)$ and $\pi_1(\E,q)$ are isomorphic (this follows from e.g.\ \cite[Theorem 2.6]{moerdijk-prodiscrete-galois}). If $\E$ does not have any points, then it is possible that $\SLC(\E)$ does not have any points either, and in that case $\SLC(\E)$ cannot be equivalent to $\Cont(G)$ for $G$ a localic group, because the latter always has a point corresponding to the forgetful functor.
\end{remark}

\begin{proposition} \label{prop:h1-fundamental-group}
Let $\E$ be a pointed, connected, locally connected topos, and let $\pi_1(\E)$ be its (localic) fundamental group. Then for a (discrete) group $G$, we have
\begin{equation*}
\HH^1(\E,G) ~\simeq~ \HH^1(\Cont(\pi_1(E)),G) ~\simeq~ \Hom(\pi_1(\E),G)/G
\end{equation*}
where the right $G$-action on $\Hom(\pi_1(G),G)$ is given by conjugation, in other words $(\phi \cdot g)(x) = g^{-1}\phi(x)g$.
\end{proposition}
\begin{proof}
This is the conclusion of \cite[\S 3.4]{moerdijk-prodiscrete-galois}, taking into account the interpretation of $\HH^1(-,G)$ in terms of $G$-torsors \cite[Theorem 9.8]{jardine-book}.
\end{proof}

Let $\rho : \F\to \E$ be a $0$-pure geometric morphism of pointed, connected and locally connected toposes. A natural question is whether the induced map on fundamental groups
\begin{equation*}
\pi_1(\F) \to \pi_1(\E)
\end{equation*}
is surjective. This turns out not to be the case. For example, consider the connected geometric morphism $\rho : \Cont(\Z) \to \Cont(\widehat{\Z})$ induced by the inclusion $\Z \to \widehat{\Z}$, with $\Z$ the discrete group of integers and $\widehat{\Z}$ its profinite completion. The induced morphism on fundamental groups is again the inclusion $\Z \to \widehat{\Z}$, which is not surjective. However, we can still prove something weaker:

\begin{corollary}
Let $\rho: \F \to \E$ be a $0$-pure geometric morphism of pointed toposes, with $\F$ and $\E$ connected and locally connected. Then the induced map
\begin{equation*}
\pi_1(\rho) : \pi_1(\F) \to \pi_1(\E)
\end{equation*}
is an epimorphism of prodiscrete localic groups.
\end{corollary}
\begin{proof}
Let $\{N_i\}_{i \in I}$ be the filtered diagram of open normal subgroups of $\pi_1(\F)$. Then we can write $\pi_1(\F)$ as a cofiltered limit of discrete groups
\begin{equation*}
\pi_1(\F) \simeq \varprojlim_{i \in I} \pi_1(\F)/N_i,
\end{equation*}
where we take the limit as localic groups \cite{moerdijk-prodiscrete-galois}. Here the quotients $\pi_1(\F)/N_i$ can be interpreted as the automorphism groups of the Galois objects in $\F$, see the proof of \cite[Theorem 3.2]{moerdijk-prodiscrete-galois}. Because $\rho$ is $0$-pure, the Galois objects of $\E$ form a full subcategory of the Galois objects of $\F$. As a result, we find
\begin{equation*}
\pi_1(\E) \simeq \varprojlim_{i \in I'} \pi_1(\F)/N_i,
\end{equation*}
for a cofiltered full subcategory $I'$ of $I$.

Because the transition maps $N_j \to N_i$ are inclusions, the dual maps $$\pi_1(\F)/N_i \to \pi_1(\F)/N_j$$ are surjections. Moreover, the map $\pi_1(\F) \to \pi_1(\F)/N_i$ is an epimorphism of locales, see \cite[Theorem 5.1]{Mo}, so in particular it is an epimorphism of localic groups.

Now via a standard diagram chasing argument it follows that $\pi_1(\F) \to \pi_1(\E)$ is an epimorphism of localic groups.
\end{proof}

\subsection{Monomorphism in the next degree}

Proposition \ref{prop:mono-in-the-next-galois} encourages us to prove the following proposition.

The cohomology groups $\HH^{n+1}$ in the statement below are the cohomology groups $\HH^{n+1}_\Set$ with respect to $\Set$. But using Proposition \ref{prop:from-Set-to-T} we then see that the statement also holds if we calculate the cohomology groups with respect to another base topos, for example with respect to $\T$.

\begin{proposition} \label{prop:mono-in-the-next}
If a geometric morphism $\rho: \F \to \E$ is $n$-pure over the base topos $\T$, then the map
\begin{equation*}
\HH^{n+1}(\rho) : \HH^{n+1}(\E, e^*A) \longrightarrow \HH^{n+1}(\F, f^*A)
\end{equation*}
is a monomorphism, for each Eilenberg--Mac Lane object $K(A, n+1)$ in $\T$.
\end{proposition}
\begin{proof}
For $n=-1$, this holds by definition.

For $n=0$, we use that $A$ is a group and $\HH^1_\Set(\E, e^*A)$ classifies $A$-torsors \cite[Theorem 9.8]{jardine-book}. Because $A$-torsors are locally isomorphic to $A$, they are in particular locally constant over $\T$ (here `constant' means that it is in the essential image of $e^*$). Because $\rho$ is $0$-pure over $\T$, we find that $\rho^*$ is fully faithful on constant objects over $\T$. Indeed, the map
\begin{equation*}
\Hom_\E(e^*X, e^*Y) \to \Hom_\F(f^*X, f^*Y)
\end{equation*}
can be rewritten as
\begin{equation*}
\Hom_\E(e^*X, e^*Y) \to \Hom_\F(e^*X, \rho_*\rho^*e^*Y)
\end{equation*}
which is an isomorphism because $e^*Y \to \rho_*\rho^*e^*Y$ is an isomorphism. But then $\rho^*$ is also fully faithful on locally constant objects, in particular on $A$-torsors. We conclude that $\HH^1(\rho)$ is injective.

For $n \geq 1$, $A$ is abelian. In this case, we consider the Grothendieck spectral sequence
\begin{equation*}
\HH^p(\E,R^q\rho_*(f^*A)) \Rightarrow \HH^{p+q}(\F,e^*A)\;.
\end{equation*}
The map $\HH^{n+1}(\rho,A)$ factors as
\begin{equation*}
\HH^{n+1}(\E,A) \twoheadrightarrow E^{n+1,0}_\infty \hookrightarrow \HH^{n+1}(\F,A)
\end{equation*}
with $\HH^{n+1}(\E,A) = E^{n+1,0}_2 \twoheadrightarrow E^{n+1,0}_\infty$ 
the quotient by the images of the boundary maps. 
So we have to show that the relevant boundary maps are zero here. 
They have the form
\begin{equation*}
d_k : E^{n-k+1,k-1}_k \to E^{n+1,0}_k\;.
\end{equation*}
For $2 \leq k \leq n+1$, by Theorem \ref{thm:characterization-derived-pushforwards} we have $R^{k-1}\rho_*(f^*A) = 0\,$, so $E_2^{n-k+1,k-1} = 0\,$. 
Because $E_k^{n-k+1,k-1}$ is a subquotient of $E_2^{n-k+1,k-1}$
we find that $E_k^{n-k+1,k-1} = 0$ as well. 
For $k > n+1\,$, we also have $E_k^{n-k+1,k-1} = 0$ because the spectral sequence is concentrated in the first quadrant. 
So all relevant boundary maps are zero, and as a result the natural map $\HH^{n+1}(\rho)$ is a monomorphism.
\end{proof}

\section{The smallest \texorpdfstring{$n$-pure}{n-pure} subtopos and dimension theory}

Throughout this section, we fix a Grothendieck base topos $\T$.

\subsection{The \texorpdfstring{$n$-pure}{n-pure} topology} 

Let $\E$ be a Grothendieck topos over $\T$. For a morphism $j: V \to U$ in $\E$, we say that $j$ is $n$-pure if the induced geometric morphism $\E/V \to \E/U$ is $n$-pure. We claim that the $n$-pure monomorphisms form a topology on $\E$. To see this, we apply Proposition \ref{prop:characterization-topologies}. Condition $(i)$ follows by combining Proposition \ref{prop:coproduct-pure} and Corollary \ref{cor:pure-reflection}. Condition $(ii)$ and $(iii)$ follow from Proposition \ref{prop:basic-properties-A-pure}.

\begin{definition}
Let $\E$ be a Grothendieck topos over $\T$.
The \emph{$n$-pure topology} on $\E$ is the topology given by the monomorphisms that are $n$-pure over $\T$. 
We write $\E_{\leq n}$ for the subtopos of $\E$ determined by the $(n-1)$-pure topology.
\end{definition}

If $m \leq n$, then each $n$-pure geometric morphism is $m$-pure, 
so we get a sequence of subtoposes
\begin{equation*}
\varnothing = \E_{\leq -1} \subseteq \E_{\leq 0} \subseteq \E_{\leq 1} \subseteq \dots \subseteq \E_{\leq n} \subseteq \dots
\end{equation*}

The subtoposes of $\E$ form a complete lattice, so we can define the union
\begin{equation*}
\E_{<\infty} = \bigcup_{n=0}^\infty \E_{\leq n}.
\end{equation*}
The subtopos $\E_{<\infty} \subseteq \E$ will be called the \textbf{content} of $\E$. The corresponding topology has as dense monomorphisms the $\infty$-pure monomorphisms, where we define a morphism to be $\infty$-pure if it is $n$-pure for each $n$.

\subsection{The smallest \texorpdfstring{$n$-pure}{n-pure} subtopos}

\begin{theorem} \label{thm:n-pure-topology-vs-n-pure-inclusion}
Let $e: \E \to \T$ be a Grothendieck topos over $\T$. Let $J$ be a (Lawvere--Tierney) topology on $\E$ and let $\E' \subseteq \E$ be the corresponding subtopos. For a fixed $n \geq 0$, the following are equivalent:
\begin{enumerate}
\item every monomorphism in $J$ is $n$-pure over $\T$;
\item the Eilenberg--Mac Lane objects of the form $K(e^*A,m)$ with $m \leq n$ satisfy descent with respect to the \v{C}ech localization associated to $\E' \subseteq \E$;
\item the Eilenberg--Mac Lane objects of the form $K(e^*A,m)$ with $m \leq n$ satisfy descent with respect to the localization $\s\E' \subseteq \s\E$;
\item the inclusion $\s\E' \subseteq \s\E$ is $n$-pure over $\T$.
\end{enumerate}
\end{theorem}
\begin{proof}
$\underline{(1) \Leftrightarrow (2)}$~ Using Proposition \ref{prop:descent-for-sheaves} and the discussion in Subsection \ref{ssec:left-exact-localization} and Subsection \ref{ssec:cech-localization},  we see that $K(e^*A,m)$ satisfies descent with respect to this \v{C}ech localization if and only if for each monomorphism $V \hookrightarrow U$ in $J$, the map
\begin{equation} \label{eq:cech-descent}
\HOM(U, K(e^*A,m)) \longrightarrow \HOM(V, K(e^*A,m))
\end{equation}
is a weak equivalence. By Proposition \ref{prop:hom-to-n-type}, the simplicial sets $\HOM(U,K(e^*A,m))$ and $\HOM(V,K(e^*A,m))$ are $m$-truncated, i.e.\ their homotopy groups vanish in degrees $>m$. In lower degrees, we can calculate
\begin{equation*}
\begin{split}
\pi_k(\HOM(U, K(e^*A,m)))
&\simeq \pi_k(\s e_{U,*}K(e_U^*A,m)) \\
&\simeq \pi_0(\s e_{U,*}\Omega^k K(e_U^*A,m)) \\
&\simeq \pi_0(\s e_{U,*}K(e_U^*A,m-k))  \\
&\simeq \HH^{m-k}(\E/U, e_U^*A)
\end{split}
\end{equation*}
with $e_U : \E/U \to \T$ the global sections geometric morphism. To construct the above isomorphisms, we use Proposition \ref{prop:looping-commutes-with-pushforward}, Proposition \ref{prop:homotopy-groups-of-loop-space} and Corollary \ref{cor:loop-space-of-em}. Using a similar notation and calculation for $V$ rather than $U$, we see that \eqref{eq:cech-descent} is a weak equivalence if and only if the maps
\begin{equation*}
\HH^{m-k}(\E/U, e_U^*A) \longrightarrow \HH^{m-k}(\E/V, e_V^*A)
\end{equation*}
are isomorphisms, for $0 \leq k \leq m$. This is the case for all $K(e^*A,m)$ if and only if each $V \hookrightarrow U$ in $J$ is $n$-pure (we take into account here that pullbacks of monomorphisms in $J$ are again in $J$).

$\underline{(2) \Leftrightarrow (3)}$~ We use the notation $\check{\s}\E'$ to denote the \v{C}ech localization of $\s\E$ along $\E' \subseteq \E$. Using Theorem \ref{thm:hypercompletion-of-cech}, we can write $\s\E'$ as the hypercompletion of $\check{\s}\E'$. Lemma \ref{lmm:descent-transitive} then says that an object satisfies descent with respect to $\s\E' \subseteq \s\E$ if and only if the following two conditions hold: it should satisfy descent with respect to the \v{C}ech localization $\check{\s}\E' \subseteq \s\E$, and, seen as an object of $\check{\s}\E'$, it should satisfy descent along the hypercompletion. It is enough to show that this second condition automatically holds for Eilenberg--Mac Lane objects. Note that an Eilenberg--Mac Lane object $X$ in $\s\E$ remains $k$-truncated (for the appropriate $k$) when it is interpreted as an object of the \v{C}ech localization $\check{\s}\E'$, by \cite[Proposition 7.6]{rezk}. We can again apply \cite[Proposition 7.6]{rezk} to see that $X\times A \to A$ is $k$-truncated in $\check{\s}\E'/A$. For an $\infty$-connected map $B\to A$ in $\check{\s}\E'$, we then find that the map
\begin{equation*}
R\HOM_{\check{\s}\E'/A}(A,X\times A) \longrightarrow R\HOM_{\check{\s}\E'/A}(B,X\times A)
\end{equation*}
is a weak equivalence (Definition \ref{def:n-connected}). We can rewrite this map as
\begin{equation*}
R\HOM_{\check{\s}\E'}(A,X) \longrightarrow R\HOM_{\check{\s}\E'}(B,X).
\end{equation*}
We conclude that $X$ satisfies descent with respect to the $\infty$-connected maps, see Proposition \ref{prop:descent-for-sheaves}, in other words it satisfies descent with respect to the hypercompletion.

$\underline{(3) \Leftrightarrow (4)}$~ This follows from Definition \ref{def:descent} and Corollary \ref{cor:preserves-em-objects}. Note that a map of the form
\begin{equation*}
K(e^*A,m) \longrightarrow \s i_*\left((\s i^*K(e^*A,m))^f\right),
\end{equation*}
with $\s i: \s\E' \to \s\E$ the inclusion geometric morphism, is a weak equivalence if and only if it is a homotopy equivalence, because both objects are fibrant (and cofibrant).
\end{proof}

\begin{corollary}
Let $e: \E \to \T$ be a Grothendieck topos over $\T$.
Then $\E_{\leq n} \subseteq \E$ is the smallest $(n-1)$-pure subtopos.
\end{corollary}
\begin{proof}
Using the equivalence $(1) \Leftrightarrow (4)$ in Theorem \ref{thm:n-pure-topology-vs-n-pure-inclusion}, a subtopos is $n$-pure if and only if the corresponding topology is contained in the $n$-pure topology. So the subtopos corresponding to the $n$-pure topology is the smallest $n$-pure topos.

Note that in Theorem \ref{thm:n-pure-topology-vs-n-pure-inclusion} we only consider the case $n \geq 0$, but the equivalence $(1) \Leftrightarrow (4)$ holds for $n=-1$ as well.
\end{proof}

For $n=-1$ and $\T=\Set$, we recover the classical result that the subtopos corresponding to the dense topology is the smallest dense subtopos.

\subsection{Definition of dimension}

We are now ready to introduce our concept of dimension for Grothendieck toposes.

\begin{definition}
Let $\E$ be a Grothendieck topos over $\T$. The \emph{dimension of $\E$ over $\T$} is the smallest integer $n \geq -1$ 
such that $\E_{\leq n} = \E_{<\infty}$.
In other words, the dimension of $\E$ is the smallest number $n$ such that every $(n-1)$-pure monomorphism over $\T$ in $\E$ is $k$-pure over $\T$ for all $k$.
If such an integer $n$ does not exist, then we say that $\E$ has dimension $\infty\,$.
We write $\dim(\E)$ for the dimension of $\E$.
\end{definition}

If $\E_{<\infty} = \E$, then we say that $\E$ is \emph{without boundary}. On the other hand, if $\E_{<\infty} \neq \E$ then we say that $\E$ has a boundary. Note that unless the subtopos $\E_{<\infty}$ has a complement, it does not make sense to view the boundary of $\E$ as a subtopos in its own right.

\subsection{Basic properties of dimensions}

We shall say that a geometric morphism $\rho : \F \to \E$ is \emph{$n$-skeletal} 
if it restricts to a geometric morphism $\F_{\leq n} \to \E_{\leq n}$. 
Equivalently, $\rho$ is $n$-skeletal if and only if $\rho^*(\phi)$ is $(n-1)$-pure 
for every $(n-1)$-pure morphism  $\phi$ of $\E$.

Let $j : \E_{\leq n} \to \E$ be the inclusion of the smallest $(n-1)$-pure subtopos. 
For an object $X$ of $\E$ consider the following pullback diagram.
\begin{equation} \label{eq:n-minimal-subtopos-of-slice}
\begin{tikzcd}
\E_{\leq n}/j^*(X) \ar[r,"{\tilde{j}}"] \ar[d] & \E/X \ar[d] \\
\E_{\leq n} \ar[r,"{j}"] & \E
\end{tikzcd}
\end{equation}
The pullback $\tilde{j}$ is then an inclusion, 
corresponding to the topology on $\E/X$ generated by the $n$-pure monomorphisms. 
It follows that 
\begin{equation*}
(\E/X)_{\leq n} ~\simeq~ \E_{\leq n}/j^*(X)\;.
\end{equation*}
In particular:
\begin{proposition}
\'Etale geometric morphisms are $n$-skeletal, for every $n \geq 0\,$.
\end{proposition}

Further, if the right vertical map in \eqref{eq:n-minimal-subtopos-of-slice} is a surjection, 
then the same holds for the left vertical map because open surjections are stable under base change. 
From this we can conclude that $\E_{\leq n} = \bigcup_{i \in I} (\E_i)_{\leq n}$ whenever $\E$ is covered by a collection of open subtoposes $(\E_i)_{i \in I}\,$. 

A geometric morphism $\rho : \F \to \E$ is $n$-skeletal if and only if 
$\rho(\F_{\leq n}) \subseteq \E_{\leq n}\,$, 
where $\rho(\F_{\leq n})$ denotes the image of the subtopos $\F_{\leq n}$ along $\rho$.
The $(n-1)$-pure geometric morphisms satisfy the reverse inclusion.

\begin{proposition} \label{prop:n-pure-converse-skeletal}
If $\rho: \F \to \E$ is $(n-1)$-pure, then $\E_{\leq n} \subseteq \rho(\F_{\leq n})$. In other words, the image of an $(n-1)$-pure geometric morphism is again $(n-1)$-pure.
\end{proposition}
\begin{proof}
Let $\phi:Y \hookrightarrow X$ be a monomorphism of $\E$.
We need to show that if $\rho^*\phi$ is $(n-1)$-pure, then $\phi$ is $(n-1)$-pure.
We have a pullback diagram
\begin{equation*}
\begin{tikzcd}
\F/\rho^*Y \ar[r,"{\rho/Y}"] \ar[d,"\rho^*\phi"'] & \E/Y \ar[d,"{\phi}"] \\
\F/\rho^*X \ar[r,"{\rho/X}"'] & \E/X
\end{tikzcd}
\end{equation*}
where $\rho/X$ and $\rho/Y$ are $(n-1)$-pure.
By Proposition \ref{prop:basic-properties-A-pure}(ii), if $\rho^*\phi$ is $(n-1)$-pure, then $\phi$ is $(n-1)$-pure as well.
\end{proof}

The following is a generalization of the standard fact that dense inclusions are skeletal.

\begin{proposition} \label{prop:n-pure-inclusions-and-skeletal}
The $(n-1)$-pure inclusions are $n$-skeletal, for every $n \geq 0$.
\end{proposition}
\begin{proof}
Let $j : \F \to \E$ be an $(n-1)$-pure inclusion. To show that $j$ is $n$-skeletal, it is enough to show that for every $(n-1)$-pure monomorphism $b : X' \to X$ in $\E$, the pullback $j^*(b)$ is again $(n-1)$-pure. 
Take $Y \to j^*X$ an object in $\F/j^*X$, and set $Y' = Y \times_{j^*X} j^*X'$. We claim that the map $\HH^i(\F/Y,A) \to \HH^i(\F/Y',A)$ is an isomorphism for all $K(A,i)$ with $i \leq n-1$. Let $Z \to X$ be the pullback of $Y \to j^*X$ along the unit $X \to j_*j^*X$, and set $Z' = Z \times_X X'$. Then we get a natural map $Z' \to Z$ in $\E$ such that after applying $j^*$ we get the map $Y' \to Y$. As a result, we find a commutative diagram
\begin{equation*}
\begin{tikzcd}
\HH^i(\F/Y,A) \ar[r] \ar[d] & \HH^i(\E/Z,A) \ar[d] \\
\HH^i(\F/Y',A) \ar[r] & \HH^i(\E/Z',A).
\end{tikzcd}
\end{equation*}
such that the upper, lower and right maps are isomorphisms, for each $K(A,i)$ with $i \leq n-1$. But then the left map is an isomorphism as well, which shows that $j^*(b)$ is $(n-1)$-pure.
\end{proof}

\begin{corollary} \label{cor:same-smallest-n-pure-subtopos}
Let $j : \F\hookrightarrow \E$ be an $(n-1)$-pure subtopos. 
Then $\F_{\leq n} = \E_{\leq n}$.
\begin{proof}
This follows by combining Proposition \ref{prop:n-pure-converse-skeletal} 
and Proposition \ref{prop:n-pure-inclusions-and-skeletal}.
\end{proof}
\end{corollary}

\begin{corollary}
For $i,j$ natural numbers, we have $(\E_{\leq i})_{\leq j} = \E_{\leq \mathrm{min}(i,j)}$. In the limit, we have $(\E_{\leq i})_{< \infty} = \E_{\leq i}$, $(\E_{< \infty})_{\leq j} = \E_{\leq j}$ and $(\E_{< \infty})_{<\infty} = \E_{< \infty}$.
\end{corollary}
\begin{proof}
We first consider the case where $i=j$ and we write $n=i=j$. The inclusion $(\E_{\leq n})_{\leq n} \subseteq \E$ is $(n-1)$-pure as a composition of two $(n-1)$-pure inclusions. It is also contained in $\E_{\leq n}$. Because $\E_{\leq n}$ is the smallest $(n-1)$-pure subtopos of $\E$, we find $(\E_{\leq n})_{\leq n} = \E_{\leq n}$. In the same way we can show that $(\E_{<\infty})_{<\infty} = \E_{<\infty}$.

As the second case, we consider $i < j$. We have inclusions $(\E_{\leq i})_{\leq i} \subseteq (\E_{\leq i})_{\leq j} \subseteq \E_{\leq i}$, which shows $(\E_{\leq i})_{\leq j} = \E_{\leq i}$. Similarly, we find $(\E_{\leq i})_{< \infty} = \E_{\leq i}$.

Finally, take $i > j$. In this case, $\E_{\leq i} \subseteq \E$ is a $(j-1)$-pure subtopos, so by Corollary \ref{cor:same-smallest-n-pure-subtopos} we get $(\E_{\leq i})_{\leq j} = \E_{\leq j}$. In the same way we get $(\E_{< \infty})_{\leq j} = \E_{\leq j}$.
\end{proof}

\begin{proposition} \label{prop:dimension-locally}
Let $\phi_i : \E_i \to \E\,$, $i \in I\,$, be a jointly surjective family of \'etale geometric morphisms. 
Then $\dim(\E) = \sup_{i \in I} \dim(\E_i)\,$.
\end{proposition}
\begin{proof}
The associated geometric morphism $\phi: \E' = \bigsqcup_{i \in I} \E_i \to \E$ is an \'etale surjection. In $\bigsqcup_{i \in I} \E_i$ a monomorphism is $n$-pure if and only if each of its components is $n$-pure, see Proposition \ref{prop:coproduct-pure}. It follows that $\dim(\E') = \sup_{i \in I} \dim(\E_i)$. Further, there is a pullback diagram
\begin{equation}
\begin{tikzcd}
\E'_{\leq n} \ar[r,"{j'}"] \ar[d,"{\xi}"'] & \E'_{< \infty} \ar[d,"{\psi}"] \\
\E_{\leq n} \ar[r,"{j}"] & \E_{<\infty}
\end{tikzcd}
\end{equation}
as in \eqref{eq:n-minimal-subtopos-of-slice}. Because $\psi$ is surjective, so is $\xi$. If $j$ is an equivalence, then so is the pullback $j'$. Conversely, if $j'$ is an equivalence, then so is $j$ by uniqueness of the (surjection, inclusion) factorization. We conclude that
\begin{equation*}
\dim(\E) = \dim(\E') = \sup_{i \in I} \dim(\E_i).
\end{equation*}
\end{proof}

\section{Calculation techniques}

Again we work over a fixed (Grothendieck) base topos $\T$.

\subsection{Puncture types}

Let $\F$ and $\F'$ be Grothendieck toposes over $\T$. We say that an open inclusion $j: \F' \hookrightarrow \F$ is \emph{nontrivial} if it is not an equivalence.

\begin{definition}
A \emph{puncture type} is an equivalence class of nontrivial open inclusions, where two open inclusions $j$ and $j'$ are equivalent if and only if there exist topos equivalences $\phi$ and $\psi$ such that $j'\phi \cong \psi j$.

The \emph{dimension} $\dim(j)$ of a puncture type $j$ is the smallest number $n$ such that $j$ is $(n-2)$-pure but not $(n-1)$-pure, or $\infty$ if such an $n$ does not exist.
\end{definition}

The dimension of a puncture type is well-defined: if $j$ and $j'$ are two equivalent open inclusions then $j$ is $n$-pure if and only if $j'$ is $n$-pure.

If a puncture type $j$ can be written as $\E/V \hookrightarrow \E/U$ for some monomorphism $V \hookrightarrow U$ in $\E$, up to equivalence, then we say that $j$ is a \emph{puncture type in} $\E$.

Now suppose that $\E$ contains a puncture type of dimension $n$. Then by definition the $(n-1)$-pure topology and the $(n-2)$-pure topology are different, in other words $\E_{\leq n-1} \neq \E_{\leq n}$. We conclude that $\dim(\E) \geq n$.

Similarly, there is a puncture type of dimension $\infty$ in $\E$ if and only if $\E_{<\infty} \neq \E$, i.e.\ if and only if $\E$ has a boundary.

\begin{definition}
Let $j$ and $j'$ be two puncture types. We say that $j$ is \emph{contained in} $j'$ precisely if there exist an open inclusion $\phi$ and a topos equivalence $\psi$ such that $j'\phi \cong \psi j$.

Let $\E$ be a Grothendieck topos. A family of puncture types $\A$ in $\E$ is called \emph{complete} if there is a family of generators $\U$ for $\E$ such that for every puncture type $j$ in $\E$ with codomain in $\U$ there is some puncture type $j'$ in $\A$ with $j$ contained in $j'$ and with $\dim(j)$ and $\dim(j')$ either both finite or both infinite.
\end{definition}

The idea behind this is that if $j$ is contained in $j'$ then $\dim(j) \leq \dim(j')$. This follows from the fact that a subtopos is $n$-pure whenever it contains an $n$-pure subtopos. So a complete family of generators will lead to an upper bound on the dimensions of the puncture types.

\begin{proposition} \label{prop:dimension-is-dimension-of-puncture-types}
Let $\E$ be a Grothendieck topos over $\T$ and let $\A$ be a complete family of puncture types for $\E$.
Then $\dim(\E) = \sup_{j \in \A'} \dim(j)$, with $\A' \subseteq \A$ the subfamily of puncture types of finite dimension. Moreover, $\E$ has a boundary if and only if $\A'\neq \A$.
\end{proposition}
\begin{proof}
The dimension of $\E$ is the largest $n$ for which $\E_{\leq n} \neq \E_{\leq n-1}$, or $\infty$ if this $n$ does not exist. In other words, the dimension of $\E$ is the largest $n$ such that there exists an $n$-puncture in $\E$, or $\infty$ if there is no upper bound on this number $n$. The statement follows, taking into account that for any puncture type of finite dimension in $\E$ we can find a puncture type in $\A'$ with at least this dimension.

If $\A' \neq \A$ then $\E$ contains a puncture type of dimension $\infty$, so $\E$ has a boundary. Conversely, if $\E$ has a boundary, then it contains a puncture type $j$ of dimension $\infty$. We can find a puncture type $j'$ in $\A$ with $j$ contained in $j'$. Then $j'$ is again of dimension $\infty$ so $j' \in \A - \A'$.
\end{proof}

Let $\E$ be a Grothendieck topos and let $\U$ be a family of generators for $\E$. To check that a puncture type $j: \E/V \to \E/U$ is $n$-pure, we first construct for each map $U' \to U$ with $U'$ in $\U$ a pullback diagram of the form:
\begin{equation*}
\begin{tikzcd}
\E/V' \ar[r,hook,"{j'}"] \ar[d] & \E/U' \ar[d] \\
\E/V \ar[r,hook,"{j}"'] & \E/U
\end{tikzcd}.
\end{equation*}
Now $j$ is $n$-pure if and only if for each choice of $U' \to U$ the pullback $j'$ induces isomorphisms in cohomology
\begin{equation*}
\HH^i(\E/U',A) \longrightarrow \HH^i(\E/V',A)
\end{equation*}
for all Eilenberg--Mac Lane objects $K(A,i)$ in $\T$ with $i \leq n$. We say that $j$ is \emph{stable} (with respect to $\U$) if each $j'$ is either an equivalence or equivalent to $j$. This simplifies the procedure of checking whether $j$ is $n$-pure or not:

\begin{proposition} \label{prop:stable-puncture-types}
Suppose that a puncture type $j$ in $\E$ is stable (with respect to some family of generators). Take $n\geq 0$. Then $j$ is $n$-pure if and only if it the induced maps in cohomology $\HH^i(j,A)$ are isomorphisms, for all Eilenberg--Mac Lane objects $K(A,i)$ in $\T$ with $i \leq n$.

Further, a stable puncture type $j$ is $(-1)$-pure if and only the induced maps in cohomology $\HH^0(j,A)$ are monomorphisms for each $A$ in $\T$.
\end{proposition}

\subsection{Sheaf cohomology for topological spaces}

We restrict to the case $\T = \Set$.

\begin{proposition} \label{prop:sheaf-cohomology-topological-spaces}
Let $X$ be a topological space. Then:
\begin{enumerate}
\item for $A$ a set, $\HH^0(\Sh(X),A)$ agrees with the set of continuous maps $X \to A$, where $A$ has the discrete topology;
\item for $A$ a group, $\HH^1(\Sh(X),A)$ agrees with $\Hom(\pi_1(X), A)/A$, assuming that $X$ is locally path-connected, path-connected and semilocally simply connected.
\item for $i \geq 0$ and $A$ an abelian group, $\HH^i(\Sh(X), A)$ agrees with the singular cohomology $\HH^i_{\mathrm{sing}}(\Sh(X),A)$, assuming that $X$ is locally contractible.
\end{enumerate}
\end{proposition}
\begin{proof}
\begin{enumerate}
\item This holds by definition.
\item For $X$ a locally path-connected, path-connected and semilocally simply connected space, the category of covering spaces over $X$ is equivalent to the category of $\pi_1(X)$-sets \cite[Theorem 2.1]{nlab:fundamental_theorem_of_covering_spaces}, so the fundamental group of the topos $\Sh(X)$ agrees with the fundamental group of $X$. Now we can use Proposition \ref{prop:h1-fundamental-group}.
\item This is proved in \cite{Sella}.
\end{enumerate}
\end{proof}

\subsection{Example: locally Euclidean topological spaces}

We restrict to the case $\T = \Set$.

\begin{definition}
A topological space $X$ is called \emph{locally Euclidean} 
if there is a natural number $n$ and an open covering 
$X = \bigcup_{i \in I} U_i$ with $U_i \cong \R^n$ for every $i \in I\,$. 
The natural number $n$ is then called the \emph{dimension} of $X\,$. 
\end{definition}

\begin{proposition} \label{prop:example-locally-euclidean}
Let $X$ be a locally Euclidean topological space of dimension $n\,$. 
Then $\dim(\Sh(X))=n\,$.
\end{proposition}
\begin{proof}
We can take a family of generators $\U$ for $\Sh(\R^n)$ corresponding to a basis of open subsets in $\R^n$ with each open subset itself homeomorphic to $\R^n$. If we remove a single point from such an open set, then what remains is a space homeomorphic to $\R^n-\{0\}$, with $0$ denoting the origin.
We conclude that the singleton consisting of the puncture type
\begin{equation}\label{eq:j}
j : \Sh(\R^n-\{0\}) \to \Sh(\R^n)
\end{equation}
is a complete family of puncture types for $\Sh(\R^n)$. It remains to show that $j$ is $(n-2)$-pure but not $(n-1)$-pure.

For $n=0$, $j$ is of the form $\Sh(\varnothing) \to \Sh(\{\ast\})$, which is $(-2)$-pure but not $(-1)$-pure; remember that every geometric morphism is $(n-2)$-pure by convention, and that a geometric morphism is $(-1)$-pure precisely if it is dense.

For $n\geq 1$, we calculate the cohomology groups explicitly. Because $j$ is a stable puncture type we can use Proposition \ref{prop:stable-puncture-types}, which simplifies the situation. We can compute the sheaf cohomology groups using Proposition \ref{prop:sheaf-cohomology-topological-spaces}.

For $n=1$, the inclusion $\R-\{0\} \hookrightarrow \R$ is dense but in zeroth cohomology we get the diagonal morphism $A \hookrightarrow A \times A$ which fails to be an isomorphism as soon as $A$ contains at least two elements.

For $n=2$, the inclusion $\R^2-\{0\} \hookrightarrow \R^2$ induces an isomorphism in zeroth cohomology but not in $\HH^1$, for example applying $\HH^1(-,\Z)$ gives the group homomorphism $0 \hookrightarrow \Z$ (using that $\pi_1(\R^2-\{0\}) = \Z$).

Finally, for $n\geq 3$, the inclusion $\R^n-\{0\} \hookrightarrow \R^n$ induces an isomorphism in degrees $i \leq n-2$, but again in $(n-1)$th cohomology this breaks down, applying $\HH^{n-1}(-,\Z)$ gives the inclusion $0 \hookrightarrow \Z$.
\end{proof}

\subsection{Example: manifolds with boundary}

Topological manifolds are in particular locally Euclidean. So for $M$ a topological manifold, we have $\dim(\Sh(M)) = \dim(M)$ using Proposition \ref{prop:example-locally-euclidean}. We will now generalize this to manifolds with boundary. Recall that for a topos $\E$, the subtopos $\E_{<\infty}$ is called the \emph{content} of $\E$.

\begin{proposition}
Let $M$ be a topological manifold with nonempty boundary $\partial M\,$. 
We write $M^\circ = M - \partial M$ for the interior. 
Then $\Sh(M)$ is an $n$-dimensional topos with content $\Sh(M^\circ)$ and boundary $\Sh(\partial M)$.
\end{proposition}
\begin{proof}
Note that $M$ has as complete family of puncture types the two inclusions
\begin{equation*}
\R^n-\{0\} ~\hookrightarrow~ \R^n
\end{equation*}
and
\begin{equation*}
\{ (x_1,\dots,x_n) \in \R^n ~:~ x_1 \geq 0 \} - \{0\} ~\hookrightarrow~ \{ (x_1,\dots,x_n) \in \R^n ~:~ x_1 \geq 0 \}.
\end{equation*}
Both puncture types are stable. We already saw in the proof of Proposition \ref{prop:example-locally-euclidean} that the first is a puncture type of dimension $n$. In the second puncture type, both domain and codomain are contractible, so the second puncture type is infinite-dimensional. Now using Proposition \ref{prop:dimension-is-dimension-of-puncture-types} we conclude that $\Sh(M)$ is $n$-dimensional. Further, the existence of an infinite-dimensional puncture type means that $\Sh(M)$ has a boundary.

Similarly, the puncture type
\begin{equation*}
\{ (x_1,\dots,x_n) \in \R^n ~:~ x_1 > 0 \} ~\hookrightarrow~ \{ (x_1,\dots,x_n) \in \R^n ~:~ x_1 \geq 0 \}
\end{equation*}
is also stable and infinite-dimensional. This observations allows us to calculate the $\infty$-pure topology in $\Sh(M)$. Namely, an inclusion of open sets $V \subseteq U$ is $\infty$-pure if and only $V \cap M^\circ = U \cap M^\circ$. It follows that the content of $\Sh(M)$ is
\begin{equation*}
\Sh(M)_{< \infty} = \Sh(M^\circ).
\end{equation*}
This is an open subtopos, with as complement the closed subtopos $\Sh(\partial M)$. So $\Sh(\partial M)$ is the boundary of $\Sh(M)$.
\end{proof}

\subsection{Example: the rational line} 
Sierpi\'nski has shown that a topological space is homeomorphic to the space 
of rational numbers $\Q$ (with the Euclidean topology) if and only if it is 
countable, metrizable and has no isolated points. 
In particular, $\Q^n$ with the Euclidean topology and $\Q$ are homeomorphic spaces.
As a result, the dimension of $\Q^n$ is independent of $n$
for any notion of dimension that is invariant under homeomorphism.

For separable metrizable spaces like $\Q$, the small and large inductive dimension agree with the Lebesgue covering dimension, 
and for the rational line $\Q$ the dimension is zero. 
However, it turns out that $\Sh(\Q)$ is one-dimensional.

\begin{proposition}
The topos $\Sh(\Q)$ is one-dimensional.
\end{proposition}
\begin{proof}
There is a complete family of puncture types for $\Sh(\Q)$ consisting only of the puncture type $j: \Q - \{0\} \hookrightarrow \Q$. It is enough to show that this puncture type is one-dimensional (i.e. dense but not $0$-pure).

It is clear that $j$ is dense. To show that $j$ is not $0$-pure, consider the continuous map 
\[
f : \Q-\{0\} \to \{0,1\}
\]
defined by $f(z) = -1$ for $z<0$ and $f(z)=1$ for $z>0\,$. 
Then $f$ does not have an extension to the whole space $\Q\,$. 
In other words, $\HH^0(\Q,A) \to \HH^0(\Q-\{0\},A)$ is not surjective, for $A = 1 \sqcup 1\,$.
\end{proof}

\subsection{Boolean toposes are zero-dimensional} Again, we take as base topos $\T=\Set$. If $\E$ is a Boolean topos, then the smallest dense subtopos is the full topos, in symbols $\E_{\leq 0} = \E$. As a result, Boolean toposes are zero-dimensional without boundary.

Conversely, suppose that $\E$ is zero-dimensional without boundary, i.e.\ $\E_{\leq 0} = \E$. The smallest dense subtopos $\E_{\leq 0}$ is Boolean, so then $\E$ is Boolean as well.

We conclude:

\begin{proposition}
A Grothendieck topos is zero-dimensional without boundary (over $\Set$) if and only if it is Boolean.
\end{proposition}

\subsection{Totally connected toposes are zero-dimensional}

Let $\rho: \F\to \E$ be a geometric morphism over $\T$.
\begin{equation}\label{eq:rho-e-2}
\begin{tikzcd}
\F \ar[rr,"{\rho}"] \ar[dr,swap,"{f}"] && \E \ar[dl,"{e}"]\\
& \T
\end{tikzcd}
\end{equation}
We say that the geometric morphism $\rho$ is \emph{totally connected} \cite[p.~707]{EL} if it is locally connected and moreover the left adjoint $\rho_!$ preserves finite limits. If $\rho$ is totally connected, then there is a geometric morphism $d : \E \to \F$ defined by $d_* = \rho^*$ and $d^* = \rho_!$. Moreover, $d$ is a dense inclusion and $\rho d \cong 1$ \cite[p.~707-708]{EL}. After base change to the simplicial setting, we find that $\s \rho$ is again totally connected with $\s d$ as corresponding dense inclusion \cite[Lemma C3.6.18(iii)]{EL}.

Totally connected morphisms are connected \cite[Theorem C3.6.16]{EL}. We claim that, more generally, they are $n$-connected for each $n$.

\begin{proposition}\label{prop:totconn-are-nconn}
Totally connected geometric morphisms are $\infty$-connected (in particular $\infty$-pure).
\end{proposition}
\begin{proof}
Let $\rho$ be a totally connected geometric morphism with corresponding dense inclusion $d$. We need to show that the map
\begin{equation} \label{eq:totconn-are-nconn}
K(A,n) \longrightarrow \s\rho_*((\s\rho^*K(A,n))^f)
\end{equation}
is a homotopy equivalence, for each Eilenberg--Mac Lane object $K(A,n)$ in $\s\E$, see Corollary \ref{cor:preserves-em-objects}.

As discussed above, we have that $\s\rho_! \dashv \s\rho^*$ agrees with $\s d^* \dashv \s d_*$, in particular it is a Quillen adjunction. As a result, $\s\rho^*$ preserves fibrant objects. So we can rewrite the map \eqref{eq:totconn-are-nconn} as the unit map
\begin{equation*}
K(A,n) \longrightarrow \s\rho_*\s\rho^*K(A,n).
\end{equation*}
Because $\s\rho$ is connected, this unit map is an isomorphism for each object, in particular for Eilenberg--Mac Lane objects.
\end{proof}

\begin{corollary} \label{cor:dense-section-is-n-pure}
Let $\rho: \F \to \E$ be totally connected with corresponding dense inclusion $d$. Then $d$ is $\infty$-pure over $\E$ (in particular, $\infty$-pure over the base topos).
\end{corollary}
\begin{proof}
Totally connected geometric morphisms and their corresponding dense inclusions are stable under pullback \cite[Lemma C3.6.18(iii)]{EL}. So it is enough to show that
\begin{equation*}
\HH^n(d,A)~:~\HH^n(\F, A) \longrightarrow \HH^n(\E, A)
\end{equation*}
is an isomorphism for all Eilenberg--Mac Lane objects $K(A,n)$ in $\s\E$. This follows because $\HH^n(\rho,A)$ is an isomorphism for each $K(A,n)$ and $\HH^n(d,A)\HH^n(\rho,A) \cong 1$.
\end{proof}

\begin{corollary} \label{cor:totally-connected-toposes-infinite-dimensional}
Let $e : \E\to\T$ be a totally connected topos over $\T$, with $\T$ nonempty.
Then $\E$ is zero-dimensional over $\T$.
\end{corollary}
\begin{proof}
From Corollary \ref{cor:dense-section-is-n-pure} we know that the dense inclusion $d: \T \to \E$ corresponding to $e$ is $n$-pure for each $n$. By Corollary \ref{cor:same-smallest-n-pure-subtopos} we then have that $\E_{\leq n} = \T_{\leq n}$ for each $n$. Because $\T$ is the base topos, if $j: \T' \hookrightarrow \T$ is a subtopos where $j$ is $n$-pure, then $j$ is by definition $n$-connected. In particular, $j$ is $(-1)$-connected, i.e.\ surjective. This is only possible if $j$ is an isomorphism. In other words, $\T_{\leq n}=\T$ for each $n$. We conclude that $\T$ and $\E$ are both zero-dimensional (unless $\T$ is empty, because in that case $\E$ must be empty as well and empty toposes are $(-1)$-dimensional by convention).
\end{proof}

\subsection{Example: irreducible topological spaces}

If $X$ is an irreducible topological space, then $\Sh(X)$ is totally connected, see \cite[Examples C3.6.17(a)]{EL}. From the results above, we know that $\Sh(X)$ is then in particular zero-dimensional.

As a special case, $\Sh(X)$ is zero-dimensional for $X$ an irreducible scheme equipped with the Zariski topology.

\subsection{Example: right Ore monoids}

Recall that a monoid $M$ is called \emph{right Ore} if for all $x,y \in M$ there are $m,n \in M$ with $xm = yn$. The groupification of $M$ is the group you get by formally inverting all elements of $M$.

\begin{proposition} \label{prop:example-right-ore-monoids}
Let $M$ be a right Ore monoid and let $G$ be its groupification. Then $\PSh(M)$ is zero-dimensional with content $\PSh(G)$.
\end{proposition}
\begin{proof}
The monoid map $M \to G$ induces an essential geometric morphism
\begin{equation*}
f: \PSh(M) \longrightarrow \PSh(G)
\end{equation*}
with $f_!(X) \simeq X \otimes_M G$ \cite[Lemma 2.4]{hr1}.
If $M$ is right Ore, then $G$ is flat as a left $M$-set \cite[Example 2.4.4]{hemelaer-localization}, in other words $f_!$ preserves finite limits. Because $\PSh(G)$ is a Boolean \'etendue and $f$ is essential, we know that $f$ is locally connected \cite{hemelaer-eilc}. So $f$ is totally connected. Let $d: \PSh(G) \to \PSh(M)$ be the dense inclusion associated to $f$. Then $d$ is $\infty$-pure by Corollary \ref{cor:dense-section-is-n-pure} and then using Corollary \ref{cor:same-smallest-n-pure-subtopos} we get
\begin{equation*}
\PSh(M)_{< \infty} = \PSh(G)_{< \infty} = \PSh(G).
\end{equation*}
So $\PSh(M)$ is zero-dimensional with content $\PSh(G) \subset \PSh(M)$.
\end{proof}

For $M$ a right Ore monoid, it follows that $\PSh(M)$ has a boundary if and only if $M$ is not a group. In this case, the boundary can not be interpreted as a subtopos. Indeed, since $\PSh(M)$ is a hyperconnected topos, it has only two open subtoposes: the empty topos and $\PSh(M)$ itself. As a result, each nonempty subtopos is dense and contains the smallest dense subtopos $\PSh(G)$. So it is impossible for $\PSh(G)$ to have a complement in the lattice of subtoposes.

\begin{remark}
Using Proposition \ref{prop:example-right-ore-monoids}, we find in particular that the cohomology group $\HH^n(\PSh(M),A)$ with constant coefficients agrees with the $n$th cohomology of the group $G$ with coefficients in $A$.
\end{remark}

\subsection{Example: free monoids}

\begin{proposition} \label{prop:example-free-monoids}
Let $M$ be a free monoid on at least two generators. Then $\PSh(M)$ is a $1$-dimensional topos without boundary.
\end{proposition}
\begin{proof}
We identify $\PSh(M)$ with the category of right $M$-sets. The unique representable presheaf then corresponds to the set $M$ equipped with the right $M$-action given by multiplication. There is then a stable complete family of puncture types in $\PSh(M)$ given by
\begin{equation} \label{eq:free-monoid-puncture-type}
\PSh(M)/(M-\{0\}) \longrightarrow \PSh(M)/M.
\end{equation}
Here $\PSh(M)/M$ is connected and $\PSh(M)/(M-\{0\})$ is isomorphic to a disjoint union of copies of $\PSh(M)/M$, one for each generator of the monoid $M$. So applying $\HH^0(-,A)$ to \eqref{eq:free-monoid-puncture-type} for some set $A$ gives the diagonal map
\begin{equation*}
A \longrightarrow \prod_{i \in S} A,
\end{equation*}
with $S$ the set of generators of $M$. We assumed $|S|\geq 2$ so this is a monomorphism that fails to be an isomorphism. By Proposition \ref{prop:stable-puncture-types}, we see that the puncture type \eqref{eq:free-monoid-puncture-type} is dense but not $0$-pure, in other words it has dimension $1$. Using Proposition \ref{prop:dimension-is-dimension-of-puncture-types} we conclude that $\PSh(M)$ is $1$-dimensional without boundary.
\end{proof}

The assumption that there are at least two generators is necessary. If $M$ is a free monoid on one generator (i.e.\ the natural numbers under addition), then $M$ is commutative, in particular right Ore. So then by Proposition \ref{prop:example-right-ore-monoids} we have that $\PSh(M)$ is zero-dimensional with boundary.

\subsection{Example: relative dimension}

This example features a base topos that is not \Set\/.
The following question was asked by Richard Squire\footnote{The question was communicated to me by Jonathon Funk. It was asked at a seminar given by Jonathon Funk around 1999 at the University of British Columbia.}: does a projection
$$
\gamma:\Sh(\R^2)\to\Sh(\R)
$$
turn $\Sh(\R^2)$ into a $1$-dimensional topos over $\Sh(\R)$? We will show that this is indeed the case.

The puncture type $j: \R^2 - \{0\} \hookrightarrow \R^2$ forms, on its own, a complete family of puncture types. We now need to show that, relative over $\Sh(\R)$, we have $\dim(j) = 1$.

Without loss of generality we assume $\gamma$ preserves the origin. Let $F$ be a sheaf over $\R$. Consider the map
\begin{equation} \label{eq:restriction-R2}
\HH^0(\R^2, \gamma^*F) \longrightarrow \HH^0(\R^2-\{0\}, \gamma^*F).
\end{equation}
We will first show that it is injective for any choice of sheaf $F$. If not, take two global sections $s_1 \neq s_2$ of $\gamma^*F$ that agree on $\R^2-\{0\}$. Now consider the fiber $Y =\gamma^{-1}(0) \subseteq \R^2$, which is itself homeomorphic to $\R$. The restrictions of $s_1$ and $s_2$ to $Y$ are then distinct sections, and if we further restrict to $Y-\{0\}$ they become equal (to see this, note that the stalks of a restricted sheaf are the same as in the original sheaf). This is a contradiction because $Y-\{0\} \to Y$ is dense and $\gamma^*F$ restricted to $Y$ is a constant sheaf. We conclude that \eqref{eq:restriction-R2} is injective for each choice of $F$.

Now consider the sheaf $F$ on $\R$ with as stalks a singleton everywhere, except above the origin, where the stalk contains two distinct elements (this is a particular case of a skyscraper sheaf). Then the map
\begin{equation} \label{eq:restriction-R2-b}
\HH^0(\R^2, \gamma^*F) \longrightarrow \HH^0(\R^2-\{0\}, \gamma^*F).
\end{equation}
is not surjective: the left hand side contains two elements, while the right hand side contains four elements.

As a result, $j$ is $(-1)$-pure but not $0$-pure, i.e.\ $\dim(j)=1$. It follows that $\Sh(\R^2)$ is $1$-dimensional (without boundary) over $\Sh(\R)$.

\section{Calculation of dimension for petit \'etale toposes}

\subsection{First considerations}

In this section we consider the petit \'etale topos $X_\etale$ for $X$ a locally Noetherian scheme. Since $X_\etale$ has a generating family consisting of Noetherian affine schemes, we know that $X_\etale$ is a locally Noetherian topos \cite[Expos\'e VI, Exemple 1.22.1, D\'efinition 2.11]{sga4-2}.

For a locally Noetherian topos, every subtopos is again locally Noetherian, which by Deligne's Theorem implies that each subtopos has enough points \cite[Expos\'e IV, Exercice 9.1.11]{sga4-1}. In particular, this holds for the toposes $(X_\etale)_{\leq n}$. In order to determine these subtoposes, it is enough to determine their points.

\begin{definition}
Let $f: Y \to X$ be a morphism of schemes. Then for any $x \in X$, we set
\begin{equation*}
W_{f,x} ~\simeq~ \Spec(\Ocal_{X,x}^\mathrm{sh}) \times_X Y
\end{equation*}
and we call $W_{f,x}$ the \emph{\'etale disk fiber} of $f$ over $x$. Note that $W_{f,x}$ is only defined up to isomorphism, because the strict henselization is only determined up to isomorphism.
\end{definition}

If $f: Y \to X$ is quasi-compact quasi-separated (qcqs), then for $x \in X$ we can use the formula
\begin{equation} \label{eq:fiberwise-cohomology}
R^qf_*(A)_x ~\simeq~ \HH^q(W_{f,x}, A)
\end{equation}
for every Eilenberg--Mac Lane object $K(A,n)$ in $X_\etale$ \cite[Expos\'e VIII, Th\'eor\`eme 5.2, Remarque 5.3]{sga4-2}. As a result:

\begin{proposition} \label{prop:pure-iff-connected-fiber}
Let $f: Y \to X$ be a qcqs morphism of schemes. Then the induced geometric morphism $Y_\etale \to X_\etale$ is $n$-pure if and only if the \'etale disk fibers $W_{f,x}$ are $n$-connected for all $x \in X$.
\end{proposition}

If $X$ and $Y$ are locally Noetherian, $f$ is automatically quasi-separated. Further, $f$ is quasi-compact whenever $X$ is Noetherian or $f$ is an immersion \cite[\href{https://stacks.math.columbia.edu/tag/01OU}{Section 01OU}]{stacks}. To verify in practice that the \'etale disk fibers are $n$-connected, we can use the following lemma.

\begin{lemma} \label{lmm:finite-coefficients-enough}
Let $X$ be a normal, locally Noetherian scheme. Then for $X_\etale$ to be $n$-connected it is sufficient that
\begin{equation} \label{eq:n-connected}
\begin{split}
\HH^i(X_\etale, A) &~\simeq~ \begin{cases}
A ~&\text{for }i=0 \\
0 ~&\text{for }1 \leq i \leq n
\end{cases}
\end{split}\quad
\end{equation}
for all Eilenberg--Mac Lane objects $K(A,i)$ in $\s\Set$ with $A$ finite. For $i\geq 2$, it is even enough to check that \eqref{eq:n-connected} holds for $A=\Z/p\Z$ for all primes $p$.
\end{lemma}
\begin{proof}
For each $0 \leq i \leq n$, we show that if \eqref{eq:n-connected} holds for $A$ finite, then it holds for general $A$. For $i=0$, we use that each set is a filtered colimit of its finite subsets and apply \cite[\href{https://stacks.math.columbia.edu/tag/0738}{Lemma 0738}]{stacks}.

In the remainder of the proof we can assume that $X$ is connected. In our setting, the fundamental group of $X_\etale$ agrees with the \'etale fundamental group of $X$ \cite{etale-fundamental-group}, in particular it is a profinite group. So, using the formula from Proposition \ref{prop:h1-fundamental-group}, if $\HH^1(X_\etale, A)$ vanishes for finite groups $A$, then the fundamental group vanishes, and then $\HH^1(X_\etale, A)$ vanishes for all groups $A$.

We now consider the case $i\geq 2$ and $A$ abelian. We assume that $\HH^i(X_\etale, A) = 0$ whenever $A$ is finite. If $A$ is a torsion group, then we can write $A$ as a filtered colimit of finite groups, so $\HH^i(X_\etale, A)$ vanishes as well, because $\HH^i$ preserves filtered colimits. For a general abelian group $A$, we have a short exact sequence
\begin{equation*}
0 \to A_\mathrm{tors} \to A \to A/A_\mathrm{tors} \to 0
\end{equation*}
inducing a long exact sequence in cohomology. It is enough to show that $\HH^i$ vanishes for $A/A_\mathrm{tors}$, which is a torsionfree group. Because torsionfree groups are flat as $\Z$-modules, we get an exact sequence
\begin{equation*}
0 \to A/A_\mathrm{tors} \to A/A_\mathrm{tors} \otimes \Q \to A/A_\mathrm{tors} \otimes \Q/\Z \to 0,
\end{equation*}
again inducing a long exact sequence in cohomology. The cohomology of $A/A_\mathrm{tors} \otimes \Q$ vanishes, using the argument as in \cite[2.1]{deninger}, and $A/A_\mathrm{tors} \otimes \Q/\Z$ is a torsion group so its cohomology vanishes as well. We conclude that the cohomology of $A/A_\mathrm{tors}$ vanishes.

Because finite abelian groups are solvable, we can use induction and short exact sequences as above to show that the cohomology vanishes for all finite abelian groups whenever it vanishes for the simple finite abelian groups, i.e.\ the groups of the form $\Z/p\Z$ with $p$ prime.
\end{proof}

\begin{corollary}
Let $j: V \hookrightarrow U$ be an open immersion with $U$ locally Noetherian and normal. Then for $j$ to be $n$-pure it is sufficient that
\begin{equation} \label{eq:n-pure}
R^ij_*(A) ~\simeq~ \begin{cases}
A &\text{for }i=0 \\
0 &\text{for }1 \leq i \leq n
\end{cases}
\end{equation}
for all Eilenberg--Mac Lane objects with $A$ finite. For $i\geq 2$, it is even enough to check that \eqref{eq:n-pure} holds for $A=\Z/p\Z$ for all primes $p$.
\end{corollary}
\begin{proof}
For a downwards closed family $\A$ of Eilenberg--Mac Lane objects, we have that $j$ is $\A$-pure if and only if the \'etale disk fibers are $\A$-connected. Taking into account Lemma \ref{lmm:finite-coefficients-enough}, it remains to prove that the \'etale disk fibers are normal and locally Noetherian. This follows from $U$ being locally Noetherian and normal and $j$ being an open immersion.
\end{proof}

\subsection{Excellent, regular schemes}

\begin{proposition}
Let $X$ be an excellent scheme of characteristic $0$ and let $x \in X$ be a regular point of codimension $c$. Then $x$ is contained in $(X_\etale)_{\leq n}$ if and only if $2c \leq n$.
\end{proposition}
\begin{proof}
\underline{If}.~ Suppose $2c \leq n$. To show that $x$ is contained in $(X_\etale)_{\leq n}$, it is enough to show that every $(n-1)$-pure map $j: V \hookrightarrow U$ in $X_\etale$ induces an isomorphism on the fibers above $x$. It is enough to prove this for $U$ in a family of generators of $X_\etale$, so we can assume that $U$ is representable by an \'etale scheme over $X$. Let $y \in U$ be a point above $x$ that is not contained in $V$. We want to derive a contradiction. Note that $y$ is again regular of codimension $c$ \cite[\href{https://stacks.math.columbia.edu/tag/07QL}{Section 07QL}]{stacks}. With $\overline{\{y\}}$ the closure of $y$ in $U$, we have $V \subseteq U-\overline{\{y\}} \subseteq U$, and therefore 
\begin{equation*}
i: U-\overline{\{y\}} \hookrightarrow U
\end{equation*}
is $(n-1)$-pure as well.

Because $X$ is excellent, $U$ and $\overline{\{y\}}$ are excellent as well, in particular their regular loci are open. So we can take an open subset $U' \subseteq U$ containing $y$ such that both $U'$ and $U' \cap \overline{\{y\}}$ are regular schemes. We use here that the quotient of a regular local ring is regular, and therefore if $y$ is regular in $U$ then it is regular in $\{y\}$. By taking a base change along $U' \to U$ if necessary, we can assume without loss of generality that $U$ and $\overline{\{y\}}$ are themselves regular. For $c > 0$, \cite[Expos\'e XIX, Th\'eor\`eme 3.2]{sga4-3} applies and we find that $R^{2c-1}i_*(\Z/p\Z) \neq 0$, so $i$ is not $(2c-1)$-pure. A fortiori it can't be $(n-1)$-pure. For $c=0$, we use that for a locally Noetherian scheme, an open subset is dense, i.e.\ $(-1)$-pure, if and only if it contains all points of codimension $0$ \cite{dense-open-locally-noetherian}.

\underline{Only if}.~ Now suppose $2c > n$. To show that $x$ is not contained in $(X_\etale)_{\leq n}$, it is enough to construct an $(n-1)$-pure map $j: V \hookrightarrow U$ in $X_\etale$ that does not induce an isomorphism on the fibers above $x$. As above we can take $U \subseteq X$ a regular open subscheme containing $x$ such that the closure $\overline{\{x\}}$ in $U$ is regular as well. We claim that $j: U-\overline{\{x\}} \subseteq U$ is $(n-1)$-pure. For each $x$ in $U$, $\Spec(\Ocal_{X,x}^\mathrm{sh})$ is normal so it is a domain; it follows that the \'etale disk fibers are $0$-connected, so $j$ is $0$-pure (for $c \geq 1$). For $c \geq 2$ and $x \in U$, the inclusion $W_{j,x} \hookrightarrow \Spec(\Ocal_{U,x}^\mathrm{sh})$ induces an equivalence of categories between finite \'etale schemes over $\Spec(\Ocal_{U,x}^\mathrm{sh})$ and finite \'etale schemes over $W_{j,x}$ \cite[\href{https://stacks.math.columbia.edu/tag/0BMA}{Lemma 0BMA}, \href{https://stacks.math.columbia.edu/tag/0EY7}{Lemma 0EY7}]{stacks}. In particular we get an equivalence of categories for the categories of $G$-torsors on these schemes, for $G$ a finite group. So we get $\HH^1(W_{j,x},G) \simeq \HH^1(\Spec(\Ocal_{U,x}^\mathrm{sh}),G) \simeq 0$, and using Lemma \ref{lmm:finite-coefficients-enough} we conclude that the \'etale disk fibers are $1$-connected, so $j$ is $1$-pure. In higher degrees we can apply \cite[Expos\'e XIX, Th\'eor\`eme 3.2]{sga4-3} again, and see that $j$ is $(2c-2)$-pure, in particular $(n-1)$-pure (we can restrict to finite cyclic groups because of Lemma \ref{lmm:finite-coefficients-enough}).
\end{proof}

\begin{corollary}
Let $X$ be an excellent, regular scheme of characteristic $0$. Then $(X_\etale)_{\leq n}$ is the subtopos corresponding to the subset $X' \subseteq X$ of points of codimension at most $\left\lfloor \frac{n}{2} \right\rfloor$. In particular, $(X_\etale)_{\leq n-1} = (X_\etale)_{\leq n}$ for $n$ odd.
\end{corollary}

\begin{corollary} \label{cor:dimension-of-excellent-regular}
Let $X$ be an excellent, regular scheme of characteristic $0$ and Krull dimension $d$. Then $X_\etale$ is of dimension $2d$.
\end{corollary}

If we weaken some of the assumptions we at least get an inequality.

\begin{proposition} \label{prop:dimension-bound-locally-noetherian}
Let $X$ be a locally Noetherian scheme of characteristic $0$ and Krull dimension $d$. Then the dimension of $X_\etale$ is at most $2d$.
\end{proposition}
\begin{proof}
For $d=0$, we can replace $X$ by its maximal reduced subscheme without changing the associated petit \'etale topos, and then locally $X$ is the prime spectrum of a field. So the dimension of $X_\etale$ is $0$. In the remainder of the proof we assume $d \geq 1$.

Let $j: V \hookrightarrow U$ be an open immersion with $U$ \'etale over $X$. We have to show that if $j$ is $(2d-1)$-pure, then it is $\infty$-pure. If $j$ is $(2d-1)$-pure, then the \'etale disk fibers $W_{j,x}$ are $(2d-1)$-connected for all $x \in U$. It follows from \cite[Expos\'e XVIII-A, Theorem 1.1]{travaux-de-gabber} that for $W_{f,x} \subseteq \Spec(\Ocal_{U,x}^\mathrm{sh})$ the cohomology groups $\HH^q(W_{f,x}, \Z/p\Z)$ vanish, for $q > 2d-1$ and $p$ prime. Using Lemma \ref{lmm:finite-coefficients-enough} we find that $W_{f,x}$ is $\infty$-connected for each $x \in U$, so $j$ is $\infty$-pure.
\end{proof}

If $X$ is a variety over a field of characteristic $0$, then its regular locus is an open subscheme $U$ that has the same Krull dimension as $X$. Since $U$ is excellent and regular, we can apply Corollary \ref{cor:dimension-of-excellent-regular} and Proposition \ref{prop:dimension-bound-locally-noetherian} to get
\begin{equation*}
2d = \dim(U_\etale) \leq \dim(X_\etale) \leq 2d
\end{equation*}
for $d$ the Krull dimension of $X$. We conclude:
\begin{corollary}
Let $X$ be a variety over a field of characteristic $0$ and Krull dimension $d$. Then $\dim(X_\etale) = 2d$.
\end{corollary}

\subsection{Pro-\'etale local structure} 

Let $X$ be a scheme and take $x \in X$. If we take the \'etale disk fiber of the identity map $X \to X$ over the point $x$, then we get $\Spec(\Ocal_{X,x}^\sh)$.

\begin{definition}
For a scheme $X$ and a point $x \in X$, we define the \emph{\'etale disk} $D_x$ at $x$ as the scheme
\begin{equation*}
D_x = \Spec(\Ocal_{X,x}^\mathrm{sh}).
\end{equation*}
\end{definition}

Because the identity map is $\infty$-pure, we conclude from Proposition \ref{prop:pure-iff-connected-fiber} that the \'etale cohomology of $D_x$ (with constant coefficients) is trivial, i.e.\ $\HH^0((D_{x})_\etale,A) = A$ and the cohomology in degrees $\geq 1$ vanishes.

We can write $\Spec(\Ocal_{X,x}^\sh)$ as a cofiltered limit in the category of affine schemes (and therefore also in the category of schemes) of affine \'etale neighborhoods $X_i$ of $X$. The transition maps $X_j \to X_i$ are then affine as well. We claim that at the level of Grothendieck toposes we have
\begin{equation} \label{eq:limit-as-toposes-D-x}
(D_x)_\etale ~\simeq~ \varprojlim_{i \in I}\, (X_i)_\etale
\end{equation}
and
\begin{equation} \label{eq:limit-of-cohomologies-D-x}
\HH^n((D_x)_\etale, A) ~\simeq~ \varinjlim_{i \in I}\, \HH^n((X_i)_\etale, A).
\end{equation}
To prove this, we need to leverage the fact that each $X_i$ is affine, as follows.

Let $X_\etale$ be the petit \'etale topos on a scheme $X$. Then $X_\etale$ is a coherent topos if and only if $X$ is qcqs (quasi-compact quasi-separated) \cite[\S 1.22.1]{sga4-2}. For a morphisms of schemes $f: Y \to X$, we similarly have that the induced geometric morphism $f: Y_\etale \to X_\etale$ is coherent if and only if $f$ is qcqs \cite[Exemple 3.10]{sga4-2}.

Now let $(X_i)_{i \in I}$ be a cofiltered diagram of qcqs schemes over a scheme $X$ and qcqs morphisms between them. Then by \cite[Expos\'e VI, Corollaire 8.7.7]{sga4-2}, there is a natural isomorphism
\begin{equation} \label{eq:cohomology-of-cofiltered-limit}
\HH^n(\varprojlim_{i\in I} (X_i)_\etale, A) ~\simeq~ \varinjlim_{i \in I} \HH^n((X_i)_\etale, A)
\end{equation}
for each abelian group object $A$ in $X_\etale$.
If moreover the transition morphisms $X_j \to X_i$ are affine, then $\varprojlim_{i \in I}X_i$ exists and by \cite[Expos\'e VII, Lemme 5.6]{sga4-2} we have
\begin{equation} \label{eq:limit-as-toposes-general}
(\varprojlim_{i \in I} X_i)_\etale ~\simeq~ \varprojlim_{i \in I} (X_i)_\etale.
\end{equation}
In particular, \eqref{eq:cohomology-of-cofiltered-limit} can be rewritten as
\begin{equation} \label{eq:limit-of-cohomologies-general}
\HH^n((\varprojlim_{i \in I} X_i)_\etale, A) ~\simeq~ \varinjlim_{i \in I}\HH^n((X_i)_\etale, A),
\end{equation}
see also \cite[Expos\'e VII, Corollaire 5.8]{sga4-2}. The analogous result holds for $A$ a group object in $X_\etale$ and $n=1$, or $A$ any object in $X_\etale$ and $n=0$, see \cite[Expos\'e VII, Remarques 5.14(a)]{sga4-2}.

Formula \eqref{eq:limit-as-toposes-D-x} and \eqref{eq:limit-of-cohomologies-D-x} then follow from formulas \eqref{eq:limit-as-toposes-general} and \eqref{eq:limit-of-cohomologies-general}, respectively.

Above, we considered $D_x$ as the \'etale fiber of the identity map $X \to X$ over the point $x$. We can similarly look at the \'etale fiber of the open subscheme $X-\overline{\{x\}} \subseteq X$.

\begin{definition}
For a scheme $X$ and a point $x \in X$, we define the \emph{punctured \'etale disk} $U_x$ at $x$ as the open subscheme
\begin{equation*}
U_x = D_x - \{x\}
\end{equation*}
of the \'etale disk $D_x$.
\end{definition}

If we write $D_x \simeq \varprojlim_{i \in I} X_i$ as above, then after base change we find
\begin{equation*}
U_x ~\simeq~ \varprojlim_{i \in I} U_i
\end{equation*}
with $U_i \simeq X_i \times_X (X-\overline{\{x\}})$. If $X$ is locally Noetherian, then for each $i \in I$, $X_i$ is Noetherian, so $U_i$ is quasi-compact and quasi-separated (as an open subset of an affine Noetherian scheme). As a result, the formulas \eqref{eq:limit-as-toposes-general} and \eqref{eq:limit-of-cohomologies-general} apply, and we get:
\begin{equation} \label{eq:limit-for-U-x}
\begin{split}
(U_x)_\etale ~&\simeq~ \varprojlim_{i \in I} (U_i)_\etale \\
\HH^n((U_x)_\etale, A) ~&\simeq~ \varinjlim_{i \in I} \HH^n((U_i)_\etale, A),
\end{split}
\end{equation}

\begin{remark}
Let $g: X' \to X$ be an \'etale morphism, and let $x' \in X'$ be a point with $g(x') = x$. Then $\Ocal_{X',x'}^\sh$ is isomorphic to $\Ocal_{X,x}$. As a result, the punctured \'etale disks associated to $(X',x')$ resp.\ $(X,x)$ are isomorphic too.
\end{remark}

\subsection{\'Etale-dimension at a point}

\begin{definition}
For a scheme $X$ and a point $x \in X$, we define the \emph{\'etale-dimension} at the point $x$ to be the dimension of the puncture type $(U_x)_\etale \hookrightarrow (D_x)_\etale$. Equivalently, it is the smallest number $n$ such that $(U_x)_\etale$ fails to be $(n-1)$-connected, or $\infty$ if $(U_x)_\etale$ is $\infty$-connected.
\end{definition}

\begin{proposition} \label{prop:dimension-of-scheme}
Let $X$ be a locally Noetherian scheme. Then the dimension of $X_\etale$ is the supremum of the \'etale-dimensions at the locally closed points of $X$, taking into account only those locally closed points for which the \'etale-dimension is finite.
\end{proposition}
\begin{proof}
The topos $X_\etale$ has a family of generators consisting of the \'etale maps $Y \to X$ with $Y$ affine. Here $Y$ is again locally Noetherian  \cite[\href{https://stacks.math.columbia.edu/tag/01T6}{Lemma 01T6}]{stacks} (even Noetherian, because affine schemes are quasi-compact). Since every nonempty closed subset of a locally Noetherian scheme contains a closed point \cite[\href{https://stacks.math.columbia.edu/tag/02IL}{Lemma 02IL}]{stacks}, every open subset of $Y$ is contained in an open subset of the form $Y-\{x\}$ for some closed point $x$ of $Y$. It follows that the inclusions
\begin{equation} \label{eq:locally-noetherian-puncture-type}
j~:~ (Y-\{x\})_\etale ~\longrightarrow~ Y_\etale,
\end{equation}
with $Y$ and $x$ as above, form a complete family of puncture types for $X_\etale$. Because $x$ is closed in $Y$ and \'etale morphisms are open, the points $x$ that appear here are precisely the locally closed points of $X$.

By checking the criterion from Theorem \ref{thm:characterization-derived-pushforwards} in each point $y \in Y$, we find that the puncture type \eqref{eq:locally-noetherian-puncture-type} is $n$-pure if and only if
\begin{equation} \label{eq:n-pure-in-each-point}
R^q j_*(A)_y ~\simeq~ \begin{cases}
A \quad& \text{for }q=0 \\
0 \quad& \text{for }1 \leq q \leq n,
\end{cases}
\end{equation}
for $A$ an abelian group and for $A$ any group if $q=1$ and for $A$ any object if $q=0$. We take stalks here for the \'etale topology, which are only defined up to isomorphism for a point $y \in Y$, so there is some abuse of notation. Note that the conditions from \eqref{eq:n-pure-in-each-point} are trivially satisfied for $y \neq x$, so we only have to check the case $y = x$.

Taking into account \eqref{eq:fiberwise-cohomology}, we see that $j$ is $n$-pure if and only if $U_x$ is $n$-connected. In other words, the dimension of the puncture type $j$ agrees with the \'etale-dimension at the point $x$. The statement now follows from Proposition \ref{prop:dimension-is-dimension-of-puncture-types}.
\end{proof}

Note that the \'etale-dimension at $x$ is the same for $X$, $\Spec(\Ocal_{X,x})$, $\Spec(\Ocal_{X,x}^\mathrm{h})$ and $\Spec(\Ocal_{X,x}^\mathrm{sh})$, because it only depends on the strict henselization at $x$.

Below we will see that we can look at the completed local rings as well, under some additional conditions. First, we need the following lemma.

\begin{proposition} \label{prop:completion}
Let $A$ be a Noetherian, normal local ring and let $\widehat{A}$ be its completion. Then the \'etale-dimensions of $\Spec(A)$ and $\Spec(\widehat{A})$ at their closed points are the same.
\end{proposition}
\begin{proof}
The henselization $A^h$ of $A$ is again Noetherian and normal \cite[\href{https://stacks.math.columbia.edu/tag/07QL}{Section 07QL}]{stacks} and has the same completion as $A$ \cite[Th\'eor\`eme 18.6.6(iv)]{ega4-4}. Moreover, the \'etale-dimensions of $\Spec(A)$ and $\Spec(A^h)$ at their closed points are the same. So by replacing $A$ by $A^h$ if necessary, we can assume that $A$ is henselian.

Because $A$ is henselian, we can compute the strict henselization as a filtered colimit $A^\mathrm{sh} \simeq \varinjlim_{i \in I} A_i$ of finite \'etale extensions $A \subseteq A_i$. The analogous statement holds for $\widehat{A}$, which is also henselian because it is complete. Moreover, the base change $- \otimes_A \widehat{A}$ establishes an equivalence of categories between the finite \'etale extensions of $A$ and the finite \'etale extensions of $\widehat{A}$, using \cite[\href{https://stacks.math.columbia.edu/tag/04GK}{Lemma 04GK}]{stacks}. Each of the $A_i$'s are finitely generated as $A$-module, so we get that $\widehat{A_i} = A_i \otimes_A \widehat{A}$ \cite[III, \S 3.4, Theorem 3]{bourbaki-comm-alg}. This shows that $\widehat{A}^\mathrm{sh} \simeq \varinjlim_{i\in I} \widehat{A_i}$.

Let $x$ and $\widehat{x}$ be the points corresponding to $A \to A^\mathrm{sh}$ and $\widehat{A} \to \widehat{A}^\sh$, respectively. We write $U_i$ and $\widehat{U_i}$ for the schemes that you get after removing the closed point from $\Spec(A_i)$ resp.\ $\Spec(\widehat{A_i})$. With these notations, we find for finite $A$ that
\begin{equation*}
\begin{split}
\HH^n((U_{\widehat{x}})_\etale, A) &~\simeq~ \varinjlim_{i \in I}~ \HH^n((\widehat{U_i})_\etale, A) \\
&~\simeq~ \varinjlim_{i \in I}~ \HH^n((U_i)_\etale, A) \\
&~\simeq~ \HH^n((U_x)_\etale, A).
\end{split}
\end{equation*}
Here the second isomorphism is the essence of the proof. We get it by applying Gabber's formal base change theorem. For a proof of this theorem, see \cite[Corollary 6.6.4]{fujiwara} for the case where $A$ is an abelian group and \cite[XX, Theor\`eme 2.1.1]{travaux-de-gabber} for $n=0,1$. The first and third isomorphisms follow from \eqref{eq:limit-for-U-x}.

By Lemma \ref{lmm:finite-coefficients-enough}, we conclude that $U_{\widehat{x}}$ is $n$-connected if and only if $U_x$ is $n$-connected (using that $A$ is normal). In other words, the \'etale-dimensions of $\Spec(A)$ and $\Spec(\widehat{A})$ at $x$ resp.\ $\widehat{x}$ are the same.
\end{proof}

\begin{proposition}
The topos $\Spec(\Z)_\etale$ is two-dimensional. 
\end{proposition}
\begin{proof}
We start by computing the \'etale-dimension of $\Spec(\Z_p)$ at its closed point $x$, for a prime number $p$. Here the punctured spectrum $U_x$ is given by $\Spec(K)$ for $K$ the maximal unramified extension of $\Q_p$. Clearly, $U_x$ is $0$-connected. It also fails to be $1$-connected: for example the extension $K(\sqrt{p})$ of $K$ determines a $\Z/2\Z$-torsor over $U_x$. So the \'etale-dimension of $\Spec(\Z_p)$ at the closed point is $2$, for each prime $p$. Combining Proposition \ref{prop:dimension-of-scheme} and Proposition \ref{prop:completion}, we conclude that $\Spec(\Z)_\etale$ is two-dimensional.
\end{proof}


\appendix

\section{Product toposes and evaluation}
\label{appendix:base-change-to-presheaf-topos}

Let $\T$ be a Grothendieck topos over $\Set$.
For each topos $\F$, we then write
\begin{equation*}
\tilde{\F} = \F \times \T.
\end{equation*}
For a geometric morphism $\rho : \F \to \E$, we will consider the induced geometric morphism
\begin{equation*}
\tilde{\rho} : \tilde{\F} \longrightarrow \tilde{\E}.
\end{equation*}
This is a morphism over $\T$. 
More precisely, if $e : \E \to \Set$ and $f : \F \to \Set$
are the global sections geometric morphisms,
then the diagram
\begin{equation*}
\begin{tikzcd}
\tilde{\F} \ar[rr,"{\tilde{\rho}}"] \ar[rd,"{\tilde{f}}"'] & & \tilde{\E} \ar[ld,"{\tilde{e}}"] \\
& \T &
\end{tikzcd}
\end{equation*}
commutes.

Now let $T$ be an object in $\T$. We can consider the diagram
\begin{equation} \label{eq:eval-in-T-base}
\begin{tikzcd}
\Set & \ar[l,"{\delta}"'] \T/T \ar[r,"{\pi}"] & \T
\end{tikzcd}
\end{equation}
with $\delta$ the global sections geometric morphism and 
$\pi$ the \'etale geometric morphism associated to $T$. 
Taking the base changes of $\rho$ along the diagram \eqref{eq:eval-in-T} then leads to a diagram of the form
\begin{equation} \label{eq:eval-in-T-big}
\begin{tikzcd}
\F \ar[d,"{\rho}"'] & \ar[l,"{\delta_\F}"'] \tilde{\F}/\tilde{f}^*T \ar[r,"{\pi_\F}"] \ar[d,"{\dot{\rho}}"] & \tilde{\F} \ar[d,"{\tilde{\rho}}"] \\
\E & \ar[l,"{\delta_\E}"'] \tilde{\E}/\tilde{e}^*T \ar[r,"{\pi_\E}"] & \tilde{\E}
\end{tikzcd}
\end{equation}
with both squares pullbacks. For an object $X$ in $\tilde{\F}$, we define its evaluation in $T$ as
\begin{equation} \label{eq:eval-in-T}
X(T) = (\delta_\F)_*\pi_\F^*X.
\end{equation}
In this way, we can interpret objects in $\tilde{\F}$ 
as certain presheaves on $\T$ with values in $\F$.

It turns out that we can compute $\tilde{\rho}_*$ in a ``pointwise'' way.
More precisely:

\begin{proposition} \label{prop:direct-image-pointwise}
With notations as above, there is an isomorphism
\begin{equation*}
(\tilde{\rho}_*X)(T) \simeq \rho_*(X(T))
\end{equation*}
natural in $X$ and $T$.
\end{proposition}
\begin{proof}
We have to show that $(\delta_\E)_*(\pi_\E)^*\tilde{\rho}_* \simeq \rho_* (\delta_\F)_* \pi_\F^*$. 
This follows from the Beck--Chevalley condition $\pi_\E^*\tilde{\rho}_* \simeq \dot{\rho}_* \pi_\F^*$ in \eqref{eq:eval-in-T-big},
which holds because $\pi$ is \'etale (in particular locally connected). This shows the existence of the isomorphism and naturality in $X$.

Proving naturality in $T$ involves a more complicated argument. As a shortcut, we assume that if two geometric morphisms commute, then the commutator is the identity (for both the left adjoints and the right adjoints). For working with natural transformations, we follow the common convention that the identity natural transformation of a functor $F$ is again denoted by $F$, that horizontal composition is denoted with concatenation, and that vertical composition is denoted by $\circ$. We use the notations from diagram \eqref{eq:eval-in-T}, and for the analogous diagram with $T'$ instead of $T$ we use the notations $\dot{\rho}'$, $\delta_\F'$, $\pi_\F'$, $\delta_\E'$, $\pi_\E'$ instead of $\dot{\rho}$, $\delta_\F$, $\pi_\F$, $\delta_\E$ and $\pi_\E$. For a geometric morphism $f$, we use the common notations $\eta_f: 1 \to f_*f^*$ and $\epsilon_f: f^*f_* \to 1$ for the unit resp.\ counit.

Consider a morphism $T' \to T$ in $\T$, and let
\begin{equation*}
\xi: \tilde{\F}/\tilde{f}^*T' \to \tilde{\F}/\tilde{f}^*T
\end{equation*}
and 
\begin{equation*}
\zeta: \tilde{\E}/\tilde{f}^*T' \to \tilde{\E}/\tilde{f}^*T
\end{equation*}
be the induced geometric morphisms. The restriction maps $(\tilde{\rho}X)(T') \to (\tilde{\rho}X)(T)$ and $\rho_*(X(T')) \to \rho_*(X(T))$ can be written as 
$(\delta_\E)_*\eta_{\zeta}\pi_\E^*\tilde{\rho}_*$ resp.\
$\rho_*(\delta_\F)_*\eta_\xi\pi_\F^*$, i.e.\ they are induced by the adjunction units $\eta_\zeta$ and $\eta_\xi$ of $\zeta^* \dashv \zeta_*$ resp.\ $\xi^* \dashv \xi_*$.

The Beck--Chevalley map $\pi_\E^*\tilde{\rho}_* \to \dot{\rho}_* (\pi_\F)^*$ is defined as $\dot{\rho}_*\pi_\F^*\epsilon_{\tilde{\rho}} \circ \eta_{\dot{\rho}}\pi_\E^*\tilde{\rho}_*$.

The fact that the isomorphism in the Proposition statement is natural in $T$ can then be verified as follows:
\begin{equation*}
\begin{split}
& (\delta_{\E}')_* \dot{\rho}_*' \pi_{\F}'^* \epsilon_{\tilde{\rho}} 
\circ (\delta_{\E}')_* \eta_{\dot{\rho}'} \pi_{\E}'^* \tilde{\rho}_* 
\circ (\delta_{\E})_* \eta_{\zeta} \pi_{\E}^* \tilde{\rho}_* \\
& \simeq (\delta_{\E})_* \dot{\rho}_* \xi_* \xi^* \pi_{\F}^* \epsilon_{\tilde{\rho}} 
\circ (\delta_{\E})_* \zeta_* \eta_{\dot{\rho}'} \zeta^* \pi_{\E}^* \tilde{\rho}_* 
\circ (\delta_{\E})_* \eta_{\zeta} \pi_{\E}^* \tilde{\rho}_* \\
& \simeq (\delta_{\E})_* \dot{\rho}_* \xi_* \xi^* \pi_{\F}^* \epsilon_{\tilde{\rho}} 
\circ (\delta_{\E})_* \left(\zeta_* \eta_{\dot{\rho}'} \zeta^* \circ \eta_{\zeta}\right) \pi_{\E}^* \tilde{\rho}_* \\
& \simeq (\delta_{\E})_* \dot{\rho}_* \xi_* \xi^* \pi_{\F}^* \epsilon_{\tilde{\rho}} 
\circ (\delta_{\E})_* \dot{\rho}_* \eta_{\xi} \dot{\rho}^{\, *} \pi_{\E}^* \tilde{\rho}_* 
\circ (\delta_{\E})_* \eta_{\dot{\rho}} \pi_{\E}^* \tilde{\rho}_* \\
& \simeq \rho_* (\delta_{\F})_* \xi_* \xi^* \pi_{\F}^* \epsilon_{\tilde{\rho}} 
\circ \rho_* (\delta_{\F})_* \eta_{\xi} \pi_{\F}^* \tilde{\rho}^{\, *} \tilde{\rho}_* 
\circ (\delta_{\E})_* \eta_{\dot{\rho}} \pi_{\E}^* \tilde{\rho}_* \\
& \simeq \rho_* (\delta_{\F})_* \eta_{\xi} \pi_{\F}^* 
\circ \rho_* (\delta_{\F})_* \pi_{\F}^* \epsilon_{\tilde{\rho}} 
\circ (\delta_{\E})_* \eta_{\dot{\rho}} \pi_{\E}^* \tilde{\rho}_* \\
& \simeq \rho_* (\delta_{\F})_* \eta_{\xi} \pi_{\F}^* 
\circ (\delta_{\E})_* \dot{\rho}_* \pi_{\F}^* \epsilon_{\tilde{\rho}} 
\circ (\delta_{\E})_* \eta_{\dot{\rho}} \pi_{\E}^* \tilde{\rho}_*.
\end{split}
\end{equation*}
\end{proof}

There is a similar result for $\tilde{\rho}^*$, provided that 
$\delta : \T/T \to \Set$
is tidy.
By \cite[Proposition 4.4]{nlab-proper}, $\delta$ is tidy if and only if the
object $T$ is finitely presented, in the sense that the functor
\begin{equation*}
\Hom_\T(T,-) : \T \to \Set
\end{equation*}
preserves filtered colimits.

\begin{proposition} \label{prop:inverse-image-pointwise}
With notations as above, if $\rho$ is locally connected or $T$ is finitely presented, then there is an isomorphism
\begin{equation*}
\tilde{\rho}^*(X)(T) \simeq \rho^*(X(T))
\end{equation*}
natural in $X$ and $T$.
\end{proposition}
\begin{proof}
We compute $(\delta_\F)_*\pi_\F^* \tilde{\rho}^* \simeq (\delta_\F)_*\dot{\rho}^*\pi_\E^* \simeq \rho^*(\delta_\E)_*\pi_\E^*$, 
where in the second isomorphism we use the Beck--Chevalley condition $(\delta_\F)_*\dot{\rho}^* \cong \rho^*(\delta_\E)_*$,
which holds whenever $\rho$ is locally connected or $\delta$ is tidy 
(and $\delta$ is tidy whenever $T$ is finitely presented).

To show naturality in $T$ we again need a more complicated calculation. We use the same strategy and notations as in the proof of Proposition \ref{prop:direct-image-pointwise}.

The isomorphism in the Proposition statement is given by
\begin{equation*}
(\delta_\F)_*\dot{\rho}^*\epsilon_{\delta_\E}\pi_\E^* \circ \eta_{\delta_\F}\rho^*(\delta_\E)_*\pi_\E^*.
\end{equation*}
We can then show naturality in $T$ through the following calculation:
\begin{equation*}
\begin{split}
& (\delta_\F')_*\dot{\rho}'^*\epsilon_{\delta_\E'}\pi_\E'^*
\circ \eta_{\delta_\F'}\rho^*(\delta_\E')_*\pi_\E'^*
\circ \rho^*(\delta_\E)_*\eta_\zeta\pi_\E^* \\
\simeq~& (\delta_\F)_*\xi_*\dot{\rho}'^*\epsilon_{\delta_\E'}\pi_\E'^*
\circ \eta_{\delta_\F'}\rho^*(\delta_\E)_*\zeta_*\zeta^*\pi_\E^*
\circ \rho^*(\delta_\E)_*\eta_\zeta\pi_\E^* \\
\simeq~& (\delta_\F)_*\xi_*\dot{\rho}'^*\epsilon_{\delta_\E'}\pi_\E'^*
\circ \eta_{\delta_\F'}\rho^*(\delta_\E)_*\eta_\zeta\pi_\E^* \\
\simeq~& (\delta_\F)_*\left(\vphantom{\int}\xi_*\dot{\rho}'^*\epsilon_\zeta\zeta^*
\circ \xi_*\dot{\rho}'^*\zeta^*\epsilon_{\delta_\epsilon}\zeta_*\zeta^*
\circ \eta_\xi\delta_\F^*\rho^*(\delta_\E)_*\eta_\zeta \right)\pi_\E^*
\circ \eta_{\delta_\F}\rho^*(\delta_\E)_*\pi_\E^* \\
\simeq~& (\delta_\F)_*\left(\vphantom{\int}\xi_*\dot{\rho}'^*\epsilon_\zeta\zeta^*
\circ \xi_*\dot{\rho}'^*\zeta^*\epsilon_{\delta_\epsilon}\zeta_*\zeta^*
\circ \eta_\xi\dot{\rho}^*\delta_\E^*(\delta_\E)_*\eta_\zeta \right)\pi_\E^*
\circ \eta_{\delta_\F}\rho^*(\delta_\E)_*\pi_\E^* \\
\simeq~& (\delta_\F)_*\left(\vphantom{\int}\xi_*\dot{\rho}'^*\epsilon_\zeta\zeta^*
\circ \xi_*\dot{\rho}'^*\zeta^*\eta_\zeta
\circ \eta_\xi\dot{\rho}^*\epsilon_{\delta_\E} \right)\pi_\E^*
\circ \eta_{\delta_\F}\rho^*(\delta_\E)_*\pi_\E^* \\
\simeq~& (\delta_\F)_*\left(\vphantom{\int}\xi_*\dot{\rho}'^*\zeta^*
\circ \eta_\xi\dot{\rho}^*\epsilon_{\delta_\E} \right)\pi_\E^*
\circ \eta_{\delta_\F}\rho^*(\delta_\E)_*\pi_\E^* \\
\simeq~& (\delta_\F)_*\left(\vphantom{\int}\xi_*\xi^*\dot{\rho}^*
\circ \eta_\xi\dot{\rho}^*\epsilon_{\delta_\E} \right)\pi_\E^*
\circ \eta_{\delta_\F}\rho^*(\delta_\E)_*\pi_\E^* \\
\simeq~& (\delta_\F)_*\left(\vphantom{\int}\eta_\xi\dot{\rho}^*
\circ \dot{\rho}^*\epsilon_{\delta_\E} \right)\pi_\E^*
\circ \eta_{\delta_\F}\rho^*(\delta_\E)_*\pi_\E^* \\
\simeq~& (\delta_\F)_*\eta_\xi\dot{\rho}^*\pi_\E^*
\circ (\delta_\F)_*\dot{\rho}^*\epsilon_{\delta_\E}\pi_\E^*
\circ \eta_{\delta_\F}\rho^*(\delta_\E)_*\pi_\E^* \\
\simeq~& (\delta_\F)_*\eta_\xi\pi_\F^*\tilde{\rho}^*
\circ (\delta_\F)_*\dot{\rho}^*\epsilon_{\delta_\E}\pi_\E^*
\circ \eta_{\delta_\F}\rho^*(\delta_\E)_*\pi_\E^*.
\end{split}
\end{equation*}
Essential here are the formulas
\begin{equation*}
\epsilon_{\delta_\E'} \simeq \epsilon_\zeta\circ\zeta^*\epsilon_{\delta_\E}\zeta_*\quad\text{and}\quad\eta_{\delta_\F'} \simeq (\delta_\F)_*\eta_\zeta\delta_\F^*\circ \eta_{\delta_\F}.
\end{equation*}
\end{proof}

In order to be able to calculate $\tilde{\rho}_!$ in a pointwise way, 
we need that
\begin{equation*}
\delta : \T/T \to \Set
\end{equation*}
is local. 
By \cite[Proposition 1.6, Proposition 2.6]{nlab-tiny}, 
this is the case if and only if $T$ is tiny, 
in the sense that the functor $\Hom_\T(T,-)$ preserves colimits.
\begin{proposition} \label{prop:connected-components-pointwise}
With notations as above, if $\rho$ is locally connected and $T$ tiny, 
then there is an isomorphism
\begin{equation*}
\tilde{\rho}_!(X)(T) \simeq \rho_!(X(T))
\end{equation*}
natural in $X$ and $T$.
\end{proposition}
\begin{proof}
If $\delta : \T/T \to \Set$ is local, then it has a center $c$ with $c^* \simeq \delta_*$.
The base change of $\delta$ along $e : \E \to \Set$ or $f : \F \to \Set$ is 
again local, with as center the corresponding pullback of $c$.
The center will be denoted by $c_\E$ and $c_\F$ respectively.
We then compute: 
\begin{equation*}
(\delta_\E)_*\pi_\E^*\tilde{\rho}_! \simeq c_\E^*\pi_\E^*\tilde{\rho}_! \simeq \rho_! c_\F^*\pi_\F^* \simeq \rho_! (\delta_\F)_* \pi_\F^*.
\end{equation*}
In the first and third isomorphism we use $c^*\simeq\delta_*$.
The second isomorphism is a Beck--Chevalley condition 
that holds because $\tilde{\rho}$ is locally connected.

To prove naturality in $T$ we proceed as in Proposition \ref{prop:direct-image-pointwise} and Proposition \ref{prop:inverse-image-pointwise}, with the difference in notation that now $\eta_\rho$ and $\epsilon_\rho$ denote the unit and counit of $\rho_! \dashv \rho^*$ (and similarly for $\dot{\rho}$ and $\tilde{\rho}$). The isomorphism from the Proposition statement is given by
\begin{equation*}
\epsilon_\rho c_\F^*\pi_\F^*\tilde{\rho}_! \circ \rho_! c_\F^*\pi_\F^*\eta_{\tilde{\rho}}~.
\end{equation*}
We now calculate
\begin{equation*}
\begin{split}
& \epsilon_\rho c_\F'^*\pi_\F'^*\tilde{\rho}_!
\circ \rho_! c_\F'^* \pi_\F'^*\eta_{\tilde{\rho}}
\circ \rho_!(\delta_\F)_*\eta_\xi\pi_\F^* \\
\simeq~& \epsilon_\rho c_\F^*\xi_*\xi^*\pi_\F^*\tilde{\rho}_!
\circ \rho_!c_\F^*\xi_*\xi^*\pi_\F^*\eta_{\tilde{\rho}}
\circ \rho_!(\delta_\F)_*\eta_\xi\pi_\F^* \\
\simeq~& \epsilon_\rho c_\F^*\xi_*\xi^*\pi_\F^*\tilde{\rho}_!
\circ \rho_!(\delta_\F)_*\xi_*\xi^*\pi_\F^*\eta_{\tilde{\rho}}
\circ \rho_!(\delta_\F)_*\eta_\xi\pi_\F^* \\
\simeq~& \epsilon_\rho (\delta_\F)_*\xi_*\xi^*\pi_\F^*\tilde{\rho}_!
\circ \rho_!(\delta_\F)_*\eta_\xi\pi_\F^*\eta_{\tilde{\rho}} \\
\simeq~& \epsilon_\rho(\delta_\F)_*\eta_\xi\pi_\F^*\tilde{\rho}_!
\circ \rho_!(\delta_\F)_*\pi_\F^*\eta_{\tilde{\rho}} \\
\simeq~& (\delta_\F)_*\eta_\xi\pi_\F^*\tilde{\rho}_!
\circ \epsilon_\rho(\delta_\F)_*\pi_\F^*\tilde{\rho}_!
\circ \rho_!(\delta_\F)_*\pi_\F^*\eta_{\tilde{\rho}} \\
\simeq~& (\delta_\F)_*\eta_\xi\pi_\F^*\tilde{\rho}_!
\circ \epsilon_\rho c_\F^*\pi_\F^*\tilde{\rho}_!
\circ \rho_! c_\F^*\pi_\F^*\eta_{\tilde{\rho}}.
\end{split}
\end{equation*}
\end{proof}

\begin{remark} \label{rem:remark-appendix}
The above results have internal versions as well,
obtained by replacing $\delta$ by $\pi$ in the above:
\begin{itemize}
\item there is a natural isomorphism $\tilde{\rho}_*(X^{\tilde{f}^*(T)}) \simeq \tilde{\rho}_*(X)^{\tilde{e}^*(T)}$;
\item there is a natural isomorphism $\tilde{\rho}^*(X^{\tilde{e}^*(T)}) \simeq \tilde{\rho}^*(X)^{\tilde{f}^*(T)}$,
whenever $\rho$ is locally connected or $\pi : \T/T \to \T$ is tidy;
\item there is a natural isomorphism $\tilde{\rho}_!(X^{\tilde{f}^*(T)}) \simeq \tilde{\rho}_!(X)^{\tilde{e}^*(T)}$,
whenever $\rho$ is locally connected and $\pi : \T/T \to \T$ is local.
\end{itemize}
The geometric morphism $\pi$ is tidy if and only if $T$ is Kuratowski-finite and decidable 
\cite[Chapter III, Examples 1.4(1)]{moerdijk-vermeulen} \cite{nlab-proper}.
Further, $\pi$ is local if and only if $T$ is internally tiny.
\end{remark}

\section{Eilenberg--Mac Lane objects and loop spaces}
\label{app:em-objects-and-loop-spaces}

Let $\E$ be a topos. For the definition of Eilenberg--Mac Lane objects, see Subsection \ref{ssec:em-objects-cohomology}. We will show that Eilenberg--Mac Lane objects exist for $\E$ and that they are unique up to homotopy equivalence.

We will use the functors
\begin{equation} \label{eq:adjunction-G-W}
\begin{tikzcd}
\s_0\mathbf{Gpd} \ar[r,bend right,"{\overline{W}}"'] & \s\Set \ar[l,bend right,"{G}"']
\end{tikzcd}
\end{equation}
as discussed in \cite[Chapter V, Lemma 7.7]{goerss-jardine}. Here $\s_0\mathbf{Gpd}$ denotes the category of groupoids enriched in simplicial sets (i.e.\ groupoid objects in the category of simplicial sets such that the simplicial set of objects is discrete).

If $X$ is a reduced simplicial set, in the sense that $X_0 = 1$, then $G(X)$ is a simplicial group \cite[p.~293]{jardine-book}. In fact, the adjunction $G \dashv \overline{W}$ restricts to an adjunction between reduced simplicial sets and simplicial groups \cite[Chapter V, Proposition 6.3]{goerss-jardine}. We write this adjunction as
\begin{equation} \label{eq:adjunction-G-W-restricted}
\begin{tikzcd}
\s\mathbf{Grp} \ar[r, bend right, "{\overline{W}}"'] & \s_+\Set \ar[l, bend right, "{G}"'],
\end{tikzcd}
\end{equation}
i.e.\ using the notation $\s\mathbf{Grp}$ for the category of simplicial groups, and the notation $\s_+\Set$ for the category of reduced simplicial sets. Both $\s\mathbf{Grp}$ \cite[Chapter II, \S 4]{goerss-jardine} and $\s_+\Set$ \cite[Chapter V, Proposition 6.2]{goerss-jardine} can be equipped with a model structure such that a map is a weak equivalence if and only if the underlying map of simplicial sets is a weak equivalence. In fact, \eqref{eq:adjunction-G-W-restricted} is a Quillen equivalence for these model structures \cite[Chapter V, Proposition 6.3]{goerss-jardine}. The fibrations in $\s\mathbf{Grp}$ are precisely the morphisms such that the underlying morphism of simplicial sets is a fibration, for the classical (Quillen) model structure. By Moore's Theorem \cite[Chapter I, Lemma 3.4]{goerss-jardine}, every object in $\s\mathbf{Grp}$ is fibrant. Similarly, a morphism in $\s_+\Set$ is a cofibration if the underlying morphism in $\s\Set$ is a cofibration, in particular all objects in $\s_+\Set$ are cofibrant.

Because every object in $\s_+\Set$ is cofibrant and every object in $\s\mathbf{Grp}$ is fibrant, we do not need to take cofibrant or fibrant replacements in calculating derived functors. In particular, the derived adjunction unit and counit can be written as $X \to \overline{W}(G(X))$ and $G(\overline{W}(H)) \to H$, respectively. Because $G \dashv \overline{W}$ is a Quillen equivalence, these morphisms are weak equivalences.

If $X$ is a reduced simplicial set, then $G(X)$ models the loop space of $X$ \cite[Chapter V, Corollary 5.11]{goerss-jardine}. So in particular, we find
\begin{equation} \label{eq:homotopy-groups-of-loop-spaces}
\pi_n(G(X)) ~\simeq~ \pi_{n+1}(X)
\end{equation}
(the base points are omitted; on the left the choice of base point does not matter because of the group structure, on the right there is only one possibility for the base point). Using the Quillen equivalence \eqref{eq:adjunction-G-W-restricted}, we can then also calculate
\begin{equation} \label{eq:homotopy-groups-of-deloop-spaces}
\begin{split}
\pi_{n+1}(\overline{W}(H)) 
~&\simeq~ \pi_n(G(\overline{W}(H))) \\
~&\simeq~ \pi_n(H).
\end{split}
\end{equation}

The functors $G$ and $\overline{W}$ have sheaf-theoretic counterparts, see e.g. \cite{jardine-book}. We briefly discuss these counterparts in our setting, following a similar approach to the one in \cite[\S9.4]{jardine-book}. We define $G$ and $\overline{W}$ for sheaves by first applying these functors pointwise, and then sheafifying.

Just like in the original setting, we can restrict to sheaves of groups on one side and sheaves of reduced simplicial sets on the other side. In this way we get an adjunction
\begin{equation} \label{eq:adjunction-topos-G-W}
\begin{tikzcd}
\mathbf{Grp}(\s\E)  \ar[r,bend right,"{\overline{W}}"'] & \s_+\E \ar[l,bend right,"{G}"']
\end{tikzcd}
\end{equation}
with $\mathbf{Grp}(\s\E)$ denoting the group objects in $\s\E$ (or equivalently, simplicial group objects in $\E$) and $\s_+\E$ denoting the reduced simplicial objects in $\E$.

Homotopy groups of sheaves are defined by calculating them in a pointwise way and then sheafifying. By writing out the definitions in detail, we see that both \eqref{eq:homotopy-groups-of-loop-spaces} and \eqref{eq:homotopy-groups-of-deloop-spaces} still hold for sheaves. Also note that sheaves of reduced simplicial sets necessarily have trivial $\pi_0$. So, for a sheaf of simplicial groups $H$, the homotopy groups of $\overline{W}(H)$ are completely determined by the homotopy groups of $H$.

As a special case, any group object $H$ in $\E$ can be interpreted as a (discrete) object in $\mathbf{Grp}(\s\E)$. In this case, we get $\pi_0(H) = H$ and $\pi_{k}(H) = 1$ for $k\geq 1$. As a result, we find
\begin{equation*}
\pi_k(\overline{W}(H)) ~=~ \begin{cases}
H &{k=1} \\
1 &{k \neq 1}
\end{cases}
\end{equation*}
so any fibrant replacement of $\overline{W}(H)$ is an Eilenberg--Mac Lane object $K(H,1)$. Similarly, an abelian group object $H$ in $\E$ can be thought of as a (discrete) object in $\mathbf{Ab}(\mathbf{Grp}(\s\E))$. In this case, $\overline{W}(H)$ is a abelian group object in $\s_+\E \subseteq \s\E$, which means we can iterate the $\overline{W}$ construction. As homotopy groups, we then find
\begin{equation*}
\pi_k(\overline{W}\vphantom{W}^n(H)) ~=~ \begin{cases}
H &k=n \\
1 &k\neq n,
\end{cases}
\end{equation*}
with $\overline{W}\vphantom{W}^n$ denoting $\overline{W}$ iterated $n$ times. So any fibrant replacement of $\overline{W}\vphantom{W}^n(H)$ is a $K(H,n)$.

Above we showed that Eilenberg--Mac Lane objects in $\E$ exist. We now claim that they are unique up to homotopy equivalence. In other words, we want to show that any object $X$ which is a $K(H,n)$ is homotopy equivalent to $\overline{W}^n(H)$.

We proceed by induction. If $n=0$, then because $\pi_0(X) = H$ there is a map $X \to H$, where $H$ is equipped with the discrete simplicial object structure, and this map is an homotopy equivalence. To construct the map, note that the functor $\pi_0: \s\E \to \E$ has as right adjoint the inclusion $\E \to \s\E$ of discrete simplicial objects, and then $X \to H$ is a unit for this adjunction.

For $n\geq 1$, we have that $\pi_0(X)$ is trivial. Using a sheafified version of the first Eilenberg subcomplex \cite{nlab:eilenberg_subcomplex}, we can find an object $X_+$ that is reduced as simplicial object, homotopy equivalent to $X$. We now find weak equivalences
\begin{equation*}
X 
\sim X_+
\sim \overline{W}(G(X_+))
\sim \overline{W}^n(H)
\end{equation*}
where at the end we use the induction step. Because both $X$ and $\overline{W}^n(H)$ are fibrant, we get that $X$ is homotopy equivalent to $\overline{W}^n(H)$.

Finally, we mention an alternative construction of the loop space. Let $X$ be a sheaf of reduced simplicial sets, and consider the diagram
\begin{equation*}
\begin{tikzcd}
& 1 \ar[d] \\
1 \ar[r] & X
\end{tikzcd}
\end{equation*}
in $\s\E$ (because $X$ is reduced, there is a unique map $1 \to X$). We define the loop space $\Omega(X)$ as the homotopy limit of the above diagram.

More generally, if $X$ is a simplicial sheaf with $\pi_0(X) \simeq 1$, then using again the sheafified version of the first Eilenberg subcomplex \cite{nlab:eilenberg_subcomplex}, we see that $X$ is homotopy equivalent to a reduced simplicial sheaf, so we can define $\Omega(X)$ in this case as well.

Let us say that $X$ is a simplicial sheaf over a category $\C$. For each object $U$ in $\C$, the evaluation functor $F \mapsto F(U)$ from simplicial sheaves to simplicial sets preserves finite homotopy limits, so we see that $(\Omega X)(U)$ is the homotopy limit of the diagram
\begin{equation*}
\begin{tikzcd}
& 1 \ar[d] \\
1 \ar[r] & X(U),
\end{tikzcd}
\end{equation*}
which is a standard definition of the loop space of $X(U)$. As a result, we get the equality
\begin{equation*}
\pi_n(\Omega(X)) ~\simeq~ \pi_{n+1}(X).
\end{equation*}

We write $\Omega^k$ for $\Omega$ iterated $k$ times. Note that in order for $\Omega^kX$ to exist, we need that $\Omega^{k-1}X$ exists and that $\pi_0(\Omega^{k-1}X) \simeq 1$. Inductively, we see that $\Omega^k(X)$ exists whenever $\pi_i(X) \simeq 1$ for all $0 \leq i \leq k-1$.

\begin{proposition} \label{prop:looping-commutes-with-pushforward}
Let $\rho: \F \to \E$ be a geometric morphism and take $X$ in $\s\E$ such that $\Omega X$ exists. Then $(\s \rho)_*(\Omega X)$ and $\Omega( (\s \rho)_* X)$ are homotopy equivalent.
\end{proposition}
\begin{proof}
This follows because $\Omega$ is constructed as a derived limit, and $(\s e)_*$ preserves derived limits (up to homotopy equivalence).
\end{proof}

Using induction, we then also see that $(\s \rho)_*(\Omega^k X)$ and $\Omega^k(\s\rho_*X)$ are homotopy equivalent, whenever $\Omega^kX$ exists.

\begin{proposition} \label{prop:homotopy-groups-of-loop-space}
Let $\E$ be a Grothendieck topos and let $X$ be an object in $\s\E$ such that $\Omega^kX$ exists. Then for $n \geq k$, we get $\pi_n(X) = \pi_{n-k}(\Omega^kX)$.
\end{proposition}

\begin{corollary} \label{cor:loop-space-of-em}
Let $\E$ be a Grothendieck topos and take an Eilenberg--Mac Lane object $K(A,n)$. Then for each $k \leq n$, the object $\Omega^kK(A,n)$ exists and is a $K(A,n-k)$.
\end{corollary}

\section{Truncation and \texorpdfstring{$n$-types}{n-types}}
\label{app:truncation}

Simplicial sets are presheaves on the simplex category $\Delta$, which has objects indexed by the natural numbers $0$, $1$, $2$, \ldots. We will establish some further notations following \cite{nlab-simplicial-skeleton}. Let $\Delta_{\leq n}$ the full subcategory of $\Delta$ on the objects $0$, $1$, \dots, $n$. The category of presheaves on $\Delta_{\leq n}$ will be denoted by $\s_n\Set$. It is an essential subtopos of $\s\Set$, more precisely there are functors
\begin{equation*}
\begin{tikzcd}[row sep=large, column sep=large]
\s_n\Set \ar[r, bend left=30, "{\mathrm{sk}_n}"] \ar[r, bend right=30, "{\mathrm{cosk}_n}"',pos=.48] & \ar[l, "{\mathrm{tr}_n}"] \s\Set
\end{tikzcd}
\end{equation*}
with $\mathrm{sk}_n \dashv \mathrm{tr}_n \dashv \mathrm{cosk}_n$, and with $\mathrm{sk}_n$ and $\mathrm{cosk}_n$ fully faithful. Here $\mathrm{tr}_n$ takes a presheaf on $\Delta$ and restricts it to a presheaf on $\Delta_{\leq n}$, the other two functors are uniquely determined by the adjunctions. For a simplicial set $X$, we define its $n$-skeleton as $\mathbf{sk}_n(X) = \mathrm{sk}_n(\mathrm{tr}_n(X))$ and its $n$-coskeleton as $\mathbf{cosk}_n(X) = \mathrm{cosk}_n(\mathrm{tr}_n(X))$.

Following Jardine \cite[\S5.6]{jardine-book}, the $n$-th \textbf{Postnikov section} $P_n(X)$ of simplicial set $X$ is defined as the image of the unit map $X \to \mathbf{cosk}_n(X)$, i.e.\ the following is an epi--mono factorization:
\begin{equation*}
X \to P_n(X) \to \mathbf{cosk}_n(X).
\end{equation*}
Viewing $\s_n\Set$ as a subtopos of $\s\Set$, we know that there is a Grothendieck topology $J_n$ on $\Delta$ such that the objects of $\s_n\Set$ are precisely the $J_n$-sheaves, and then the unit map $X \to \mathbf{cosk}_n(X)$ is precisely the universal map to the $J_n$-sheafification. With this interpretation in mind, we have that $P_n(X)$ is the universal $J_n$-separated presheaf associated to $X$.

It can be checked that if $X$ is fibrant, then $P_n(X)$ has the same homotopy groups as $X$ in degrees $\leq n$, and that it has trivial homotopy groups in degrees $> n$, see \cite[\S5.6]{jardine-book}. As explained there, in general we have to derive the functor $P_n$, so we define
\begin{equation*}
\mathbf{P}_n(X) ~=~ P_n(X^f)
\end{equation*}
for $X^f$ a fibrant replacement of $X$ (Jardine uses the $\mathrm{Ex}^\infty$-construction). We can compose the map to the fibrant replacement $X \to X^f$ with the natural map $X^f \to P_n(X^f)$ to get a map $X \to \mathbf{P}_n(X)$.

Jardine in \cite[\S5.6]{jardine-book} extends the above definition of Postnikov sections to simplicial presheaves $X$ via the formula $(P_nX)(U) = P_n(X(U))$. We can then define it for simplicial sheaves by first using the construction for presheaves and then sheafifying.
For $X$ a fibrant object of $\s\E$, we still find that $P_n(X)$ has the same homotopy groups as $X$ in degrees $\leq n$, and trivial homotopy groups in degrees $>n$. If $X$ is a general object of $\s\E$, we can again look at the derived version $\mathbf{P}_n(X) = P_n(X^f)$ where $X^f$ is a fibrant resolution for $X$.

We can also describe Postnikov sections in $\s\E$ as follows. The pullback of the subtopos $\s_n\Set$ of $\s\Set$ along $\s e: \s\E \to \s\Set$ gives a subtopos
\begin{equation*}
\begin{tikzcd}[row sep=large, column sep=large]
\s_n\E \ar[r,bend right,"{\cosk_n}"'] & \s\E \ar[l,bend right,"{\mathrm{tr}_n}"']
\end{tikzcd}
\end{equation*}
and then we can alternatively define $P_n(X)$ via the epi--mono factorization
\begin{equation*}
X \to P_n(X) \to \mathbf{cosk}_n(X)
\end{equation*}
as before, with $\mathbf{cosk}_n = \cosk_n \circ \mathrm{tr}_n$. We will reuse the name $J_n$ for the (Lawvere--Tierney) topology on $\s\E$ corresponding to this subtopos. Note that $\mathrm{tr}_n(X)$ is completely determined by $\mathrm{tr}_n(X)_i = X_i$. Using this description, we can confirm that $\mathrm{tr}_n(X)$ is computed pointwise, i.e.\ 
\begin{equation*}
\mathrm{tr}_n(X)(U) = \mathrm{tr}_n(X(U)).
\end{equation*}
Further, we can show that $\cosk_n$ is computed pointwise as well, using Proposition \ref{prop:direct-image-pointwise}. It then follows that this alternative construction for $P_n(X)$ agrees with the one introduced earlier.

The properties that we discussed for simplicial sets, generalize to simplicial sheaves.

\begin{lemma} \label{lmm:truncation}
For an object $X$ of $\s\E$, the following are equivalent:
\begin{enumerate}
\item $X$ has trivial homotopy groups in degrees $>n$;
\item the map
\begin{equation*}
X \to \mathbf{P}_n(X)
\end{equation*}
is a weak equivalence;
\item $X$ is weakly equivalent to a $J_n$-separated object.
\end{enumerate}
\end{lemma}

\begin{definition}
We say that an object $X$ of $\s\E$ is \textbf{$n$-truncated} if its homotopy groups are trivial in degrees $>n$. If in addition $X$ is fibrant, then we say that $X$ is an \textbf{$n$-type}.
\end{definition}

\begin{example}[Eilenberg--Mac Lane objects] \label{eg:eilenberg-mac-lane-n-type}
An Eilenberg--Mac Lane object $K(A,i)$ is an $n$-type if and only if $i \leq n$.
\end{example}

\begin{lemma} \label{lmm:limits-of-n-types}
A limit of $n$-types is $n$-truncated.
\end{lemma}
\begin{proof}
If $\{X_i\}_{i \in I}$ is a diagram of $n$-types, then $\{\mathbf{P}_n(X_i)\}_{i \in I}$ is a homotopy equivalent diagram, and each $\mathbf{P}_n(X_i)$ is $J_n$-separated. Because homotopy equivalent diagrams have homotopy equivalent limits, we can assume that the $X_i$ were $J_n$-separated to begin with. A limit of $J_n$-separated objects is again $J_n$-separated, in particular $n$-truncated. 
\end{proof}

\begin{proposition} \label{prop:hom-to-n-type}
Let $X$ be an $n$-type in $\s\E$. Then $\HOM(Y,X)$ is an $n$-type as well, for any object $Y$ of $\s\E$.
\end{proposition}
\begin{proof}
The topos $\s\E$ is generated (under colimits) by the objects of the form
\begin{equation*}
\gamma^*U \times (\s e)^*\Delta^n
\end{equation*}
for $U$ in a set of generators for $\E$ and $n$ a natural number. So we can write $\HOM(Y,X)$ as a limit of simplicial sets of the form $\HOM(\gamma^*U\times (\s e)^* \Delta^n, X)$. We claim that the latter kind of object is an $n$-type. It then follows from Lemma \ref{lmm:limits-of-n-types} that $\HOM(Y,X)$ is $n$-truncated. Because the functor $\HOM(Y,-) = \gamma_*((-)^Y)$ preserves fibrations, we also know that $\HOM(Y,X)$ is fibrant, so it is an $n$-type.

It remains to show that $\HOM(\gamma^*U\times (\s e)^* \Delta^n, X)$ is an $n$-type. We calculate
\begin{equation*}
\begin{split}
\HOM(\gamma^*U\times \s e^* \Delta^n, X) &\simeq \s e_*\left(X^{(\gamma^*U\times \s e^* \Delta^n)}\right) \\
&\simeq \s e_*\left((X^{\gamma^*U})^{\s e^*\Delta^n}\right) \\
&\simeq \s e_*(X^{\gamma^*U})^{\Delta^n} \\
&\simeq \HOM(\Delta^n, X(U)),
\end{split}
\end{equation*}
and because $X(U)$ is an $n$-type, $\HOM(\Delta^n, X(U))$ is an $n$-type as well, see \cite[\S 7.1]{rezk}.
\end{proof}

\begin{remark} \label{rem:rezk-truncation}
In \cite[\S7.1]{rezk}, an object $X$ in a simplicial model category $\mathbf{M}$ is called $n$-truncated if for each $Y$ in $\mathbf{M}$ we have that the derived mapping space $R\HOM(Y,X)$ is $n$-truncated (see Definition \ref{def:derived-mapping-space}). By Proposition \ref{prop:hom-to-n-type}, our definition agrees with Rezk's for $\mathbf{M} = \s\E$.
\end{remark}

\end{document}